\documentclass[a4paper,12pt]{article}
\usepackage[a4paper,hmargin=2.5cm,vmargin=2.5cm]{geometry}
\usepackage[english]{babel}

\usepackage{
  amsmath,
  amsfonts,
  amssymb,
  amsthm,
  amscd,
  braket,
  bm,
  mathtools,
  centernot,
  enumitem,
  ragged2e,
  tikz-cd
}
\usepackage[all]{xy}
\usepackage{verbatim}
\usepackage{lmodern}
\usepackage{booktabs}
\usepackage{tabularx}
\usepackage{makecell}
\usepackage{caption}
\usepackage{placeins}
\usepackage{aliascnt}

\usepackage{hyperref}
\usepackage[capitalise]{cleveref}

\crefname{theorem}{Theorem}{Theorems}
\crefname{lemma}{Lemma}{Lemmas}
\crefname{proposition}{Proposition}{Propositions}
\crefname{corollary}{Corollary}{Corollaries}
\crefname{definition}{Definition}{Definitions}
\crefname{example}{Example}{Examples}
\crefname{exercise}{Exercise}{Exercises}
\crefname{conjecture}{Conjecture}{Conjectures}
\crefname{remark}{Remark}{Remarks}
\crefname{equation}{Equation}{Equations}
\crefname{assumption}{Assumption}{Assumptions}

\hypersetup{
  colorlinks=true,
  linkcolor=black,
  filecolor=magenta,
  urlcolor=cyan,
  citecolor=black,
  pdftitle={Moduli spaces of semistable coherent sheaves on bielliptic surfaces},
  pdfpagemode=FullScreen
}

\usepackage{fancyhdr}
\DeclareMathOperator{\rank}{rank}

\DeclareMathOperator{\Ext}{Ext}
\DeclareMathOperator{\ext}{ext}
\DeclareMathOperator{\Hom}{Hom}
\DeclareMathOperator{\Hilb}{Hilb}

\DeclareMathOperator{\supp}{Supp}
\DeclareMathOperator{\fsupp}{supp}
\DeclareMathOperator{\pic}{Pic}

\DeclareMathOperator{\num}{Num}
\DeclareMathOperator{\ch}{ch}
\DeclareMathOperator{\td}{td}

\DeclareMathOperator{\ord}{ord}
\DeclareMathOperator{\im}{Im}
\DeclareMathOperator{\nt}{nt}
\renewcommand{\det}{\operatorname{det}}

\newcommand{\mv}{\mathbf{v}}
\newcommand{\mw}{\mathbf{w}}

\newcommand\la{\lambda}
\newcommand{\restrict}[2]{\left.#1\right|_{#2}}
\newcommand\nn{\mathbb{N}}
\newcommand\cc{\mathbb{C}}
\newcommand\zz{\mathbb{Z}}
\newcommand\qq{\mathbb{Q}}

\newcommand\ps{\mathbb{P}}

\newcommand\uc{\mathcal{C}}

\newcommand{\pcoor}[1]{%
  \begingroup\lccode`~=`: \lowercase{\endgroup
  \edef~}{\mathbin{\mathchar\the\mathcode`:}\nobreak}%
  [%
  \begingroup
  \mathcode`:=\string"8000
  #1%
  \endgroup
  ]%
}

\newcommand{\Pic}{\pic}
\newcommand{\Nm}{\operatorname{Nm}}
\newcommand{\Alb}{\operatorname{Alb}}

\makeatletter
\renewcommand{\citation}[1]{%
  \@for\citation@key:=#1\do{%
    \global\@namedef{citation@\citation@key}{}%
  }%
}

\newif\ifdiscarding@bibitem
\newcommand{\stop@discarding@bibitem}{%
  \ifdiscarding@bibitem
    \egroup
    \endgroup
    \global\discarding@bibitemfalse
  \fi
}

\let\original@thebibliography\thebibliography
\let\original@endthebibliography\endthebibliography
\renewenvironment{thebibliography}[1]{%
  \original@thebibliography{#1}%
  \let\original@bibitem\bibitem
  \let\bibitem\filtered@bibitem
}{%
  \stop@discarding@bibitem
  \original@endthebibliography
}

\def\filtered@bibitem{%
  \stop@discarding@bibitem
  \@ifnextchar[{\filtered@lbibitem}{\filtered@bibitem@noopt}%
}

\def\filtered@lbibitem[#1]#2{%
  \@ifundefined{citation@#2}{%
    \global\discarding@bibitemtrue
    \begingroup
    \setbox\z@\vbox\bgroup
  }{%
    \original@bibitem[#1]{#2}%
  }%
}

\def\filtered@bibitem@noopt#1{%
  \@ifundefined{citation@#1}{%
    \global\discarding@bibitemtrue
    \begingroup
    \setbox\z@\vbox\bgroup
  }{%
    \original@bibitem{#1}%
  }%
}
\makeatother

\theoremstyle{plain}
\newtheorem{theorem}{Theorem}[section]
\newtheorem*{theorem*}{Theorem}
\newtheorem{maintheorem}{Theorem}

\crefname{maintheorem}{Theorem}{Theorems}
\Crefname{maintheorem}{Theorem}{Theorems}
\newaliascnt{proposition}{theorem}
\newtheorem{proposition}[proposition]{Proposition}
\aliascntresetthe{proposition}
\newaliascnt{corollary}{theorem}
\newtheorem{corollary}[corollary]{Corollary}
\aliascntresetthe{corollary}
\newaliascnt{lemma}{theorem}
\newtheorem{lemma}[lemma]{Lemma}
\aliascntresetthe{lemma}
\newaliascnt{conjecture}{theorem}

\aliascntresetthe{conjecture}

\theoremstyle{definition}
\newaliascnt{definition}{theorem}
\newtheorem{definition}[definition]{Definition}
\aliascntresetthe{definition}
\newtheorem*{definition*}{Definition}
\newaliascnt{example}{theorem}

\aliascntresetthe{example}
\newaliascnt{exercise}{theorem}

\aliascntresetthe{exercise}

\theoremstyle{remark}
\newaliascnt{remark}{theorem}
\newtheorem{remark}[remark]{Remark}
\aliascntresetthe{remark}
\newtheorem*{remark*}{Remark}
\newaliascnt{assumption}{theorem}

\aliascntresetthe{assumption}
\newtheorem*{assumption*}{Assumption}

\begin{document}

\fontfamily{lmr}\selectfont

\title{Topology and geometry of moduli spaces of semistable sheaves on bielliptic surfaces}
\author{Aleksei Piskunov}
\date{}
\maketitle

\begin{abstract}
We study the topology and low-degree Hodge theory of moduli spaces of
semistable sheaves on bielliptic surfaces. For primitive rank-zero
Mukai vectors $\mv=(0,c_1(L),\chi)$ satisfying the positivity condition
$\nt(L)\geq3$, the distinguished fixed-determinant
component $M_{H,S}(\mv,L)$ admits a support morphism to $|L|$ and is
interpreted as a relative compactified Jacobian. Using this fibration,
Lefschetz-type properties of positive linear systems, monodromy, and
mixed Hodge structures, we construct a surjective homomorphism
$
\pi_1(S)\twoheadrightarrow\pi_1\bigl(M_{H,S}(\mv,L)\bigr)
$
and compute the first two Betti numbers of an Albanese fiber $F$,
$M_{H,S}(\mv,L)$, and the distinguished component
$M^\circ_{H,S}(\mv)$. Fourier--Mukai transforms and
Bridgeland wall crossing extend these computations to primitive
admissible Mukai vectors of positive rank. We further prove that
$H^2(F,\cc)$ is of pure Hodge type $(1,1)$. If
$\la_S=\ell(\mv)=1$, then $F$ is a strict irreducible Calabi--Yau variety up to a finite quasi-\'etale cover. Under the additional genericity assumption on the pullback
polarization, $M^\circ_{H,S}(\mv)$ is smooth and is noncanonically
birational to
$\operatorname{Pic}^0(S)\times\operatorname{Hilb}^{\mv^2/2}(S)$.
\end{abstract}

\newpage
\tableofcontents

\newpage

\section*{Introduction}
\addcontentsline{toc}{section}{Introduction}

Moduli spaces of sheaves on smooth projective surfaces form a central class of
varieties in algebraic geometry. It turns out these moduli spaces often retain strong traces of the geometry
of the surface. For K3 surfaces, smooth projective moduli spaces of stable
sheaves provide basic examples of irreducible holomorphic symplectic manifolds
\cite{Muk84, O'G97, Huy97}. For abelian surfaces, the moduli spaces themselves
have positive irregularity, while the fibers of their Albanese morphisms give
the corresponding irreducible holomorphic symplectic varieties \cite{Yos01}.

The remaining minimal surfaces of Kodaira dimension zero, namely Enriques and
bielliptic surfaces, behave differently because their canonical bundles are
nontrivial torsion line bundles. In the Enriques case, Sacc\`a studied relative
compactified Jacobians of linear systems using the universal K3 cover
\cite{Sac19}. Bielliptic surfaces provide a natural counterpart to this
setting. They have irregularity one, their canonical cover is an abelian
surface, and they carry two elliptic fibrations. Their sheaf moduli spaces
therefore share features with the abelian-surface case, but also exhibit
additional arithmetic and fundamental-group phenomena arising from the finite
canonical cover.

On a bielliptic surface $S$, recent work \cite{Nue25} shows that for any
Mukai vector $\mv$ with $\mv^2>0$ and any $\mv$-generic polarization $H$
(i.e. one lying in the complement of a locally finite set of walls in the nef
cone of $S$), the moduli space of Gieseker $H$-semistable coherent sheaves
$M_{H,S}(\mv)$ is a nonempty projective variety. H. Nuer also shows that if
$\mv^2\geqslant 3\ord(K_S)$, then it is normal and Gorenstein, with at worst
terminal l.c.i. singularities and torsion canonical bundle.

This leads naturally to topological and geometrical questions:
irreducibility, fundamental groups, low-degree Hodge theory, and the
geometry of the Albanese fibers.
This paper studies these properties for sheaves of pure dimension one and then
reduces the study of positive-rank sheaves to the rank-zero case. For pure
one-dimensional sheaves, we use techniques developed by G. Sacc\`a in
\cite{Sac19} for sheaves on Enriques surfaces, though these techniques require
substantial adjustments, mostly because the fundamental group of an Enriques
surface is finite abelian, while that of a bielliptic surface is infinite
nonabelian. We also use ideas and techniques developed by K. Yoshioka in
\cite{Yos01} to study fibers of the Albanese morphism of these moduli spaces.

The rank-zero case is the geometric core of the present paper. For $\mv=(0,c_1(L),\chi),$
the distinguished fixed-determinant component $M_{H,S}(\mv,L)$
admits a support morphism
$$
h:M_{H,S}(\mv,L)\longrightarrow |L|
$$
and can therefore be viewed as a relative compactified Jacobian. We use
this fibration to study its fundamental group and low-degree cohomology,
as well as those of the distinguished component
$M^\circ_{H,S}(\mv)$ and their common Albanese fibers. Fourier--Mukai
transforms and Bridgeland wall crossing then extend the Betti-number
computations from rank zero to primitive admissible Mukai vectors of positive
rank.

\subsection*{Main results}
For rank-zero Mukai vectors, $M^\circ_{H,S}(\mv)$ denotes the
distinguished irreducible component containing sheaves that are locally free
on their integral support, and $M_{H,S}(\mv,L)$ denotes the corresponding
component with fixed determinant
$L\in\operatorname{Pic}^{c_1(\mv)}(S)$; see \cref{subsec:moduli-spaces}.
Admissible Mukai vectors and the distinguished component
$M^\circ_{H,S}(\mv)$ in positive rank are defined in \cref{subsec:admissible}.
In positive rank, by abuse of notation, $M_{H,S}(\mv,L)$ denotes any
irreducible component of
$
(\det^\circ)^{-1}(L);
$
see \cref{sec:ext-to-pos-rank}. Let $F$ be an Albanese fiber of
$M_{H,S}(\mv,L)$. The invariants $\la_S$ and $\ell(\mv)$ and the canonical
cover $\pi:X\to S$ are introduced in
\cref{sec:num-eq-of-div,sec:mukai-vector,sec:canonica-cover}, respectively.
Throughout, we use singular cohomology and write
$
b_i(X):=\dim_{\mathbb Q}H^i(X,\mathbb Q).
$
\begin{maintheorem}\label{thm:intro-rank-zero}
Let $(S,H)$ be a polarized bielliptic surface and let $\mv$ be a primitive
admissible Mukai vector. Assume that $H$ is $\mv$-generic.
Then the first two Betti numbers of $M^\circ_{H,S}(\mv)$, $M_{H,S}(\mv,L)$ and $F$ are presented in the following table:
$$
\begin{array}{c|cc}
 & b_1 & b_2\\
\hline
F & 0 & 3\\
M_{H,S}(\mv,L) & 2 & 4\\
M^\circ_{H,S}(\mv) & 4 & 9
\end{array}
$$
If $\mv=(0,c_1(L),\chi)$ has rank zero, then there is a surjective homomorphism
$$
\pi_1(S)\twoheadrightarrow\pi_1(M_{H,S}(\mv,L)).
$$
Moreover, if $\ell(\mv)=\la_S=1$ and $\pi^*H$ is $\pi^*\mv$-generic, then
$M^\circ_{H,S}(\mv)$ is a smooth irreducible projective variety of dimension
$\mv^2+1$ and is noncanonically birational to
$$
\Pic^0(S)\times\Hilb^{\mv^2/2}(S).
$$
\end{maintheorem}

\begin{maintheorem}\label{thm:intro-albanese-fibers}
    Let $(S,H)$ be a polarized bielliptic surface, let $\mv$ be a primitive
admissible Mukai vector, and assume that $H$ is $\mv$-generic. Then $H^2(F,\cc)$ is a
pure Hodge structure of weight two and type $(1,1)$. Equivalently,
$
H^2(F,\cc)=H^{1,1}(F).
$
In particular,
$$
H^0\bigl(F,\Omega_F^{[2]}\bigr)=0.
$$
If, moreover,
$
\la_S=\ell(\mv)=1,
$
then there exists a finite quasi-\'etale cover
$$
\gamma:\widetilde F\longrightarrow F
$$
such that $\widetilde F$ is a Calabi--Yau variety in the sense of
\cite[Definition~1.3]{GGK19}. In particular, the reflexive tangent
sheaf $\mathcal T_F$ is strongly stable.
\end{maintheorem}

The rank-zero computations underlying the table are
\cref{thm:first-betti-fixed-det,thm:b2-albanese-fiber,thm:b1-b2-moduli},
while their extension to arbitrary primitive admissible Mukai vectors is
\cref{thm:betti-numbers-admissible}. The fundamental-group statement is
\cref{thm:pi1}. Smoothness is established in \cref{thm:smooth_dim}, and the
Hilbert-scheme birational model is obtained in
\cref{thm:rank-zero-hilbert-model}. The two assertions of \cref{thm:intro-albanese-fibers} are proved in
\cref{prop:albanese-fiber-hodge-type} and \cref{thm:albanese-fiber-cy-cover}, respectively.

The rank-zero proofs combine the support morphism with Lefschetz-type
properties of sufficiently positive linear systems on bielliptic surfaces,
mixed Hodge structures, and an adaptation of the fundamental-group method of
Leibman and Sacc\`a. The Albanese morphisms are locally constant analytic
fibrations, which relates the cohomology of the determinant-fixed and
determinant-unfixed spaces to that of their common fibers. For general
admissible vectors, Fourier--Mukai reduction and Bridgeland wall crossing
produce birational rank-zero models. Minimal-model theory and invariance of
low-degree cohomology under the resulting isomorphisms in codimension one then
transfer the rank-zero calculations.
\cref{thm:intro-albanese-fibers} follows from a refinement of the mixed-Hodge calculation
for the Albanese fibers and, when $\la_S=\ell(\mv)=1$, from their
comparison in codimension one with Albanese fibers of Hilbert schemes,
together with the Calabi--Yau construction of Oguiso and
Schr\"oer~\cite{OS11}.

\subsection*{Open questions}

The results of this paper suggest several directions for further
investigation.

\begin{enumerate}
\item \textbf{Moduli spaces outside the admissible range.}
Can the topological and geometric conclusions of Theorems~A and~B be
extended to rank-zero vectors with
$
\nt(c_1(\mv))<3
$
and to positive-rank vectors which cannot be transformed to an
admissible rank-zero vector by Fourier--Mukai autoequivalences? In many
of the latter cases, the standard reduction procedures still produce
vectors of rank two or three. It would be interesting to determine
whether these lower-rank moduli spaces can be studied directly.

\item \textbf{Irreducibility.}
Under what conditions are the full moduli space $M_{H,S}(\mv)$ and its
determinant fibers irreducible? In particular, when does
$$
M^\circ_{H,S}(\mv)=M_{H,S}(\mv),
$$
and when is $M_{H,S}(\mv,L)$ the entire determinant fiber rather than
only one of its irreducible components? This question is the subject
of work in preparation.

\item \textbf{Order of the canonical bundle.}
The moduli spaces and Albanese fibers considered above have torsion
canonical bundle. What is its exact order? Can it be expressed in
terms of the type of the bielliptic surface, $\ord(K_S)$, and the numerical invariants of $\mv$? It would also be useful to
understand how this order changes upon fixing the determinant or
passing to an Albanese fiber.

\item \textbf{Fundamental groups.}
In the rank-zero setting, Theorem~A gives a surjective homomorphism
$
\pi_1(S)\twoheadrightarrow\pi_1\bigl(M_{H,S}(\mv,L)\bigr).
$
What is the kernel of this homomorphism, and what is the resulting
fundamental group? Is it independent of the choice of $L$ and of the
$\mv$-generic polarization $H$?

\item \textbf{Beauville--Bogomolov type of the Albanese fibers.}
For a general primitive admissible Mukai vector, what is the singular
Beauville--Bogomolov decomposition type of an Albanese fiber $F$?
Theorem~B shows that, when
$
\la_S=\ell(\mv)=1,
$
the fiber admits a finite quasi-\'etale Calabi--Yau cover and
$\mathcal T_F$ is strongly stable. Does the same conclusion hold for
all admissible vectors? If not, which abelian, Calabi--Yau, or
irreducible holomorphic symplectic factors can occur after passing to
a finite quasi-\'etale cover?
\end{enumerate}

\subsection*{Organization of the text}

\Cref{sec:preliminaries} recalls the required facts about bielliptic surfaces, their divisor
theory, canonical covers, linear systems, and moduli spaces of semistable
sheaves. \Cref{sec:orbits} defines admissible Mukai vectors, studies their
reduction to rank zero, and characterizes the rank-zero vectors whose
Fourier--Mukai orbits contain a rank-one vector.
\Cref{subsec:first-betti-fixed-det,subsec:fundamental-group-fixed-det,subsec:albanese-fibers-rank-zero,subsec:betti-rank-zero}
establish the rank-zero fundamental-group and Betti-number results, while
\cref{sec:ext-to-pos-rank} extends the Betti-number calculations to arbitrary
primitive admissible vectors. \Cref{subsec:hodge-cy-albanese-fibers} proves
\cref{thm:intro-albanese-fibers} by studying the Hodge structure of the
Albanese fibers and constructing quasi-\'etale Calabi--Yau covers in the
rank-one-orbit case.
\Cref{app:cohomological-tools} collects the cohomological tools used in these
arguments, and \cref{app:betti-hilbert} computes the first two Betti numbers of
$\Pic^0(S)\times\Hilb^n(S)$.

\subsection*{Conventions and Notations}
Throughout this paper, the ground field is assumed to be $\cc$ unless otherwise stated explicitly. The Mukai vector $\mv(F)$ of a sheaf $F$ keeps track of the actual class $c_1(F)\in H^2(S,\zz)$. By abuse of notation, we use the same symbol for its class in the algebraic Mukai lattice whenever numerical constructions are involved; see \cref{rem:mukai_abuse}.

\subsection*{Acknowledgments}
The author would like to thank his scientific advisor Howard Nuer for constant attention to his work and Gleb Terentiuk for fruitful discussions.

\section{Preliminaries}\label{sec:preliminaries}
\subsection{Bielliptic surfaces}
\label{sec:bielliptic}

\subsubsection{Classification}
\begin{definition}\label{def:bielliptic}
A minimal smooth projective surface $S$ is called \textit{bielliptic}\footnote{also hyperelliptic in some literature} if
$$
\kappa(S)=0,
\qquad
q(S)=1,
\qquad
p_g(S)=0.
$$
Throughout the paper, $S$ denotes a bielliptic surface unless otherwise stated.
\end{definition}

We will use the following equivalent description. It is known that any bielliptic surface is isomorphic to a quotient $S\cong(A\times B)/G$, where $A$ and $B$ are two elliptic curves and $G$ is a finite abelian group of translations of $A$ (so $A/G$ is still elliptic) which acts on $B$ such that $B/G\cong\ps^1$. By the latter description we have \cite[Section~2.1]{Nue25} two projections $A\times B\to A$ and $A\times B\to B$. The name comes from the fact that such surfaces admit a ``double'' elliptic fibration in a sense that both the base and generic fiber of the fibration are elliptic curves. The projections are $G$-equivariant and so there are two fibrations $p_A:S\to A/G$ and $p_B:S\to B/G\cong\ps^1$.

Throughout the paper, we write $k_S\coloneqq\ord(K_S)$ for the order of the canonical bundle.

Since $A/G$ is an elliptic curve and the projection to $A$ is étale, all
fibres of $p_A$ are smooth. A fibre of $p_B$ over a point $[x]\in B/G$ is a
multiple of a smooth elliptic curve, with multiplicity equal to that of $[x]$
for the quotient map $B\to B/G$. By abuse of notation, we denote the classes
of the smooth fibers of $p_A$ and $p_B$ in $\num(S)$ by $B$ and $A$,
respectively (all smooth fibers are isomorphic for each of these maps).

\begin{remark}\label{rem:albanese_morphism}
$p_A$ is actually the Albanese morphism of $S$; see \cite[Chapter V]{BPVdV84}.
\end{remark}

The seven types of bielliptic surfaces and their topological invariants are summarized in \cref{tab:types} and \cref{tab:invariants}. We let $B \simeq \cc /(\zz \oplus \omega_G \zz)$ where $\omega_G$ is uniquely determined in the fundamental domain and denote by $(m_1,\dots,m_s)$  the multiplicities of the singular fibers of $p_B \colon S \to B/G \simeq \ps^1$.
\begin{proposition}\label{prop:euler}
For a bielliptic surface $S$ both topological and holomorphic Euler characteristic vanish $\chi(\mathcal{O}_S) = \chi_{top} = 0$ and the Hodge diamond looks like:
$$
\begin{matrix}
 & & 1 & & \\
 & 1 & & 1 & \\
0 & & 2 & & 0 \\
 & 1 & & 1 & \\
 & & 1 & &
\end{matrix}
$$
\end{proposition}
\begin{proof}
This is part of the standard Enriques--Kodaira classification of surfaces of Kodaira dimension zero; see \cite[Chapter VI, Table 10]{BHPV}.
\end{proof}

\begin{table}[h!]
\caption{Seven types of bielliptic surfaces}
\label{tab:types}
\centering
\begin{tabularx}{\textwidth}{XXXX}
\toprule
Type & $\omega_G$ & $G$ & Action of $G$ on $B$ \\
\midrule
1 & any & $\zz/2\zz$ & $x \mapsto -x$ \\[4pt]
2 & any & $\zz/2\zz \times \zz/2\zz$ & $x \mapsto -x$ \\
  &     &                                & $x \mapsto x+\varepsilon$ for $2\varepsilon=0$ \\[4pt]
3 & $i$ & $\zz/4\zz$ & $x \mapsto ix$ \\[4pt]
4 & $i$ & $\zz/4\zz \times \zz/2\zz$ & $x \mapsto ix$ \\
  &     &                                & $x \mapsto x+\frac{1+i}{2}$ \\[4pt]
5 & $e^{2\pi i/3}$ & $\zz/3\zz$ & $x \mapsto e^{2\pi i/3}x$ \\[4pt]
6 & $e^{2\pi i/3}$ & $\zz/3\zz \times \zz/3\zz$ & $x \mapsto e^{2\pi i/3}x$ \\
  &                 &                                  & $x \mapsto x+\frac{1-e^{2\pi i/3}}{3}$ \\[4pt]
7 & $e^{2\pi i/3}$ & $\zz/6\zz$ & $x \mapsto -e^{2\pi i/3}x$ \\
\bottomrule
\end{tabularx}
\end{table}

\begin{table}[h!]
\caption{Topological invariants of bielliptic surfaces}
\label{tab:invariants}
\centering
\begin{tabularx}{\textwidth}{XXll}
\toprule
Type & $(m_1,\dots,m_s)$ & $H^2(S,\zz)$ & $k_S$ \\
\midrule
1 & $(2,2,2,2)$ & $\zz[\frac{1}{2}A]\oplus\zz[B]\oplus\zz/2\zz\oplus\zz/2\zz$ & 2 \\[4pt]
2 & $(2,2,2,2)$ & $\zz[\frac{1}{2}A]\oplus\zz[\frac{1}{2}B]\oplus\zz/2\zz$ & 2 \\[4pt]
3 & $(2,4,4)$ & $\zz[\frac{1}{4}A]\oplus\zz[B]\oplus\zz/2\zz$ & 4 \\[4pt]
4 & $(2,4,4)$ & $\zz[\frac{1}{4}A]\oplus\zz[\frac{1}{2}B]$ & 4 \\[4pt]
5 & $(3,3,3)$ & $\zz[\frac{1}{3}A]\oplus\zz[B]\oplus\zz/3\zz$ & 3 \\[4pt]
6 & $(3,3,3)$ & $\zz[\frac{1}{3}A]\oplus\zz[\frac{1}{3}B]$ & 3 \\[4pt]
7 & $(2,3,6)$ & $\zz[\frac{1}{6}A]\oplus\zz[B]$ & 6 \\
\bottomrule
\end{tabularx}
\end{table}
\FloatBarrier

By the classification, the canonical bundle of a bielliptic surface is non-trivial, but torsion. Since in general both $\Pic^0(S)$ and $H^2(S,\zz)_{\mathrm{tor}}$ are non-zero, it is interesting to ask whether $K_S \in \Pic^0(S)$ or not. The answer is given by the following:

\begin{lemma}\label{lem:canonical}
For a bielliptic surface $S$ we always have $K_S \in \Pic^0(S)$.
\end{lemma}
\begin{proof}
Recall that the canonical bundle formula for an elliptic fibration \cite[Corollary 12.3]{BPVdV84} shows that, for the elliptic fibration $p_A:S\to A/G$ with no multiple fibers, we have $K_S=p_A^*(L)$ for some line bundle $L$ on $A/G$ of degree $\chi(\mathcal{O}_S)-2\chi(\mathcal{O}_{A/G})=0$.
In other words, $K_S$ is the pullback of a degree 0 line bundle on the elliptic curve $A/G$, hence $K_S$ itself lies in $\Pic^0(S)$.
\end{proof}

A more detailed treatment of the classification of bielliptic surfaces can be
found in \cite[Section~1]{Ser90} and \cite[Section~2.1]{Nue25}. The
multiplicities of the singular fibers are listed in
\cite[Theorem~2.2]{Nue25}.

\subsubsection{Numerical equivalence of divisors}\label{sec:num-eq-of-div}
In this section we recall the numerical divisor theory of bielliptic surfaces.

\begin{proposition}\label{prop:H2Q_spanned}
\cite[Section 1]{Ser90} $H^2(S,\qq)$ is spanned by $A$ and $B$ such that $A^2 = B^2 = 0$ and $AB = |G|$, where $S \cong (A \times B)/G$ is a bielliptic surface.
\end{proposition}

We will also use the following numerical criterion.

\begin{lemma}\label{lem:num_div}
\cite[Lemma 1.3]{Ser90} Suppose $S \cong (A \times B)/G$ is a bielliptic surface. Then for a divisor $D \equiv aA + bB$ on $S$ the following holds:
\begin{itemize}
    \item $\chi(\mathcal{O}_S(D)) = \frac{1}{2}D^2 = ab|G|$.
    \item $D$ is ample iff $a > 0$ and $b > 0$.
    \item If $D$ is ample then $h^0(\mathcal{O}_S(D)) = \chi(\mathcal{O}_S(D))$.
    \item If $H^0(S, \mathcal{O}_S(D)) \neq 0$ (i.e. $D$ is effective) then $a \geqslant 0$ and $b \geqslant 0$.
\end{itemize}
\end{lemma}

The lattice of divisors modulo numerical equivalence (which we denote by $\equiv$) is also known and was described in \cite[Theorem 2.2]{Nue25}:
$$
\num(S)
=
\zz\left[\frac{1}{k_S}A\right]
\oplus
\zz\left[\frac{k_S}{|G|}B\right].
$$
To simplify notations from now and later on, set $\la_S\coloneqq\frac{|G|}{k_S}$, $A_0\coloneqq\frac{1}{k_S}A$, and $B_0\coloneqq\frac{1}{\la_S}B$.
Then
\begin{equation}\label{eq:num_S}
\num(S) = \zz[A_0] \oplus \zz[B_0]
\end{equation}
and $\la_S$ is the smallest possible degree of a multisection of the elliptic fibration $p_A \colon S \longrightarrow A/G$, which is provided by $A_0$. Also note that the intersection pairing in this basis is given by
$$
A_0 \cdot B_0 = 1,
\qquad
A_0^2 = B_0^2 = 0.
$$

\subsubsection{Canonical cover}\label{sec:canonica-cover}
Let $S=(A\times B)/G$ and $G=(\zz/k_S\zz)\times(\zz/\la_S\zz)$ be a bielliptic surface. We follow notations of \cite[Section 2.2]{Nue25} here and, for simplicity, set $G'=\zz/k_S\zz$, $H=\zz/\la_S\zz$, and $G=G'\times H$.

Since the canonical bundle of $S$ is torsion of order $k_S$, there exists the uniquely defined étale cover $\pi:X\to S$ of order $k_S$. By construction $\mathcal{O}_X(K_X)\cong\mathcal{O}_X$ and $\pi$ is the quotient map by the free action of the group $G'=\zz/k_S\zz$. It turns out that $X$ is an abelian surface which is a covering of $A\times B$ modulo the action of $H=\zz/\la_S\zz$ (so for families 3,5,7 we have $X=A\times B$ because $\la_S=1$).

Also $X$ admits two smooth elliptic fibrations $X\to A/H$ and $X\to B/H$, with fibers isomorphic to $B$ and $A$ respectively, and we denote the corresponding classes in $\num(X)$ by $B_X$ and $A_X$, so $A_X\cdot B_X=\la_S$. Putting everything together we have $A\times B\xrightarrow{\pi'}X\xrightarrow{\pi}S$, where $\pi'$ is the degree $\la_S$ isogeny of the abelian surfaces $A \times B$ and $X$, and $\pi$ is the degree $k_S$ étale (canonical) cover. Note that for families 3,5 and 7 $\pi'$ is the identity map and $X = A \times B$.

\subsubsection{Linear systems of divisors}\label{subsec:linear-systems}
We now recall the numerical threshold used to formulate very-ampleness conditions for divisors on bielliptic surfaces.

\begin{definition}\label{def:num_threshold}
Let $D\equiv aA_0+bB_0$ be a numerical divisor class on a bielliptic surface $S$. We define its numerical threshold by
$$
\nt(D)\coloneqq\min\{a,b\}.
$$
For a line bundle $L$, or for an actual algebraic class $c\in H^2(S,\zz)$, the notation $\nt(L)$ or $\nt(c)$ refers to the numerical class of $c_1(L)$ or $c$, respectively.
\end{definition}

Recall the definition of $k$-very ample line bundle for later use.

\begin{definition}\label{def:k_very_ample}
A line bundle $L$ on a smooth projective variety $X$ is called $k$-very ample if, for any 0-dimensional subscheme $Z\subset X$ of length $k+1$, the restriction map
$$
H^0(X,L)\to H^0(X,L\otimes\mathcal{O}_Z)
$$
is surjective.
\end{definition}

By construction 0-very ampleness means being globally generated and 1-very ampleness is very ampleness in the usual sense. We will be mainly interested in these two cases, but state theorems related to this notion in general when possible.

\begin{proposition}\label{prop:k_very_ample_condition}
\cite[Proposition 2.7]{Far16} Let $L$ be a line bundle of type $(a, b)$ on a bielliptic surface $S$. If $a \geqslant k + 2$ and $b \geqslant k + 2$ where $k \in \nn$ then $L$ is $k$-very ample.
\end{proposition}

If one takes into account differences between the seven families in \cref{tab:types} then on families 2, 4 and 6 there are more precise statements regarding $k$-very ampleness, but \cref{prop:k_very_ample_condition} as stated works in general and is enough for our purposes. For more details one can see \cite[Theorems 3.2-3.4]{MP93}.

\begin{lemma}\label{lem:bertini_results}
Let $D$ be an effective divisor of numerical type $(a, b)$ on a bielliptic surface $S$. Assume that $\nt(D) \geqslant 3$ (in particular $D^2 \geqslant 18$). Then the linear system $|D|$ is base point free, the general member of $|D|$ is a smooth connected curve, the discriminant $\Delta_D$ is an irreducible hypersurface and there is an open dense subset of the discriminant $\Delta_D$ parameterizing irreducible curves with one single node.
\end{lemma}
\begin{proof}
The condition $\nt(D)\geqslant 3$ means that $a\geqslant 3$ and $b\geqslant 3$. By \cref{prop:k_very_ample_condition}, applied with $k=1$, the line bundle $\mathcal O_S(D)$ is $1$-very ample. Hence it is very ample and defines a closed embedding
$$
\varphi_{|D|}\colon S\hookrightarrow \ps^N.
$$
Since $\mathcal O_S(D)$ is very ample, it is in particular base point free. By Bertini's theorem \cite[Chapter III, Corollary 10.9]{Har77}, the general member of $|D|$ is smooth. Moreover, every effective ample divisor on the smooth connected
projective surface $S$ is connected
\cite[Chapter III, Corollary 7.9]{Har77}.
Consequently, every smooth member of $|D|$ is a smooth connected curve.

Let
$$
\Delta_D\subset |D|\simeq(\ps^N)^\vee
$$
be the discriminant parametrizing singular divisors. Under the embedding above, $\Delta_D$ is the projective dual variety $S^\vee$. Since $S$ is smooth and positive-dimensional, the conormal variety
$$
\ps(T_S^*\ps^N)\subset S\times(\ps^N)^\vee
$$
is irreducible, and its image is $S^\vee=\Delta_D$. Hence $\Delta_D$ is irreducible (one can see \cite[Chapter 1, Sections 1.1--1.2]{Tev05} for more details).

Now we prove that $\Delta_D=S^\vee$ is a hypersurface. The image $\varphi_{|D|}(S)\subset\mathbb P^N$ is a smooth nonlinear projective surface. By Landman's parity theorem and its surface consequence \cite[Theorem~6.15 and Example~6.16]{Tev05}, its dual variety is a hypersurface. Therefore
$
\operatorname{codim}_{(\mathbb P^N)^\vee}S^\vee=1.
$
Thus $\Delta_D$ is an irreducible hypersurface.

By the Lefschetz-pencil theorem, a general line
$\Lambda\subset |D|$ meets $\Delta_D$ transversely in its smooth
locus, and every singular member of the corresponding pencil has
exactly one singular point, which is an ordinary quadratic
singularity; see
\cite[Exposé XVII, Sections 2--3]{SGA7II}
or \cite[Chapter 2, Section 2.1]{Voi02}.
It follows that there is a nonempty Zariski-open subset $\Delta_D^\circ\subset \Delta_D$
whose members have exactly one ordinary double point and are otherwise smooth.

It remains to show that these members are irreducible. Let $C\in\Delta_D^\circ$, and suppose that $C$ is reducible. Since $C$ has only an ordinary double point, it is reduced. We may therefore write $C=C_1+C_2$, where $C_1$ and $C_2$ are nonzero effective divisors with no common irreducible components. Since $C$ is an effective ample divisor, it is connected. Hence $C_1\cap C_2\neq\varnothing$. Every point of $C_1\cap C_2$ is singular on $C$, and $C$ has only one singular point, which is an ordinary double point. Thus $C_1$ and $C_2$ meet transversely at precisely that point, and consequently $C_1\cdot C_2=1$.
Write $C_i\equiv a_iA_0+b_iB_0$. By effectivity, \cref{lem:num_div} gives $a_i,b_i\ge 0$, and $a_1+a_2=a, \, b_1+b_2=b$.
Furthermore, neither $(a_i,b_i)$ is equal to $(0,0)$, since a nonzero effective divisor cannot be numerically trivial. Using $A_0^2=B_0^2=0,\, A_0\cdot B_0=1$, we obtain $C_1\cdot C_2=a_1b_2+a_2b_1$.
If $a_1,a_2>0$, then $a_1b_2+a_2b_1\ge b_1+b_2=b$.
If one of $a_1,a_2$ is zero, say $a_1=0$, then $b_1>0$ and $a_1b_2+a_2b_1=ab_1\ge a$.
The analogous argument applies after interchanging the two indices or the two coordinates. Therefore $C_1\cdot C_2\ge \min{a,b}
=\nt(D)\ge 3$,
contradicting $C_1\cdot C_2=1$. Hence every member of $\Delta_D^\circ$ is irreducible.
\end{proof}

\subsubsection{Homotopy groups}
Fix a bielliptic surface $S$ with
$$
G=\zz/k_S\zz\times\zz/\la_S\zz
$$
and canonical cover $X$. Since $X$ is topologically a 4-dimensional torus, we have that $\pi_1(X)=\zz^4$ and all higher homotopy groups vanish. By construction $G'=\zz/k_S\zz$ acts freely on $X$, and this gives rise to a short exact sequence of fundamental groups:
$$
1\to\pi_1(X)=\zz^4\to\pi_1(S)\to\zz/k_S\zz\to1.
$$
The fundamental group $\pi_1(S)$ is non-abelian and can be realized as a non-trivial extension of $\zz^4$ by the finite group $G' = \zz/k_S\zz$ (explicit presentation of $\pi_1(S)$ in terms of generators and relations can be found in \cite[Lemmas 2.1.5-11]{Ume75}). Groups with this structure are called Bieberbach and have been studied extensively; for more details one can see \cite[Chapter~I]{Cha12}.

Note that since the universal covering of $S$ is $\cc^2$ which is contractible it follows that $S$ is aspherical, i.e. it is the $K(\pi_1(S), 1)$-space (all higher homotopy groups of $S$ vanish). Since $S$ is a finite-dimensional $K(\pi_1(S),1)$, its fundamental group is torsion-free by \cite[Proposition~2.45]{Hat02}.

\subsection{Moduli spaces of semistable sheaves}
\label{sec:moduli_spaces}

\subsubsection{Two notions of support}
\label{sec:support}
There are two main notions of support of a sheaf on a surface both of which are needed for our purposes.

\begin{definition}\label{def:scheme_support}
Let $S$ be a smooth projective surface and let $F$ be a coherent sheaf on $S$. The scheme-theoretic support of $F$ is the closed subscheme
$$
\supp(F)\coloneqq V\bigl(\operatorname{Ann}(F)\bigr).
$$
Its underlying closed set is
$$
\{x\in S\mid F_x\neq 0\}.
$$
The dimension of $F$ is the dimension of this support.
\end{definition}

\begin{definition}\label{def:pure_sheaf}
A coherent sheaf $F$ is called pure of dimension $d$ if $\dim(G) = d$ for any non-trivial coherent subsheaf $0 \neq G \subset F$.
\end{definition}

\begin{definition}\label{def:fitting_support}
Let $F$ be a coherent sheaf on a scheme $X$. Its Fitting support
$\fsupp(F)$ is the closed subscheme defined by the zeroth Fitting ideal
$\operatorname{Fitt}_0(F)$.
\end{definition}

If $F$ is pure of dimension one on a smooth projective surface $S$, then it
has a length-one locally free resolution
$$
0\longrightarrow L_1\xrightarrow{a}L_0\longrightarrow F\longrightarrow 0,
$$
where $L_0$ and $L_1$ have the same rank. Locally,
$\operatorname{Fitt}_0(F)$ is generated by $\det a$, so $\fsupp(F)$ is an
effective Cartier divisor.

\begin{lemma}\label{lem:fitting_base_change}
\cite[Lemma C.10]{BBHR09} Let $p : X \to B$ be a morphism of algebraic varieties and $F$ a coherent sheaf on $X$. For every point $s \in B$, the restriction to the fiber $X_s$ satisfies
$$
\fsupp(F)_s = \fsupp(F) \cap X_s \cong \fsupp(F_s).
$$
Thus it is the closed subscheme defined by the same equation as the Fitting support of the restriction sheaf $F_s$.
\end{lemma}

In other words the Fitting support $\fsupp(F)$ is compatible with base changes. It is worth noting that the scheme-theoretic support $\supp(F)$ is not compatible with base change in general.

\begin{proposition}\label{prop:det_fitting}
\cite[Proposition C.11]{BBHR09} Let $F$ be a pure one-dimensional coherent
sheaf on a smooth projective surface $S$. Then
$$
\det F\cong\mathcal O_S(\fsupp(F)),
\qquad
[\fsupp(F)]=c_1(F).
$$
\end{proposition}

\begin{remark}\label{rem:det_formula}
We use brackets when viewing the effective Cartier divisor $\fsupp(F)$ as
its cohomology class.
\end{remark}

\subsubsection{Slope stability condition for coherent sheaves}
Recall that $(S, H)$ denotes a polarized bielliptic surface and $F$ is a coherent sheaf of pure dimension one on $S$. For our purposes it is enough to consider stability with respect to the slope function
$$
\mu_H(F) = \frac{\chi(F)}{c_1(F) \cdot H},
$$
where $\chi(F)$ denotes the Euler characteristic of $F$ as usual.

\begin{definition}\label{def:slope_stability}
A coherent sheaf $F$ on a polarized surface $(S,H)$ is called slope
$H$-semistable if
$$
\mu_H(G)\leqslant\mu_H(F)
$$
for every nonzero proper coherent subsheaf $0\neq G\subsetneq F$, and slope
$H$-stable if the corresponding strict inequality holds.
\end{definition}

\begin{remark}\label{rem:gieseker_slope}
For pure one-dimensional sheaves, Gieseker stability is equivalent to slope
stability because
$$
p_F(m)=m+\frac{\chi(F)}{c_1(F)\cdot H}.
$$
\end{remark}

\subsubsection{Mukai vector}\label{sec:mukai-vector}
For a coherent sheaf $F$ on a bielliptic surface $S$, we define its Mukai
vector by
$$
\mv(F)
\coloneqq
\left(\rank(F),c_1(F),\ch_2(F)\right)
\in H^0(S,\zz)\oplus H^2(S,\zz)\oplus H^4(S,\zz),
$$
where
$$
\ch_2(F)=\frac{1}{2}c_1^2(F)-c_2(F).
$$
The degree-four component is integral: by \cref{lem:num_div}, the
intersection form on a bielliptic surface is even, and hence
$c_1^2(F)/2\in\zz$.
For a pure dimension one sheaf this simplifies further to
$$
\mv(F)
=
\left(0,c_1(F),\frac{1}{2}c_1^2(F)-c_2(F)\right).
$$
For numerical purposes, we pass from $c_1(F)$ to its class in $\num(S)$. This
gives a class in the algebraic Mukai lattice
\begin{equation}\label{eq:mukai_lattice}
H^*_{\mathrm{alg}}(S,\zz) = H^0(S,\zz) \oplus \num(S) \oplus H^4(S,\zz)
\end{equation}
which, by abuse of notation, we also denote by $\mv(F)$. In this lattice it
is the usual Mukai vector
$$
\mv(F)=\ch(F)\sqrt{\td(S)}.
$$
Indeed, $\td(S)=1$ for a bielliptic surface, so numerically
$
\mv(F)=\ch(F).
$
The Mukai pairing on $H^*_{\mathrm{alg}}(S,\zz)$ is defined by
$$
\langle \mv(F),\mv(E)\rangle
\coloneqq
-\chi(F,E)
\coloneqq
-\sum_i(-1)^i\ext^i(F,E).
$$
By the Hirzebruch--Riemann--Roch theorem,
$$
\langle (r,c,s),(r',c',s')\rangle
=
c\cdot c'-rs'-r's.
$$
We write $\mv^2\coloneqq\langle\mv,\mv\rangle$. A Mukai vector $\mv$ is
called \textit{primitive} if it is primitive numerically, that
is, if it is not divisible as an element of $H^*_{\mathrm{alg}}(S,\zz)$.
If, numerically,
$
\mv=(r,aA_0+bB_0,s),
$
we define
$$
\ell(\mv)
\coloneqq
\gcd\left(r,a,\frac{k_S}{\la_S}b,k_Ss\right).
$$

\begin{remark}\label{rem:mukai_abuse}
The moduli space $M_{H,S}(\mv)$, the determinant morphism, and the condition
of having fixed first Chern class use the actual Mukai vector in
$$
H^0(S,\zz)\oplus H^2(S,\zz)\oplus H^4(S,\zz).
$$
The Mukai pairing, the square $\mv^2$, primitivity, $\ell(\mv)$, coordinates
in the basis $A_0,B_0$, and numerical Fourier--Mukai calculations use its
class in $H^*_{\mathrm{alg}}(S,\zz)$. We use the same notation $\mv$ for both
throughout.
\end{remark}

\subsubsection{Moduli spaces}\label{subsec:moduli-spaces}
Fix a Mukai vector $\mv$ with $\mv^2>0$ on a polarized bielliptic surface $(S,H)$. We denote by
$
M_{H,S}(\mv)
$
the moduli space of Gieseker $H$-semistable coherent sheaves on $S$ with Mukai vector $\mv$. There is a determinant morphism
$$
\det:M_{H,S}(\mv)\longrightarrow \Pic^{c_1(\mv)}(S),
$$
where $\Pic^{c_1(\mv)}(S)$ is a $\Pic^0(S)$-torsor of line bundles on $S$ with the first Chern class equal to $c_1(\mv)$.

Suppose now that $\mv=(0,[C],\chi)$ has rank zero. Since the Fitting support of a coherent sheaf behaves well in families, there is a support morphism
$$
h_{\mathrm{full}}:M_{H,S}(\mv)\longrightarrow \Hilb^{c_1(\mv)}(S)
$$
sending a pure one-dimensional sheaf $F$ to its Fitting support $\fsupp(F)$.
For a curve $C\subset S$ defining a point $[C]\in\Hilb^{c_1(\mv)}(S)$,
the fiber $h_{\mathrm{full}}^{-1}([C])$ is the Simpson moduli space of
$H_C$-semistable sheaves on $C$ with the prescribed invariants.

Assume that the locus of integral curves in the relative linear system
associated with $c_1(\mv)$ is nonempty and irreducible. Over this locus, the
sheaves which are locally free on their support form a relative Picard torsor
and hence an irreducible locus. If $M_{H,S}(\mv)$ is normal, its irreducible
components are pairwise disjoint, so this locus is contained in a unique
irreducible component. We denote this component by $M^\circ_{H,S}(\mv)$ and
call it the distinguished component.

We impose these hypotheses whenever the notation $M^\circ_{H,S}(\mv)$ is used
for a rank-zero vector. They hold in the admissible rank-zero cases considered
below. Indeed, if $\nt(c_1(\mv))\geqslant 3$, the relative complete linear
system is an irreducible projective bundle and its integral locus is a nonempty
open subset. Moreover,
$$
\mv^2=c_1(\mv)^2\geqslant 18\geqslant 3k_S,
$$
so $M_{H,S}(\mv)$ is normal by \cite[Theorem~1.2]{Nue25}.

Set
$
\det^\circ:=\det|_{M^\circ_{H,S}(\mv)}.
$
For $L\in\Pic^{c_1(\mv)}(S)$, denote by $M_{H,S}(\mv,L)$ the
irreducible component of
$
(\det^\circ)^{-1}(L)
$
containing the locus of sheaves with integral Fitting support which are locally
free as sheaves on their support. This component is unique in the cases
considered below. Indeed, \cref{lem:translation-base-change} shows that the
full determinant fiber is normal, so its irreducible components are pairwise
disjoint, while the locally free integral-support locus is irreducible.
When $H,S,\mv$ of rank $0$, and $L$ are fixed and no ambiguity can arise, we denote
$$
M\coloneqq M^\circ_{H,S}(\mv),
\qquad
M(L)\coloneqq M_{H,S}(\mv,L).
$$
Restricting the full support morphism to the distinguished component gives the
Le Potier map
$$
h:=h_{\mathrm{full}}|_{M^\circ_{H,S}(\mv)}:
M^\circ_{H,S}(\mv)\longrightarrow \Hilb^{c_1(\mv)}(S).
$$
Here $\Hilb^{c_1(\mv)}(S)$ is the Hilbert scheme parametrizing subschemes with cohomology class
$
[\fsupp(F)]=c_1(\mv).
$
Its fiber over $[C]$ is the intersection of $M^\circ_{H,S}(\mv)$ with the
corresponding Simpson fiber of $h_{\mathrm{full}}$.

On a bielliptic surface the long exact sequence in cohomology corresponding to the exponential short exact sequence reads as
$$
0 \to \Pic^0(S) \to \Pic(S) \xrightarrow{c_1} H^2(S,\zz) \to 0.
$$
Thus, when the determinant is fixed, $M_{H,S}(\mv,L)$ is the distinguished
irreducible component of the degree-$d$ relative compactified Jacobian of
$|C|$, and its support morphism is
$$
h:M_{H,S}(\mv,L)\longrightarrow |C|,
$$
where
$
d = \chi-\chi(\mathcal{O}_C).
$
Over the locus of integral curves, its fibers are the full compactified
Jacobians.

\begin{remark}\label{rem:grr}
By Grothendieck--Riemann--Roch for $i:C\hookrightarrow S$, if $E$ has rank $r$
and degree $d$, then
$$
\mv(i_*E)=(0,r[C],d+r(1-g)).
$$
Thus the support fiber over an integral curve $\Gamma\in |C|$ is the degree
$\chi-\chi(\mathcal O_\Gamma)$ compactified Jacobian of $\Gamma$.
\end{remark}

The following consequence of \cite[Theorem 1.2]{Yuan2023} will become useful for us later.
\begin{lemma}\label{ass:comp_jacobian}
Let \(S\) be a bielliptic surface, let
$
\mv=(0,c_1(L),\chi)
$
be a primitive Mukai vector, and let \(H\) be an \(\mv\)-generic
polarization. Assume that \(|L|\) contains integral curves, and let
$
h:M_{H,S}(\mv,L)\longrightarrow |L|
$
be the fixed-determinant support morphism. Then every non-empty fiber of \(h\) has dimension at most \(g(L)\). Moreover, over the locus of integral curves the fibers have dimension exactly \(g(L)\).

In particular, if \(Z\subset |L|\) is a closed subset of codimension \(c\), then
$$
\operatorname{codim}_{M_{H,S}(\mv,L)}h^{-1}(Z)\geqslant c
$$
whenever \(h^{-1}(Z)\neq\varnothing\).
\end{lemma}
\begin{proof}
Let \(D\in |L|\). Write
$
D=\sum_{j=1}^s n_jD_j
$
with \(D_j\) pairwise distinct integral curves. Since \(K_S\) is torsion on a bielliptic surface, it is numerically trivial. Hence
$
D_j\cdot K_S=0
$
for every \(j\).

The dimension estimate of \cite[Theorem 1.2]{Yuan2023} for
one-dimensional sheaves on possibly non-reduced curves applies under the
hypothesis
$
D_j\cdot K_S\leqslant 0
$
for all \(j\). More precisely, \cite[Corollary 1.3]{Yuan2023} bounds the
dimension of the moduli-stack fiber by \(g(L)-1\). Since \(\mv\) is
primitive and \(H\) is \(\mv\)-generic, every
\(H\)-semistable sheaf of Mukai vector \(\mv\) is stable and has
automorphism group \(\cc^*\). Hence the dimension of the corresponding
coarse support fiber is one greater than the stack dimension. Moreover,
if \(\fsupp(F)=D\), then \cref{prop:det_fitting} gives
$$
\det F\cong\mathcal O_S(D)\cong L,
$$
so this coarse support fiber is precisely the fixed-determinant fiber
\(h^{-1}(D)\). Therefore \cite[Corollary 1.3]{Yuan2023} gives
$$
\dim h^{-1}(D)\leqslant g(L).
$$

If \(D\) is integral, then the fiber is the degree
$
d=\chi-\chi(\mathcal O_D)
$
compactified Jacobian of \(D\). Since \(D\) is an integral locally planar curve, its compactified Jacobian has dimension \(g(L)\). Thus the bound is sharp over the integral locus.

Finally, if \(Z\subset |L|\) has codimension \(c\), then every irreducible component of \(h^{-1}(Z)\) has dimension at most
$
\dim Z+g(L).
$
Since \(|L|\) contains integral curves, the open locus over the integral part of \(|L|\) is non-empty and has fibers compactified Jacobians of dimension \(g(L)\). Therefore the component \(M_{H,S}(\mv,L)\) has dimension \(\dim |L|+g(L)\).
We get
$
\operatorname{codim}h^{-1}(Z)\geqslant c.
$
\end{proof}

\subsubsection{Smoothness and dimension}
First we prove a technical lemma which will be very useful later.
\begin{lemma}\label{lem:translation-base-change}
Let $(S,H)$ be a polarized bielliptic surface where $S = (A \times B)/G$, and let
$
\mv=(0,c_1(L),\chi)
$
be a rank-zero Mukai vector with \(\mv^2>0\), and assume that \(c_1(L)\) is represented by an effective divisor.
Then the morphism
$$
\rho_L:A\longrightarrow \Pic^{c_1(\mv)}(S),
\qquad
a\longmapsto t_a^*L,
$$
is finite étale and surjective. Moreover, translation induces an isomorphism
$$
A\times_{\Pic^{c_1(\mv)}(S)} M
\cong
A\times(\det^\circ)^{-1}(L),
\qquad
(a,F)\longmapsto(a,t_{-a}^*F),
$$
whose inverse is
$
(a,E)\longmapsto(a,t_a^*E).
$
Consequently, if \(N\) is an irreducible component of \((\det^\circ)^{-1}(L)\), then
$
\dim N=\dim M-1.
$
Furthermore:
\begin{enumerate}
\item if \(M\) is smooth, then \(N\) is smooth;
\item if \(M\) is normal, then the full determinant fiber
$(\det^\circ)^{-1}(L)$ is normal. In particular, its irreducible components
are pairwise disjoint and normal;
\item if \(M\) has terminal l.c.i. singularities and torsion canonical bundle, then \(N\) has terminal l.c.i. singularities and torsion canonical bundle.
\end{enumerate}
\end{lemma}
\begin{proof}
Write
$
c_1(L)\equiv aA_0+bB_0.
$
Since \(c_1(L)\) is represented by an effective divisor, \cref{lem:num_div} gives \(a,b\geqslant 0\). Since
$
\mv^2=c_1(L)^2=2ab>0,
$
we have \(a,b>0\). Hence \(L\) is ample by \cref{lem:num_div}.

Translations of the first factor of \(A\times B\) induce automorphisms
$$
t_a:S\to S,
\qquad
a\in A.
$$
Since \(A\) is connected, these automorphisms act trivially on cohomology and preserve the numerical class of \(H\). Thus pullback preserves the Mukai vector \(\mv\) and \(H\)-semistability. It also preserves integral Fitting support and local freeness on the support, and therefore preserves the distinguished component $M^\circ_{H,S}(\mv)$.

Consider
$$
\varphi_L:A\longrightarrow\Pic^0(S),
\qquad
a\longmapsto t_a^*L\otimes L^{-1}.
$$
The pullback of \(L\) to the finite cover \(A\times B\to S\) is ample, so its restriction to \(A\times\{b\}\) has positive degree. Hence the induced polarization homomorphism on \(A\) is nonzero, and therefore \(\varphi_L\) is nonzero. Since both \(A\) and \(\Pic^0(S)\) are elliptic curves, \(\varphi_L\) is an isogeny. Consequently
$$
\rho_L:A\to P,
\qquad
a\mapsto t_a^*L,
$$
is finite étale and surjective.

The restricted determinant morphism $\det^\circ:M\to P$ is surjective. Indeed, choose $F_0$ in the distinguished locus and set $L_0:=\det(F_0)$. The same argument shows that
$$
A\longrightarrow P,
\qquad
a\longmapsto t_a^*L_0,
$$
is surjective. Since translation preserves $M$, the determinants of the translates $t_a^*F_0$ run through all of $P$.

Moreover, translation induces an isomorphism
$$
A\times_P M \xrightarrow{\ \sim\ } A\times (\det^\circ)^{-1}(L),
\qquad
(a,F)\longmapsto (a,t_{-a}^*F),
$$
whose inverse is
$
(a,E)\longmapsto (a,t_a^*E).
$
Indeed, if $(a,F)\in A\times_P M$, then $\det(F)\simeq t_a^*L$,
and hence 
$
\det(t_{-a}^*F)
\simeq t_{-a}^*\det(F)
\simeq t_{-a}^*t_a^*L
\simeq L.
$

Conversely, if $\det(E)\simeq L$, then
$\det(t_a^*E)\simeq t_a^*L=\rho_L(a)$,
so that $(a,t_a^*E)\in A\times_P M$. These two maps are mutually inverse.

Assume that $M$ is normal. Since $A\times_P M\to M$ is finite étale,
$A\times_P M$ is normal. Under the displayed isomorphism, this says that
$
A\times(\det^\circ)^{-1}(L)
$
is normal. The projection
$
A\times(\det^\circ)^{-1}(L)
\longrightarrow(\det^\circ)^{-1}(L)
$
is smooth and surjective, and normality descends under smooth surjective
morphisms. Hence the full determinant fiber $(\det^\circ)^{-1}(L)$ is normal.
The irreducible components of a normal scheme are disjoint, so each of them is
both open and closed and is itself normal.

Let \(N\) be an irreducible component of \((\det^\circ)^{-1}(L)\). Then
\(A\times N\) is an irreducible component of
\(A\times(\det^\circ)^{-1}(L)\). Since \(A\times_P M\to M\) is finite
étale, smoothness and terminal l.c.i. singularities pass from \(M\) to
\(A\times_P M\), and hence to its irreducible component \(A\times N\).
These properties then descend from \(A\times N\) to \(N\) because
$
A\times N\to N
$
is smooth and surjective. Assume in addition that the canonical bundle of
$M$ is torsion. Since
$A\times_P M\to M$ is finite étale, the canonical bundle of $A\times_P M$
is torsion. Under the isomorphism
$
A\times_P M\simeq A\times(\det^\circ)^{-1}(L),
$
the component corresponding to $N$ is $A\times N$, so $K_{A\times N}$ is
torsion. Since $A$ is an elliptic curve, $K_A\simeq\mathcal O_A$ and
$
K_{A\times N}\simeq \operatorname{pr}_N^*K_N.
$
Restricting a torsion trivialization to a slice $\{a\}\times N$ shows that
$K_N$ is torsion.

The dimension statement follows from
$
\dim(A\times N)=\dim(A\times_P M)=\dim M
$
and \(\dim A=1\).
\end{proof}

\begin{theorem}\label{thm:smooth_dim}
Let $(S,H)$ be a polarized bielliptic surface, and let $\pi:X\to S$ be the canonical cover. Let $\mv$ be a primitive Mukai vector with $\mv^2>0$, and let $N$ be an irreducible component of $M_{H,S}(\mv)$. Assume that $\ell(\mv)=1$, that $H$ is $\mv$-generic, and that $\pi^*H$ is $\pi^*\mv$-generic. Then $N$ is a smooth irreducible projective variety of dimension $\mv^2+1$.

Moreover, assume that $\mv=(0,c_1(L),\chi)$ has rank zero, and $L \in \pic^{c_1(\mv)}(S)$ is
ample. Then
$M_{H,S}(\mv,L)$ is smooth, irreducible, and projective of dimension $\mv^2$.
\end{theorem}

\begin{remark}\label{rem:yoshioka_walls}
Since $\pi$ is finite, $\pi^*H$ is ample. The assumption that $\pi^*H$ is
$\pi^*\mv$-generic ensures that $\pi^*H$-semistability coincides with
$\pi^*H$-stability. The assumption $\ell(\mv)=1$ implies that both
$\pi^*\mv$ and $\mv$ are primitive \cite[Lemma~2.3]{Nue25}.
\end{remark}

\begin{proof}
Since \(H\) is \(\mv\)-generic, every \(H\)-semistable sheaf of Mukai vector \(\mv\) is stable. Let $F\in M_{H,S}(\mv)$. We first prove smoothness of \(M_{H,S}(\mv)\). The pullback of an \(H\)-semistable sheaf under the finite étale cover \(\pi\) is \(\pi^*H\)-semistable. By Nuer's Lemma 2.3, the condition \(\ell(\mv)=1\) implies that \(\pi^*\mv\) is primitive. Since \(\pi^*H\) is \(\pi^*\mv\)-generic, \(\pi^*H\)-semistability coincides with \(\pi^*H\)-stability for sheaves of Mukai vector \(\pi^*\mv\). Hence \(\pi^*F\) is stable. In particular,
$
\Hom_X(\pi^*F,\pi^*F)=\cc.
$
The canonical cover is the cyclic cover associated with the torsion canonical bundle. Thus
$$
\pi_*\mathcal O_X\simeq \bigoplus_{i=0}^{k_S-1}K_S^{-i}.
$$
By the projection formula,
$$
\pi_*\pi^*F
\simeq
F\otimes\pi_*\mathcal O_X
\simeq
\bigoplus_{i=0}^{k_S-1}F\otimes K_S^{-i}.
$$
Therefore
$$
\Hom_X(\pi^*F,\pi^*F)
\simeq
\Hom_S(F,\pi_*\pi^*F)
\simeq
\bigoplus_{i=0}^{k_S-1}
\Hom_S(F,F\otimes K_S^{-i}).
$$
The summand with \(i=0\) is
$
\Hom_S(F,F)=\cc,
$
because \(F\) is stable. Since the whole direct sum is already one-dimensional, all the other summands vanish. In particular,
$
\Hom_S(F,F\otimes K_S)=0.
$
By Serre duality,
$$
\Ext^2_S(F,F)
\simeq
\Hom_S(F,F\otimes K_S)^\vee
=0.
$$
Thus the obstruction space vanishes at every point of \(M_{H,S}(\mv)\), and \(M_{H,S}(\mv)\) is smooth.

Since \(F\) is stable, we have $\ext^0(F,F)=1$. Moreover, Riemann--Roch gives $\chi(F,F)=-\langle\mv,\mv\rangle=-\mv^2$.
Using \(\Ext^2(F,F)=0\), we obtain
$$
\dim \Ext^1(F,F)
=
\ext^0(F,F)-\chi(F,F)
=
1+\mv^2.
$$
Therefore every irreducible component of $M_{H,S}(\mv)$ has dimension $\mv^2+1$.
Projectivity follows from Simpson's construction of the Gieseker moduli space
\cite[Theorem~1.21(2)]{Simpson94}.
In particular, $N$ is a smooth irreducible projective variety of dimension $\mv^2+1$.

Now assume that $\mv=(0,c_1(L),\chi)$ has rank zero and that \(L\) is ample. We prove the fixed-determinant statement.
Since \(L\) is ample, \cref{lem:num_div} gives \(h^0(S,L)>0\), so
\(c_1(L)\) is represented by an effective divisor. Hence
\cref{lem:translation-base-change} applies and shows that
$M_{H,S}(\mv,L)$ is smooth and
$$
\dim M_{H,S}(\mv,L)=\dim M^\circ_{H,S}(\mv)-1=\mv^2.
$$
Its irreducibility and projectivity follow from its definition as an
irreducible component of the projective moduli space.
\end{proof}

\begin{corollary}\label{cor:rank-zero-terminal-lci}
Let $(S,H)$ be a polarized bielliptic surface, let
$
\mv=(0,c_1(L),\chi)
$
be a primitive Mukai vector, and assume that $H$ is $\mv$-generic and
$\nt(L)\geqslant 3$. Then $M^\circ_{H,S}(\mv)$ and
$M_{H,S}(\mv,L)$ are normal projective varieties with terminal l.c.i.
singularities and torsion canonical bundle.
\end{corollary}
\begin{proof}
Write $c_1(L)\equiv aA_0+bB_0$. The numerical-threshold assumption gives
$a,b\geqslant 3$, and hence
$$
\mv^2=L^2=2ab\geqslant 18\geqslant 3k_S.
$$
By \cite[Theorem~1.2]{Nue25}, the full moduli space $M_{H,S}(\mv)$ is
normal and has terminal l.c.i. singularities and a torsion Cartier canonical
divisor. Therefore its distinguished irreducible component
$M^\circ_{H,S}(\mv)$ has the same properties. Moreover, $L$ is ample by
\cref{lem:num_div}, so $c_1(L)$ is represented by an effective divisor.
Thus \cref{lem:translation-base-change} gives the corresponding properties
for $M_{H,S}(\mv,L)$. Projectivity follows because both spaces are
irreducible components of projective moduli spaces.
\end{proof}

\section{Orbits of Mukai vectors}\label{sec:orbits}
\subsection{Admissible Mukai vectors and reduction from positive rank to rank zero}\label{subsec:admissible}

This subsection establishes a large class of positive-rank Mukai vectors whose moduli spaces are birationally reducible to rank-zero cases.

Every Fourier--Mukai autoequivalence acts both on the full Grothendieck group, and hence on the actual Chern data, and on the algebraic Mukai lattice. All coordinate formulas in this section describe the numerical action. Thus, when we write $\Phi_*(\mv)=(r,D,s)$ numerically, $D$ is the numerical class of the actual first Chern class of the transformed vector. Its torsion component is retained in the actual Mukai vector but does not affect the rank, Mukai square, numerical threshold, or the invariant $\ell$.

Before proceeding, we introduce the following definition which significantly simplifies notations in later considerations.

\begin{definition}\label{def:admissible-vector}
A Mukai vector $\mv=(R,c_1(L),s)$ on a bielliptic surface \(S\) is called \textit{admissible} if either $R=0$ and $\nt(c_1(L))\geqslant 3$, or \(R>0\) and there exists a derived autoequivalence $\Phi\in\operatorname{Auteq}D^b(S)$ such that \(\Phi_*(\mv)\) has rank zero and $\nt(c_1(\Phi_*(\mv)))\geqslant 3$.
\end{definition}

\begin{remark}
The condition \(\nt(c_1)\geqslant 3\) is stronger than the very-ampleness condition needed for some intermediate arguments. We use it in the definition of admissibility because this is precisely the hypothesis under which the rank-zero Betti-number computations in \cref{thm:b1-b2-moduli} are proved. By \cref{lem:bertini_results}, \(\nt(c_1)\geqslant 3\) implies the required very-ampleness and Lefschetz properties of the corresponding linear system.
\end{remark}

An immediate and the most important use of the definition is the following lemma.

\begin{lemma}\label{lem:admissible-positive-rank-to-rank-zero}
Let $S$ be a bielliptic surface, and let $\mv=(R,aA_0+bB_0,s)$ be a primitive admissible Mukai vector with $R>0$. Let $H$ be a
$\mv$-generic polarization. Then there exist a primitive rank-zero Mukai vector
$
\mw=(0,D,\chi),\, \nt(D)\geqslant 3,
$
a $\mw$-generic polarization $H'$, and an irreducible component
$
N\subset M_{H,S}(\mv)
$
such that
$$
N\dashrightarrow M^\circ_{H',S}(\mw)
$$
is birational.
\end{lemma}

\begin{proof}
By admissibility, there exists a Fourier--Mukai autoequivalence $\Phi\in\operatorname{Auteq}D^b(S)$ such that $\mw\coloneqq\Phi_*(\mv)$ has rank zero and \(c_1(\mw)=D\) satisfies $\nt(D)\geqslant 3$. Since $\Phi_*$ is an isometry of the Mukai lattice, $\mw$ is primitive and $\mw^2=\mv^2$.

Let $\sigma$ be a generic Bridgeland stability condition in the Gieseker
chamber corresponding to the $\mv$-generic polarization $H$. Then
$$
M_{\sigma,S}(\mv)\cong M_{H,S}(\mv).
$$
The autoequivalence $\Phi$ acts on the stability manifold and identifies
moduli spaces:
$$
M_{\sigma,S}(\mv)\cong M_{\Phi_*\sigma,S}(\mw).
$$
Indeed, an object $E\in D^b(S)$ is $\sigma$-stable of class $\mv$ if and
only if $\Phi(E)$ is $\Phi_*\sigma$-stable of class $\mw$.

The stability condition $\Phi_*\sigma$ need not lie in a Gieseker chamber
for $\mw$. However, Nuer's wall-crossing theorem
\cite[Theorem 1.3]{Nue25} implies that moduli spaces corresponding to
generic Bridgeland stability conditions for the same primitive Mukai
vector are related by wall crossing. Choose a generic stability condition
$\sigma'$ in a Gieseker chamber corresponding to a $\mw$-generic
polarization $H'$. At each wall, restrict the wall-crossing correspondence
to the common dense stable locus of the chosen component and define the
corresponding component on the adjacent side as the closure of its image.
Tracing the distinguished component
$$
M^\circ_{H',S}(\mw)\cong M^\circ_{\sigma',S}(\mw)
$$
backwards through these wall crossings and through the equivalence induced by
$\Phi$ determines an irreducible component
$$
N\subset M_{H,S}(\mv)
$$
birational to $M^\circ_{H',S}(\mw)$.
\end{proof}

For every positive-rank admissible Mukai vector $\mv$ and every
$\mv$-generic polarization $H$, fix the reduction data and
the irreducible component $N$ supplied by
\cref{lem:admissible-positive-rank-to-rank-zero}. By abuse of notation, we set
$$
M^\circ_{H,S}(\mv)\coloneqq N
$$
and also call it the distinguished component. We also set
$$
\det^\circ\coloneqq\det|_{M^\circ_{H,S}(\mv)}.
$$
All statements involving the positive-rank distinguished component are
understood with respect to these fixed choices.

Now we present a large class of positive-rank Mukai vectors whose
Fourier--Mukai orbits contain rank-zero vectors. To conclude admissibility
in the sense of \cref{def:admissible-vector}, one still has to ensure that
the resulting rank-zero first Chern class has numerical threshold at least
\(3\); this will be addressed after the rank-zero reduction results. We start with the
following lemma.
\begin{lemma}\label{lem:starting-vector-exceptional-congruences}
Let $S$ be a bielliptic surface, and let
$
\mv=(R,aA_0+bB_0,s)
$
be a primitive positive-rank Mukai vector with $\mv^2>0$.
Assume the following:
$$
\begin{array}{c|l}
\text{type of }S&\text{condition on }(R,a,b)\\
\hline
1,3,5,7 & \text{no extra condition},\\
2&R\equiv0\pmod 4,\ \text{or }R\equiv2\pmod 4\text{ and }
(a,b)\not\equiv(0,0)\pmod 2,\\
4&R\equiv0\pmod 8,\\
&\text{or }R\equiv4\pmod 8\text{ and }
(a\not\equiv0\pmod 4\text{ or }b\not\equiv0\pmod 2),\\
&\text{or }R\equiv2,6\pmod 8\text{ and }a\not\equiv0\pmod 2,\\
6&R\equiv0\pmod 9,\ \text{or }R\equiv3,6\pmod 9\text{ and }
(a,b)\not\equiv(0,0)\pmod 3.
\end{array}
$$
Then there exists $\Phi\in\operatorname{Auteq}D^b(S)$ such that
$\rank(\Phi_*\mv)=0$.
\end{lemma}

In order to prove this lemma, we use Nuer's reduction for positive-rank Mukai
vectors. Let
$
\mv=(r,aA_0+bB_0,s),
$
be a primitive Mukai vector of positive rank with $\mv^2\geqslant 0$. Nuer's
reduction transforms $\mv$, by a finite sequence of standard
autoequivalences, into a primitive Mukai vector $\mv_0$ of one of the forms
listed in \cref{tab:reduced-positive-rank-vectors}, with
$$
1\leqslant\rank(\mv_0)\leqslant r;
$$
see \cite[Theorem 1.4 and Table 1]{Nue25}. Thus the numerical problem of
finding a rank-zero vector in the Fourier--Mukai orbit of $\mv$ reduces to
the following list, where \(q>0\), \(s\in\zz\), and in the last row
\(b\in\zz\). In each case,
\(\mv_0\) is assumed to be primitive.

\begin{table}[h!]
\caption{Reduced forms of primitive positive-rank Mukai vectors on bielliptic
surfaces, see \cite[Table 1]{Nue25}}
\label{tab:reduced-positive-rank-vectors}
\centering
\begin{tabularx}{\textwidth}{ccX}
\toprule
Type of \(S\) & \(k_S=\ord(K_S)\) & Reduced forms of \(\mv_0\) \\
\midrule
1 & \(2\) &
\((q,0,s)\), \((2q,qB_0,s)\) \\[4pt]
2 & \(2\) &
\((q,0,s)\), \((2q,qA_0,s)\), \((2q,qB_0,s)\),
\((2q,qA_0+qB_0,s)\) \\[4pt]
3 & \(4\) &
\((q,0,s)\), \((4q,qB_0,s)\), \((2q,qB_0,s)\) \\[4pt]
4 & \(4\) &
\((q,0,s)\), \((2q,qA_0,s)\), \((4q,qB_0,s)\),
\((2q,qB_0,s)\), \((4q,2qA_0+qB_0,s)\),
\((2q,qA_0+qB_0,s)\) \\[4pt]
5 & \(3\) &
\((q,0,s)\), \((3q,qB_0,s)\) \\[4pt]
6 & \(3\) &
\((q,0,s)\), \((3q,qA_0,s)\), \((3q,qB_0,s)\),
\((3q,qA_0+qB_0,s)\) \\[4pt]
7 & \(6\) &
\((q,0,s)\), \((6q,qB_0,s)\), \((3q,qB_0,s)\),
\((2q,qB_0,s)\), \((q,bB_0,s)\), where
\(\frac{1}{3}<\frac{b}{q}<\frac{1}{2}\) \\
\bottomrule
\end{tabularx}
\end{table}
\FloatBarrier

At the level of Mukai vectors, Nuer's reduction is realized by a finite
sequence of standard autoequivalences, including tensoring by line bundles,
dualizing and shifting, and relative Fourier--Mukai transforms along the two
elliptic fibrations \(p_A\) and \(p_B\); see
\cite[Section~10, Proposition~10.5, and Theorem~1.4]{Nue25}.
Consequently, it suffices below to study the corresponding action on Mukai
vectors.

\begin{lemma}
\label{lem:reduction-positive-rank-to-zero}
Let \(S\) be a bielliptic surface and let
$
\mv_0=(r,aA_0+bB_0,s)
$
be one of the reduced positive-rank Mukai vectors listed in
\cref{tab:reduced-positive-rank-vectors}.
Assume \(\mv_0^2>0\). Suppose either that \(\mv_0\) is not one of the
four exceptional reduced forms
$$
\begin{array}{c|c}
\text{type of }S&\text{reduced vector}\\
\hline
2,4,6&(q,0,s)\\
4&(2q,qB_0,s),
\end{array}
$$
or that it is one of these reduced forms and satisfies the corresponding
additional condition
$$
\begin{array}{c|c|c}
\text{type of }S&\text{reduced vector}&\text{condition}\\
\hline
2&(q,0,s)&4\mid q\\
4&(q,0,s)&8\mid q\\
6&(q,0,s)&9\mid q\\
4&(2q,qB_0,s)&4\mid q.
\end{array}
$$
Then a finite composition of relative Fourier--Mukai autoequivalences
sends \(\mv_0\) to a rank-zero Mukai vector
$
\mw=(0,D,\chi)
$
with
$
\mw^2=\mv_0^2.
$
\end{lemma}

The next two lemmas are technical results that will be used in the proof of this fact.

\begin{lemma}\label{lem:admissible-rank-zero-transform}
Let \(p\colon Y\to C\) be a projective relatively minimal elliptic surface
with minimal multisection degree \(n\), and let $\mv$ be a Mukai vector of rank \(R>0\)
and fiber degree \(d\). Set \(g=\gcd(R,d)\). A pair consisting of a relative
moduli surface $J\to C$ and a relative Fourier--Mukai equivalence
$$
\Phi\colon D^b(Y)\xrightarrow{\sim}D^b(J)
$$
for which $\Phi_*(\mv)$ has rank zero exists if and only if
$$
\gcd\left(\frac{d}{g},n\right)=1.
$$
If such a pair has $J\simeq Y$, then choosing an isomorphism turns $\Phi$ into an
autoequivalence of $D^b(Y)$.
\end{lemma}

\begin{proof}
By \cite[Theorem~5.3]{Bri98}, the inverse of a relative Fourier--Mukai
equivalence associated with $p$, composed with a shift if necessary, acts on
the rank and fiber degree by a matrix
$$
\begin{pmatrix}
\alpha&\beta\\
\gamma&\delta
\end{pmatrix}
\in\operatorname{SL}_2(\zz),
\qquad
\beta>0,
\qquad
n\mid\gamma.
$$
More precisely, its target is the relative moduli surface determined by the
admissible matrix, and
$$
\binom{R'}{d'}
=
\begin{pmatrix}
\alpha&\beta\\
\gamma&\delta
\end{pmatrix}
\binom{R}{d}.
$$
To kill the rank, its first row must be
$$
\left(-\frac{d}{g},\frac{R}{g}\right).
$$
Writing \(\gamma=nt\), this row can be completed to an admissible matrix
precisely when
$$
-\frac{d}{g}\delta-\frac{R}{g}nt=1
$$
has an integral solution. Since
\(\gcd(d/g,R/g)=1\), this is equivalent to
$$
\gcd\left(\frac{d}{g},\frac{Rn}{g}\right)
=
\gcd\left(\frac{d}{g},n\right)
=1.
$$
\end{proof}

\begin{remark}\label{rem:bielliptic-relative-fm}
When $Y=S$ is bielliptic and $p$ is one of the natural elliptic fibrations
$p_A$ and $p_B$, the relative moduli surface $J$ occurring in
\cref{lem:admissible-rank-zero-transform} is isomorphic to $S$; see
\cite[Section~2.5, in particular the discussion following
Theorem~2.14]{Toc26}. After choosing this identification, the relative
Fourier--Mukai equivalence becomes an autoequivalence of $D^b(S)$. We make
this choice for all relative transforms below.
\end{remark}

Recall that for the two elliptic fibrations $p_A$ and $p_B$ on \(S\), the minimal multisection degrees
and fiber degrees of
\(\mv=(R,aA_0+bB_0,s)\) are correspondingly
$$
n_A=\la_S,
\qquad
d_A=\la_Sa,
\qquad
n_B=k_S,
\qquad
d_B=k_Sb.
$$

\begin{lemma}
\label{lem:two-step-exceptional-rank-zero}
Let $S$ be a bielliptic surface. Suppose
\begin{enumerate}
    \item Either $S$ is of type 2, 4, or 6 and
    $
    \mv_0=(q,0,s)
    $
    is a primitive Mukai vector such that
    $
    |G| = \la_Sk_S\mid q,
    $

    \item Or \(S\) is of type \(4\) and
    $
    \mv_0=(2q,qB_0,s)
    $
    is a primitive Mukai vector such that \(4\mid q\).
\end{enumerate}
Then \(\mv_0\) can be sent to a rank-zero Mukai vector by a composition of
two relative Fourier--Mukai autoequivalences.
\end{lemma}

\begin{proof}
For the first statement, primitivity of \(\mv_0=(q,0,s)\) gives
\(\gcd(q,s)=1\). Apply the inverse of the relative Fourier--Mukai
autoequivalence along \(p_A\) corresponding to the matrix
$\begin{pmatrix}
    1 & 1 \\
    0 & 1
\end{pmatrix}$
(see \cite[Proposition 10.5]{Nue25}).
Its numerical action is
$$
\Phi_A^{-1}(R,aA_0+bB_0,s)
=
(R-\la_Sa,aA_0+(b-\la_Ss)B_0,s),
$$
so
$
\Phi_A^{-1}(\mv_0)
=
(q,-\la_SsB_0,s).
$
For \(p_B\), the transformed vector has
$$
R=q,
\qquad
d_B=-k_S\la_Ss,
\qquad
n_B=k_S.
$$
Since \(k_S\la_S\mid q\) and \(\gcd(q,s)=1\),
$$
g=\gcd(R,d_B)=k_S\la_S
\qquad\text{and}\qquad
\frac{d_B}{g}=-s.
$$
Moreover, \(k_S\mid q\) implies \(\gcd(s,k_S)=1\). Hence
$$
\gcd\left(\frac{d_B}{g},n_B\right)
=
\gcd(-s,k_S)
=1.
$$
By \cref{lem:admissible-rank-zero-transform} and
\cref{rem:bielliptic-relative-fm}, a relative Fourier--Mukai
autoequivalence along \(p_B\) sends \(\Phi_A^{-1}(\mv_0)\) to rank zero.

For the second statement, \(k_S=4\) and \(\la_S=2\). Primitivity of
\(\mv_0=(2q,qB_0,s)\) gives \(\gcd(q,s)=1\). The same first transform
gives
$$
\Phi_A^{-1}(\mv_0)
=
(2q,(q-2s)B_0,s).
$$
Write \(q=4Q\). Then \(s\) is odd, and for \(p_B\) we have
$$
R=8Q,
\qquad
d_B=4(q-2s)=8(2Q-s),
\qquad
n_B=4.
$$
Since
$$
\gcd(Q,2Q-s)=\gcd(Q,s)=1,
$$
it follows that
$$
g=\gcd(R,d_B)=8
\qquad\text{and}\qquad
\frac{d_B}{g}=2Q-s.
$$
The latter integer is odd, so
$$
\gcd\left(\frac{d_B}{g},n_B\right)
=
\gcd(2Q-s,4)
=1.
$$
Another application of \cref{lem:admissible-rank-zero-transform} and
\cref{rem:bielliptic-relative-fm} completes the proof.
\end{proof}

\begin{proof}[Proof of \cref{lem:reduction-positive-rank-to-zero}]
For each type of bielliptic surface, number the reduced forms in
\cref{tab:reduced-positive-rank-vectors} from left to right. For a chosen
fibration \(p\in\{p_A,p_B\}\), let \(d_p\) be the fiber degree of \(\mv_0\),
let \(n_p\) be its minimal multisection degree, and set
$$
g_p\coloneqq\gcd(\rank(\mv_0),d_p).
$$
Thus
$$
d_A=\la_Sa,\qquad n_A=\la_S,
\qquad
d_B=k_Sb,\qquad n_B=k_S
$$
for \(\mv_0=(R,aA_0+bB_0,s)\). The following table records a fibration to
which \cref{lem:admissible-rank-zero-transform} applies and the corresponding
value of \(d_p/g_p\). A dash indicates one of the exceptional reduced forms
treated separately by \cref{lem:two-step-exceptional-rank-zero}.
\begingroup
\small
\setlength{\arraycolsep}{4pt}
$$
\begin{array}{c|c|c|c}
\text{type}&\text{indices of reduced forms in
\cref{tab:reduced-positive-rank-vectors}}&
\text{fibration }p&d_p/g_p\\
\hline
1&1,\ 2&p_A,\ p_B&0,\ 1\\
2&1,\ 2,\ 3,\ 4&-,\ p_A,\ p_B,\ p_A&-,\ 1,\ 1,\ 1\\
3&1,\ 2,\ 3&p_A,\ p_B,\ p_A&0,\ 1,\ 0\\
4&1,\ 2,\ 3,\ 4,\ 5,\ 6&
-,\ p_A,\ p_B,\ -,\ p_A,\ p_A&
-,\ 1,\ 1,\ -,\ 1,\ 1\\
5&1,\ 2&p_A,\ p_B&0,\ 1\\
6&1,\ 2,\ 3,\ 4&-,\ p_A,\ p_B,\ p_A&-,\ 1,\ 1,\ 1\\
7&1,\ 2,\ 3,\ 4,\ 5&p_A,\ p_B,\ p_A,\ p_A,\ p_A&
0,\ 1,\ 0,\ 0,\ 0
\end{array}
$$
\endgroup
Whenever \(d_p/g_p=1\), the criterion
$$
\gcd(d_p/g_p,n_p)=1
$$
of \cref{lem:admissible-rank-zero-transform} is automatic. When
\(d_p/g_p=0\), the chosen fibration has \(n_p=1\), so the criterion again
holds. Thus, by \cref{rem:bielliptic-relative-fm}, every nonexceptional
reduced form admits a single relative Fourier--Mukai autoequivalence sending
it to rank zero.

\cref{lem:two-step-exceptional-rank-zero} applies to the exceptional reduced
forms satisfying the additional divisibility conditions. Indeed,
\(|G|\) is \(4\), \(8\), and \(9\) in Types \(2\), \(4\), and
\(6\), respectively, while the remaining exceptional form in Type \(4\) is
covered when
\(4\mid q\). Thus all the cases in the statement admit a finite
composition of relative Fourier--Mukai autoequivalences sending them to rank
zero.

After choosing the identifications of the intervening relative moduli surfaces
with $S$ as in \cref{rem:bielliptic-relative-fm}, let
$\Phi\in\operatorname{Auteq}D^b(S)$ denote the resulting composition. It
preserves the Mukai pairing, and hence
$
\mw^2=\mv_0^2.
$
\end{proof}

\begin{proof}[Proof of \cref{lem:starting-vector-exceptional-congruences}]
We show that under the initial assumptions, \cref{lem:reduction-positive-rank-to-zero} can be applied and the result immediately follows from it.
We use the congruence behavior of the elementary operations appearing in
Nuer's reduction. To simplify notations, write $(R,a,b,s)$ for the Mukai vector
$(R,aA_0+bB_0,s)$. Tensoring by a line bundle of numerical class
$xA_0+yB_0$ sends
$$
(R,a,b,s)\longmapsto
(R,a+Rx,b+Ry,s+ay+bx+Rxy).
$$
The elementary relative Fourier--Mukai transforms used in the reduction
have numerical form
$$
\Phi_A^{\pm1}(R,a,b,s)
=
(R\mp\la_Sa,a,b\mp\la_Ss,s)
$$
and
$$
\Psi_B^{\pm1}(R,a,b,s)
=
(R\mp k_Sb,a\mp k_Ss,b,s).
$$
Dualizing and shifting only change signs and therefore do not affect the
congruence arguments below.

For types $1,3,5,7$, there is nothing to prove: none of the reduced rows
which occur in these types is exceptional in
\cref{lem:reduction-positive-rank-to-zero}.

Consider type $2$. Here $k_S=\la_S=2$, so the parity of the rank is
invariant. Assume first that $R\equiv2\pmod 4$ and at least one of $a,b$
is odd. Since every rank occurring in the reduction is even, the parities
of $a$ and $b$ are invariant. Therefore the endpoint cannot be
$(q,0,s_0)$, whose divisor coefficients are both even, and hence it is
nonexceptional. If $R\equiv0\pmod 4$ and the endpoint is $(q,0,s_0)$,
then parity invariance shows that $a$ and $b$ are even throughout the
reduction. Every rank change is therefore divisible by $4$, so $q\equiv R\equiv0\pmod 4$.

Consider type $6$. Here $k_S=\la_S=3$. Assume first that
$R\equiv3,6\pmod 9$ and at least one of $a,b$ is not divisible by $3$.
All ranks occurring in the reduction are divisible by $3$, so $a$ and
$b$ are invariant modulo $3$. Thus the endpoint cannot be $(q,0,s_0)$
and is nonexceptional. If $R\equiv0\pmod 9$ and the endpoint is
$(q,0,s_0)$, then $a$ and $b$ are divisible by $3$ throughout the
reduction. Every rank change is therefore divisible by $9$, and hence $q\equiv R\equiv0\pmod 9$.

It remains to treat type $4$, where $k_S=4$ and $\la_S=2$. The exceptional
reduced rows are
$$
(q,0,s_0)
\qquad\text{and}\qquad
(2q,qB_0,s_0).
$$
Assume first that $R\equiv2,6\pmod 8$ and $a$ is odd. All ranks occurring
in the reduction are even, and both line-bundle twists and the elementary
relative Fourier--Mukai transforms preserve the parity of $a$. Since both
exceptional rows have vanishing $A_0$-coefficient, the endpoint cannot be
exceptional.

Suppose next that $R\equiv4\pmod 8$. If $a$ is odd, the same parity argument
rules out both exceptional rows. If $a\equiv2\pmod 4$, every rank change is
divisible by $4$, so all ranks remain divisible by $4$. Line-bundle twists
and the elementary relative Fourier--Mukai transforms therefore preserve
$a$ modulo $4$. Since both exceptional rows have vanishing
$A_0$-coefficient, neither can occur. Finally, suppose $a\equiv0\pmod 4$
and $b$ is odd. Every rank change is divisible by $4$, and line-bundle
twists and the elementary relative Fourier--Mukai transforms preserve the
parity of $b$. The exceptional endpoints in rank congruence class $4$
modulo $8$ have even $B_0$-coefficient, namely zero or $q$ with $q$ even.
Thus the endpoint is nonexceptional in every subcase.

Finally, assume $R\equiv0\pmod 8$. Suppose first that the endpoint is
$(q,0,s_0)$. Its final divisor coefficients vanish, so parity invariance
shows that $a$ and $b$ are even throughout the reduction. All ranks are
therefore divisible by $4$, and $a$ is preserved modulo $4$. Since its
final value is zero, $4\mid a$ throughout the reduction. Together with
$2\mid b$, this shows that every rank change is divisible by $8$. Thus $q\equiv R\equiv0\pmod 8$.
Suppose instead that the endpoint is $(2q,qB_0,s_0)$. If $q$ were odd,
then the final rank would be congruent to $2$ or $6$ modulo $8$. Since
the final $A_0$-coefficient is zero, parity invariance forces $a$ to be even
throughout the reduction. Every rank change is then divisible by $4$, so
all ranks remain divisible by $4$, a contradiction. Hence $q$ is even. If
$q\equiv2\pmod 4$, then the final rank is congruent to $4$ modulo $8$, the
final $A_0$-coefficient is divisible by $4$, and the final
$B_0$-coefficient $q$ is even. The preceding congruence arguments force
$4\mid a$ and $2\mid b$ throughout the reduction, so every rank change is
divisible by $8$. This contradicts the starting congruence
$R\equiv0\pmod 8$ and the final rank congruence $2q\equiv4\pmod 8$.
Therefore $4\mid q$.

In every exceptional case, the reduced vector satisfies the additional
divisibility condition required by
\cref{lem:reduction-positive-rank-to-zero}.
\end{proof}

\begin{lemma}\label{lem:l-invariant-FM}\label{lem:l-invariant-under-autoeq}\label{lem:ell-square-invariant}
For every Fourier--Mukai autoequivalence $\Phi\in\operatorname{Auteq}D^b(S)$ and every Mukai vector $\mv$, one has $\ell(\Phi_*\mv)=\ell(\mv)$.
\end{lemma}

\begin{proof}
By \cite[Lemma 2.3]{Nue25}, \(\ell(\mv)\) is the divisibility of
\(\pi^*\mv\) in \(H^*_{\mathrm{alg}}(X,\zz)\). Since we work over
\(\cc\), the hypotheses of the canonical-cover lifting theorem apply.
Thus, by \cite[Theorem~2.10]{Toc26}, \(\Phi\) admits a lift $\widetilde{\Phi}\in\operatorname{Auteq}D^b(X)$ such that $\pi^*\circ\Phi\cong\widetilde{\Phi}\circ\pi^*$. Passing to Mukai vectors gives $\pi^*(\Phi_*\mv)=\widetilde{\Phi}_*(\pi^*\mv)$.
The induced map
$$
\widetilde{\Phi}_*\colon H^*_{\mathrm{alg}}(X,\zz)
\longrightarrow H^*_{\mathrm{alg}}(X,\zz)
$$
is an automorphism of the integral Mukai lattice, and lattice
automorphisms preserve divisibility. Hence
$$
\operatorname{div}(\pi^*(\Phi_*\mv))
=
\operatorname{div}(\widetilde{\Phi}_*(\pi^*\mv))
=
\operatorname{div}(\pi^*\mv).
$$
Using Lemma 2.3 of Nuer's paper for \(\ell\) once more gives
\(\ell(\Phi_*\mv)=\ell(\mv)\).
\end{proof}

\cref{lem:starting-vector-exceptional-congruences,lem:reduction-positive-rank-to-zero}
show that many positive-rank Mukai vectors can be reduced to rank zero.
To use \cref{def:admissible-vector}, we must also ensure that the first
Chern class of the resulting rank-zero vector has numerical threshold at least
\(3\). This condition implies very ampleness by \cref{lem:bertini_results}
and is the hypothesis needed later for the rank-zero Betti-number
computations. We now present an explicit sufficient criterion expressed in
terms of the reduced vector \(\mv_0\). Recall that \(c_1(\mw)\) of the
resulting rank-zero vector \(\mw\) in
\cref{lem:reduction-positive-rank-to-zero} is denoted by \(D\), or
\(D(\mv_0)\) when we want to specify dependence on \(\mv_0\).

\begin{proposition}\label{prop:very-ample-rank-zero-reductions}
Let \(S\) be a bielliptic surface, and let
$
\mv_0=(R,aA_0+bB_0,s)
$
be one of the reduced positive-rank Mukai vectors in
\cref{tab:reduced-positive-rank-vectors}. Assume \(\mv_0^2>0\) and that
\(\mv_0\) satisfies the hypotheses of
\cref{lem:reduction-positive-rank-to-zero}. Then the values of
\(D(\mv_0)\) produced by the standard reductions are listed in
\cref{tab:very-ample-reduced-rows}.

In particular, if the condition in the last column of
\cref{tab:very-ample-reduced-rows} holds, then one can choose the standard
reduction so that
$
\nt(D(\mv_0))\geqslant 3.
$
Consequently, \(\mv_0\) is admissible in the sense of
\cref{def:admissible-vector}.
\end{proposition}

\begin{table}[p]
\centering
\scriptsize
\renewcommand{\arraystretch}{1.35}
\resizebox{\textwidth}{!}{%
\begin{tabular}{clll}
\toprule
Type & Reduced vector \(\mv_0\) & Resulting \(D(\mv_0)\) &
Sufficient condition for \(\nt(D(\mv_0))\geqslant 3\) \\
\midrule
1 & \((q,0,s)\) & \(qA_0-sB_0\) & \(q\geqslant 3,\ -s\geqslant 3\) \\
1 & \((2q,qB_0,s)\) & \(-2sA_0+qB_0\) & \(q\geqslant 3,\ -s\geqslant 2\) \\
\midrule
2 & \((q,0,s)\), \(q=4Q\) &
\(-2(Q/h)sA_0+2hB_0\), for some \(h\mid Q\) &
\(\exists h\mid Q:\ h\geqslant 2,\ (Q/h)(-s)\geqslant 2\) \\
2 & \((2q,qA_0,s)\) & \(qA_0-2sB_0\) & \(q\geqslant 3,\ -s\geqslant 2\) \\
2 & \((2q,qB_0,s)\) & \(-2sA_0+qB_0\) & \(q\geqslant 3,\ -s\geqslant 2\) \\
2 & \((2q,qA_0+qB_0,s)\) & \(qA_0+(q-2s)B_0\) &
\(q\geqslant 3,\ q-2s\geqslant 3\) \\
\midrule
3 & \((q,0,s)\) & \(qA_0-sB_0\) & \(q\geqslant 3,\ -s\geqslant 3\) \\
3 & \((4q,qB_0,s)\) & \(-4sA_0+qB_0\) & \(q\geqslant 3\) \\
3 & \((2q,qB_0,s)\) & \(2qA_0-sB_0\) & \(q\geqslant 2,\ -s\geqslant 3\) \\
\midrule
4 & \((q,0,s)\), \(q=8Q\) &
\(-4(Q/h)sA_0+2hB_0\), for some \(h\mid Q\) & \(Q\geqslant 2\) \\
4 & \((2q,qA_0,s)\) & \(qA_0-2sB_0\) & \(q\geqslant 3,\ -s\geqslant 2\) \\
4 & \((4q,qB_0,s)\) & \(-4sA_0+qB_0\) & \(q\geqslant 3\) \\
4 & \((2q,qB_0,s)\), \(q=4Q\) &
\(-4(Q/h)sA_0+2hB_0\), for some \(h\mid Q\) & \(Q\geqslant 2\) \\
4 & \((4q,2qA_0+qB_0,s)\) & \(2qA_0+(q-2s)B_0\) &
\(q\geqslant 2,\ q-2s\geqslant 3\) \\
4 & \((2q,qA_0+qB_0,s)\) & \(qA_0+(q-2s)B_0\) &
\(q\geqslant 3,\ q-2s\geqslant 3\) \\
\midrule
5 & \((q,0,s)\) & \(qA_0-sB_0\) & \(q\geqslant 3,\ -s\geqslant 3\) \\
5 & \((3q,qB_0,s)\) & \(-3sA_0+qB_0\) & \(q\geqslant 3\) \\
\midrule
6 & \((q,0,s)\), \(q=9Q\) &
\(-3(Q/h)sA_0+3hB_0\), for some \(h\mid Q\) & automatic under \(9\mid q\) \\
6 & \((3q,qA_0,s)\) & \(qA_0-3sB_0\) & \(q\geqslant 3\) \\
6 & \((3q,qB_0,s)\) & \(-3sA_0+qB_0\) & \(q\geqslant 3\) \\
6 & \((3q,qA_0+qB_0,s)\) & \(qA_0+(q-3s)B_0\) &
\(q\geqslant 3,\ q-3s\geqslant 3\) \\
\midrule
7 & \((q,0,s)\) & \(qA_0-sB_0\) & \(q\geqslant 3,\ -s\geqslant 3\) \\
7 & \((6q,qB_0,s)\) & \(-6sA_0+qB_0\) & \(q\geqslant 3\) \\
7 & \((3q,qB_0,s)\) & \(3qA_0-sB_0\) & \(-s\geqslant 3\) \\
7 & \((2q,qB_0,s)\) & \(2qA_0-sB_0\) & \(q\geqslant 2,\ -s\geqslant 3\) \\
7 & \((q,bB_0,s)\), \(\frac{1}{3}<\frac{b}{q}<\frac{1}{2}\) &
\(qA_0-sB_0\) & \(q\geqslant 3,\ -s\geqslant 3\) \\
\bottomrule
\end{tabular}%
}
\caption{Standard rank-zero reductions of reduced positive-rank Mukai
vectors and sufficient numerical-threshold conditions.\\
In the exceptional rows, the divisor
\(h\mid Q\) records the choice of the preliminary relative Fourier--Mukai
transform along \(p_A\).}
\label{tab:very-ample-reduced-rows}
\end{table}

\begin{proof}
Write
$
D=\alpha A_0+\beta B_0.
$
By \cref{def:admissible-vector}, it is enough to ensure
$
\alpha\geqslant 3,
\,
\beta\geqslant 3,
$
that is,
$
\nt(D)\geqslant 3.
$
This condition implies very ampleness by \cref{lem:bertini_results} and is
the condition needed in \cref{thm:b1-b2-moduli}.

Recall that we use the notations
$
k_S=\ord(K_S),
\,
\la_S=\frac{|G|}{k_S}.
$
The fiber degrees with respect to the two elliptic fibrations are
$
d_A=\la_Sa,
\,
d_B=k_Sb.
$
First consider the nonexceptional rows. If a row is killed by the
normalized transform along $p_A$, so that $R=\la_Sa$, then the inverse
transform has numerical action
$$
\Phi_A^{-1}(R,aA_0+bB_0,s)
=
(R-\la_Sa,aA_0+(b-\la_Ss)B_0,s).
$$
Thus the resulting first Chern class is
$
D=aA_0+(b-\la_Ss)B_0.
$
If a row is killed by the normalized transform along $p_B$, so that
$R=k_Sb$, then
$$
\Psi_B^{-1}(R,aA_0+bB_0,s)
=
(R-k_Sb,(a-k_Ss)A_0+bB_0,s),
$$
and hence
$
D=(a-k_Ss)A_0+bB_0.
$

Finally, suppose a row is killed by the $p_A$-relative transform with
matrix
$$
\begin{pmatrix}
0&1\\
-1&0
\end{pmatrix}.
$$
This case occurs only when $\la_S=1$ and $a=0$. The transform sends
$(R,d_A)=(R,0)$ to a rank-zero pair with fiber degree $R$, up to an
overall sign. After composing with a shift if necessary, we may take the
effective sign, so the $A_0$-coefficient of $D$ is $R$. Since
Fourier--Mukai transforms preserve the Mukai pairing,
$
D^2=\mv_0^2.
$
For all rows of this type, the original divisor part is a multiple of
$B_0$, and therefore
$
\mv_0^2=-2Rs.
$
On the other hand,
$
(RA_0+\beta B_0)^2=2R\beta.
$
It follows that $\beta=-s$, and hence
$
D=RA_0-sB_0.
$
These three computations give all nonexceptional rows in
\cref{tab:very-ample-reduced-rows}.

It remains to discuss the exceptional rows. First let
$
\mv_0=(q,0,s)
$
be one of the exceptional rows in type $2$, $4$, or $6$. Thus
$
q=k_S\la_SQ
$
for some positive integer $Q$. Since $\mv_0$ is primitive,
$\gcd(q,s)=1$, and in particular $\gcd(s,k_S)=1$. Let $h\mid Q$.
Applying the $h$-fold composition of the inverse normalized transform
along $p_A$ gives
$$
\Phi_A^{-h}(R,aA_0+bB_0,s)
=
(R-\la_Sha,aA_0+(b-\la_Shs)B_0,s).
$$
It sends $\mv_0$ to
$
(q,-\la_ShsB_0,s).
$
For the fibration $p_B$, this vector has
$$
R=k_S\la_SQ,
\qquad
d_B=-k_S\la_Shs.
$$
Because $h\mid Q$ and $\gcd(Q,s)=1$, one has
$$
g=\gcd(R,d_B)=k_S\la_Sh,
\qquad
\frac{d_B}{g}=-s.
$$
Since $\gcd(s,k_S)=1$,
\cref{lem:admissible-rank-zero-transform} applies and gives a
$p_B$-relative Fourier--Mukai transform to rank zero. The resulting
$B_0$-coefficient has absolute value $g/k_S=\la_Sh$. Taking the effective
sign, write
$
D=\alpha A_0+\la_ShB_0.
$
Preservation of the Mukai square gives
$$
2\alpha\la_Sh=\mv_0^2=-2k_S\la_SQs,
$$
so
$
D=-k_S(Q/h)sA_0+\la_ShB_0.
$
This gives the exceptional rows $(q,0,s)$ in types $2$, $4$, and $6$.

In type $4$ there is one more exceptional row,
$$
\mv_0=(2q,qB_0,s),
\qquad
q=4Q.
$$
Here $k_S=4$ and $\la_S=2$. Again $\gcd(q,s)=1$, so $s$ is odd. Let
$h\mid Q$. Applying $\Phi_A^{-h}$ gives
$
(2q,(q-2hs)B_0,s).
$
For $p_B$, this vector has
$$
R=2q=8Q,
\qquad
d_B=4(q-2hs)=8(2Q-hs).
$$
Since $h\mid Q$ and $\gcd(Q,s)=1$,
$$
\gcd(R,d_B)=8h,
\qquad
\frac{d_B}{8h}=2(Q/h)-s.
$$
The latter integer is odd and hence coprime to $k_S=4$, so
\cref{lem:admissible-rank-zero-transform} gives a $p_B$-relative
transform to rank zero. The resulting $B_0$-coefficient has absolute
value $8h/4=2h$. Writing
$
D=\alpha A_0+2hB_0
$
and using preservation of the Mukai square, we obtain
$$
2\alpha(2h)=\mv_0^2=-4qs=-16Qs,
$$
and therefore
$
D=-4(Q/h)sA_0+2hB_0.
$

The conditions in the last column of
\cref{tab:very-ample-reduced-rows} are obtained by requiring, for at least
one of the choices described above, both coefficients of $D(\mv_0)$ to be
at least $3$. For the two exceptional type $4$ rows and the exceptional
type $6$ row, take $h=Q$; the condition
$\mv_0^2>0$ implies $s<0$, so the bounds displayed in the table follow.
For the exceptional type $2$ row, choose a divisor $h\mid Q$ satisfying
the condition stated in the table. Thus \(\nt(D(\mv_0))\geqslant 3\), so
\(\mv_0\) is admissible by \cref{def:admissible-vector}. Finally, suppose $\mv$ satisfies
\cref{lem:starting-vector-exceptional-congruences}, and let $\mv_0$ be
its reduced vector. If $\mv_0$ satisfies the corresponding condition in
the table, composing the reduction from $\mv$ to $\mv_0$ with the
rank-zero reduction of $\mv_0$ shows that $\mv$ is admissible.
\end{proof}

\begin{remark}
The conditions in
\cref{lem:reduction-positive-rank-to-zero,tab:very-ample-reduced-rows}
are sufficient for the standard reductions used here; no necessity
assertion is intended.
\end{remark}
\FloatBarrier

\subsection{Transition from rank zero to rank one}\label{subsec:rank-one-transition}
We characterize when a rank-zero Mukai vector lies in the Fourier--Mukai
orbit of a rank-one vector.

\begin{theorem}\label{thm:rank-zero-hilbert-model}
Let $(S,H)$ be a polarized bielliptic surface, let $\mv=(0,aA_0+bB_0,\chi)$ be a primitive rank-zero Mukai vector with \(\mv^2>0\), and let \(H\) be
\(\mv\)-generic. If
$\la_S=1$ and $\ell(\mv)=1$, then there is a rank-one Mukai vector
$\mv'=(1,D,s')$ with $(\mv')^2=\mv^2$
such that
$$
M_{H,S}(\mv)\dashrightarrow M_{H',S}(\mv')
$$
is birational for a suitable $\mv^\prime$-generic polarization \(H'\). Moreover,
$$
M_{H',S}(\mv')
\cong
\Pic^D(S)\times\Hilb^{\mv^2/2}(S),
$$
and hence, noncanonically,
$$
M_{H,S}(\mv)
\dashrightarrow
\Pic^0(S)\times\Hilb^{\mv^2/2}(S).
$$

Conversely, if such a birational map is obtained by first applying a
Fourier--Mukai autoequivalence sending \(\mv\) to rank one and then crossing
Bridgeland walls, then $\la_S=1$ and $\ell(\mv)=1$.
\end{theorem}

\begin{proposition}\label{prop:rank-zero-rank-one}
Let \(S\) be a bielliptic surface, and let $\mv=(0,aA_0+bB_0,\chi)$ be a primitive rank-zero Mukai vector with \(\mv^2>0\). Then \(\mv\) lies
in the Fourier--Mukai orbit of a rank-one Mukai vector if
and only if $\la_S=1$ and $\ell(\mv)=1$.
Equivalently, this can happen only for bielliptic surfaces of type
\(1,3,5\), or \(7\), and in those cases the condition is $\gcd(a,k_Sb,k_S\chi)=1$.
\end{proposition}

\begin{proof}
In the proof of this proposition, we rely on Tochitani's description of
\(\operatorname{Auteq}D^b(S)\): it is generated by standard
autoequivalences together with relative Fourier--Mukai transforms along
\(p_A\) and \(p_B\); see \cite[Main Theorem~1]{Toc26}.
Suppose first that \(\mv\) lies in the Fourier--Mukai autoequivalence orbit
of a rank-one Mukai vector. By \cref{lem:l-invariant-FM}, \(\ell\) is
invariant under Fourier--Mukai autoequivalences, and every rank-one vector
has \(\ell=1\). Hence \(\ell(\mv)=1\).

It remains to show that \(\la_S=1\). Assume, for a contradiction, that
\(\la_S>1\). By the classification this is precisely the case of types
\(2,4,6\), and in these cases \(\la_S\mid k_S\). For a Mukai vector $(R,aA_0+bB_0,s)$,
the fiber degrees with respect to \(p_A\) and \(p_B\) are
$$
d_A=\la_Sa,
\qquad
d_B=k_Sb,
$$
so both are divisible by \(\la_S\). A relative Fourier--Mukai transform
along either elliptic fibration acts on the pair \((R,d)\) by an
admissible matrix in \(\mathrm{SL}_2(\zz)\), see
\cite[Theorem~2.14]{Toc26}. In particular the new rank has the form $R'=cR+\alpha d$ for some \(c,\alpha\in\zz\). Thus if \(\la_S\mid R\), then
\(\la_S\mid R'\). Shifts change the sign of the rank, automorphisms preserve
the rank, and tensoring by a line bundle preserves the rank. Hence standard
autoequivalences preserve divisibility of the rank by \(\la_S\). By
\cite[Main Theorem~1]{Toc26}, these standard autoequivalences, together with
the relative Fourier--Mukai transforms along \(p_A\) and \(p_B\), generate
the autoequivalence group. Since \(\mv\) has rank zero, every
vector in its Fourier--Mukai orbit has rank divisible by \(\la_S\). If
\(\la_S>1\), such an orbit cannot contain a rank-one vector.

Conversely, assume that \(\la_S=1\) and \(\ell(\mv)=1\). Then $\ell(\mv)=\gcd(a,k_Sb,k_S\chi)=1$.
Since \(\mv^2>0\), both \(a\) and \(b\) are nonzero.

We first tensor by a line bundle to arrange a useful value of \(\chi\).
Tensoring \(\mv\) by a line bundle of numerical class \(xA_0+yB_0\)
sends it to
$$
(0,aA_0+bB_0,\chi+xb+ya).
$$
Let \(h=\gcd(a,b)\). Hence we may replace \(\chi\) by any integer
\(\chi'\equiv\chi\pmod h\). We claim that \(\chi'\) can be chosen so that $\gcd(a,k_S(b-\chi'))=1$.
Indeed, since \(\gcd(a,k_Sb,k_S\chi)=1\), every prime \(p\mid a\) is
prime to \(k_S\). Thus it is enough to arrange \(p\nmid b-\chi'\) for
every prime \(p\mid a\). If \(p\mid h\), then \(p\mid b\) and
\(\chi'\equiv\chi\pmod p\), while the same gcd condition gives
\(p\nmid\chi\). Hence \(p\nmid b-\chi'\). If \(p\nmid h\), then
\(\chi'=\chi+hm\) varies freely modulo \(p\), so we choose \(m\) with
\(\chi'\not\equiv b\pmod p\). The Chinese remainder theorem gives a
single choice of \(m\) working for all primes \(p\mid a\).

After this line-bundle twist, it is enough to consider $\mv'=(0,aA_0+bB_0,\chi')$ with $\gcd(a,k_S(b-\chi'))=1$.
Since \(\la_S=1\), the inverse of the normalized relative Fourier--Mukai
transform along \(p_A\) has numerical action
\cite[Proposition 10.5]{Nue25}
$$
\Phi_A^{-1}(R,aA_0+bB_0,s)
=
(R-a,\ aA_0+(b-s)B_0,\ s).
$$
Therefore
$$
\Phi_A^{-1}(\mv')
=
(-a,\ aA_0+(b-\chi')B_0,\chi').
$$
For \(p_B\), this vector has rank \(R=-a\) and fiber degree $d_B=k_S(b-\chi')$.
By construction, \(\gcd(R,d_B)=1\).

We now apply a relative Fourier--Mukai transform along \(p_B\) sending
this vector to rank one. Since \(\gcd(R,d_B)=1\), choose
\(c_0,\alpha_0\in\zz\) such that
$$
c_0R+\alpha_0d_B=1.
$$
All solutions are of the form
$$
c=c_0+td_B,
\qquad
\alpha=\alpha_0-tR,
\qquad
t\in\zz.
$$
Since \(R\neq 0\), we may choose \(t\) so that \(\alpha>0\).
As \(k_S\mid d_B\), the integer \(c\) is automatically coprime to
\(k_S\). Hence \(c\) and \(\alpha k_S\) are coprime, so we can complete
\((c,\alpha)\) to an admissible matrix
$$
\begin{pmatrix}
c&\alpha\\
-k_S\gamma&\beta
\end{pmatrix}
\in\mathrm{SL}_2(\zz).
$$
By \cite[Theorem~2.14]{Toc26}, the corresponding relative Fourier--Mukai
transform has new rank $cR+\alpha d_B=1$.
Thus \(\mv\) lies in the Fourier--Mukai orbit of a rank-one Mukai vector.
\end{proof}

\begin{proof}[Proof of \cref{thm:rank-zero-hilbert-model}]
By \cref{prop:rank-zero-rank-one}, there is a Fourier--Mukai autoequivalence
\(\Phi\) such that \(\mv'=\Phi_*(\mv)\) has rank one. Since
Fourier--Mukai autoequivalences preserve the Mukai pairing,
\((\mv')^2=\mv^2\). The equivalence
\(\Phi\) identifies the full moduli spaces in the corresponding Bridgeland
chambers, and Nuer's wall-crossing theorem
\cite[Theorem 1.3]{Nue25} gives a birational map
$$
M_{H,S}(\mv)\dashrightarrow M_{H',S}(\mv')
$$
after moving to the respective Gieseker chambers.

Write \(\mv'=(1,D,s')\), where $D=c_1(\mv')\in H^2(S,\zz)$ is the actual first Chern class; its numerical class is the one determined by the Fourier--Mukai calculations above. Every rank-one torsion-free sheaf of this Mukai
vector is \(I_Z\otimes L\), where \(L\in\Pic^D(S)\) and $\operatorname{length}(Z)=\frac{(\mv')^2}{2}=\frac{\mv^2}{2}$.
Rank-one torsion-free sheaves are Gieseker stable with respect to every
polarization, so no additional stability condition cuts out a proper
sublocus.
Therefore
$$
M_{H',S}(\mv')
\cong
\Pic^D(S)\times\Hilb^{\mv^2/2}(S).
$$
In particular, $M_{H',S}(\mv')$ is irreducible. Since $M_{H,S}(\mv)$ is
birational to $M_{H',S}(\mv')$, it is irreducible as well. Finally,
\(\Pic^D(S)\) is a torsor under \(\Pic^0(S)\), so choosing a point of
\(\Pic^D(S)\) gives the stated noncanonical identification.

For the converse, the underlying Mukai vector lies in the
Fourier--Mukai orbit of a rank-one vector, so
\cref{prop:rank-zero-rank-one} gives \(\la_S=1\) and \(\ell(\mv)=1\).
\end{proof}

\begin{remark}
Under the hypotheses of \cref{thm:rank-zero-hilbert-model}, the full moduli
space $M_{H,S}(\mv)$ is irreducible. Consequently, whenever the distinguished
rank-zero component $M^\circ_{H,S}(\mv)$ is defined as in
\cref{subsec:moduli-spaces}, in particular in the admissible rank-zero cases,
one has
$$
M^\circ_{H,S}(\mv)=M_{H,S}(\mv).
$$
Similarly, the distinguished component selected on the rank-one side is the
full moduli space $M_{H',S}(\mv')$.
\end{remark}

\section{Betti numbers and fundamental group of the moduli spaces}
\label{sec:topology-moduli}
\Cref{subsec:first-betti-fixed-det,subsec:fundamental-group-fixed-det,subsec:albanese-fibers-rank-zero,subsec:betti-rank-zero}
treat primitive rank-zero Mukai vectors
$
\mv=(0,c_1(L),\chi)
$
with $\nt(L)\geqslant 3$ and $H$ $\mv$-generic. Throughout these subsections,
$M$ and $M(L)$ have the meanings fixed in \cref{subsec:moduli-spaces}.
\Cref{sec:ext-to-pos-rank} extends the Betti-number computations to arbitrary primitive
admissible Mukai vectors, while \cref{subsec:hodge-cy-albanese-fibers} studies the Hodge structure of the Albanese fibers and constructs
quasi-\'etale Calabi--Yau covers in the rank-one-orbit case.
\subsection{The first Betti number of \texorpdfstring{$M(L)$}{M(L)}}\label{subsec:first-betti-fixed-det}

\begin{theorem}\label{thm:first-betti-fixed-det}
Let $(S,H)$ be a polarized bielliptic surface and let
$
\mv=(0,c_1(L),\chi)
$
be a primitive Mukai vector with $\nt(c_1(L)) \geqslant 3$. Assume that $H$ is $\mv$-generic. Then
$$
b_1(M(L))=2.
$$
\end{theorem}

We prove the upper and lower bounds separately.
\begin{lemma}\label{lem:b1-upper-bound}
Under the assumptions of \cref{thm:first-betti-fixed-det}, one has
$$
b_1(M(L))\leqslant 2.
$$
\end{lemma}
\begin{proof}
Let $h:M(L)\to |L|$ be the support morphism, let $\Delta\subset |L|$ be
the discriminant, and set $U\coloneqq |L|\setminus\Delta$ and
$M(L)_U\coloneqq h^{-1}(U)$.
Since $M(L)_U\subset M(L)$ is a dense Zariski open subset, every loop in $M(L)$ can be moved away from $M(L)\setminus M(L)_U$. Hence the natural map
$$
\pi_1(M(L)_U)\longrightarrow\pi_1(M(L))
$$
is surjective. Therefore
$$
H_1(M(L)_U,\qq)\longrightarrow H_1(M(L),\qq)
$$
is also surjective, and it is enough to show that
$$
\dim_{\qq}H^1(M(L)_U,\qq)\leqslant 2.
$$

Over $U$, the support morphism
$
h_U \colon M(L)_U \longrightarrow U
$
is the relative Picard variety whose fiber over a smooth curve $D\in U$ is $\Pic^d(D)$. The Leray spectral sequence for $h_U$ gives
$$
\dim_{\qq}H^1(M(L)_U,\qq)
\leqslant
\dim_{\qq}H^1(U,\qq)
+
\dim_{\qq}H^0\bigl(U,R^1h_{U,*}\qq\bigr).
$$
Since $\nt(L)\geqslant 3$, \cref{lem:bertini_results} shows that $\Delta$ is irreducible. The standard computation of the first cohomology of the complement of an irreducible hypersurface in projective space gives $H^1(U,\qq)=0$. Let $q \colon \uc \longrightarrow |L|$ be the universal curve, and write $q_U \colon \uc_U \longrightarrow U$ for its restriction over $U$. Since $h_U$ is the family of Picard varieties of the smooth curves parametrized by $U$, the local system $R^1h_{U,*}\qq$ is dual to $R^1q_{U,*}\qq$. In particular,
$$
\dim_{\qq}H^0\bigl(U,R^1h_{U,*}\qq\bigr)
=
\dim_{\qq}H^0\bigl(U,R^1q_{U,*}\qq\bigr).
$$
By the global invariant cycle theorem applied to $q_U \colon \uc_U \to U$, for any $D \in U$ we have
$$
H^0\bigl(U,R^1q_{U,*}\qq\bigr)
=
\im\left[
H^1(\uc,\qq)
\longrightarrow
H^1(D,\qq)
\right].
$$
Because $|L|$ is base point free, the universal curve $\uc$ is a projective bundle over $S$: the fiber over a point $x \in S$ is the projective space of divisors in $|L|$ passing through $x$. Hence $H^1(\uc,\qq)\cong H^1(S,\qq)$. Since $S$ is bielliptic, $b_1(S)=2$. Therefore $\dim_{\qq}H^0\bigl(U,R^1q_{U,*}\qq\bigr)\leqslant 2$.
Combining the inequalities above gives
$$
\dim_{\qq}H^1(M(L)_U,\qq)\leqslant 2.
$$
Since $H_1(M(L)_U,\qq)\to H_1(M(L),\qq)$ is surjective, we obtain
$
b_1(M(L))\leqslant 2.
$
\end{proof}

Now we compute the lower bound for $b_1(M(L))$, but before doing this, we need the following technical lemma.

\begin{lemma}\label{lem:norm-map-albanese}\label{lem:albanese-norm-map}\label{lem:alpha-construction}
Under the assumptions of \cref{thm:first-betti-fixed-det}, let
$a\colon S\longrightarrow E\coloneqq\Alb(S)$ be the Albanese morphism and set
$m\coloneqq L\cdot B$, where $B$ is the numerical class of a fiber of $a$.
Let $h\colon M(L)\longrightarrow |L|$ be the support morphism, and let
$U\subset |L|$ be the open subset parametrizing smooth curves.
Then $m>0$, and there exists a morphism
$$
\alpha\colon M(L)\longrightarrow \Pic^0(E).
$$
For every smooth curve $D\in U$, the restriction of $\alpha$ to $h^{-1}(D)$ agrees, up to a fixed translation of the target, with the norm map
$$
\Nm_{a_D}\colon \Pic^d(D)\longrightarrow \Pic^0(E),
\qquad
a_D\coloneqq\restrict{a}{D}\colon D\longrightarrow E.
$$
Moreover, for every $N\in\Pic^0(E)$ and every $F\in M(L)$ one has
$$
\alpha(F\otimes a^*N)=\alpha(F)+[m]N.
$$
Consequently:
\begin{enumerate}
\item For every $D\in |L|$ such that $h^{-1}(D)$ is nonempty, the restriction
$$
\restrict{\alpha}{h^{-1}(D)}\colon h^{-1}(D)\longrightarrow\Pic^0(E)
$$
is surjective;
\item The restriction of $\alpha$ to every irreducible component of a nonempty support fiber is surjective;
\item If $D$ is smooth, then
$
\Nm_{a_D}\colon\Pic^0(D)\longrightarrow\Pic^0(E)
$
is surjective.
\end{enumerate}
\end{lemma}
\begin{proof}
By \cref{cor:rank-zero-terminal-lci}, $M(L)$ is normal and projective.
Over $U$, the support morphism identifies $M(L)_U$ with the relative Picard variety $\Pic^d(\uc_U/U)$, where $d=\chi+g(D)-1$. The Albanese morphism induces $a_{\uc_U}\colon \uc_U\longrightarrow E\times U$.
The relative norm map gives
$$
\Nm_{a_{\uc_U}}\colon
\Pic^d(\uc_U/U)\longrightarrow \Pic^e(E)\times U
$$
for some integer $e$. Choose a base point of $\Pic^e(E)$ and use it to identify $\Pic^e(E)$ with $\Pic^0(E)$. Composing the relative norm map with the projection to the first factor defines
$$
\alpha_U\colon M(L)_U\longrightarrow \Pic^0(E).
$$
Changing the base point only translates $\alpha_U$ and has no effect on the assertions below.
By \cref{cor:rank-zero-terminal-lci}, $M(L)$ has terminal, hence rational,
singularities. The morphism $\alpha_U$ defines a rational map from $M(L)$ to
the abelian variety $\Pic^0(E)$. By \cite[Lemma~2.24(2)]{Zha09}, this rational
map extends to a regular morphism
$$
\alpha\colon M(L)\longrightarrow\Pic^0(E).
$$

For $N\in\Pic^0(E)$, set
$
T_N(F)\coloneqq F\otimes a^*N.
$
Tensoring by $a^*N$ preserves the Fitting support, Mukai vector, and
$H$-semistability. Moreover, the determinant identity
$$
\det(F\otimes Q)\simeq\det(F)\otimes Q^{\rank(F)}
$$
and $\rank(F)=0$ give $\det(F\otimes a^*N)\simeq\det(F)$. Thus a
Poincar\'e line bundle defines an algebraic $\Pic^0(E)$-action on
$M(L)$ preserving every support fiber. Since $\Pic^0(E)$ is connected, it
preserves every irreducible component.

Let $D\in U$, let $i_D\colon D\hookrightarrow S$, and write $F=i_{D,*}A$ with $A\in\Pic^d(D)$. Then
$$
F\otimes a^*N\simeq i_{D,*}(A\otimes a_D^*N).
$$
The standard norm identity gives
$$
\Nm_{a_D}(A\otimes a_D^*N)\simeq\Nm_{a_D}(A)\otimes N^{\deg(a_D)}.
$$
The degree is constant in the linear system and satisfies
$$
\deg(a_D)=D\cdot B=L\cdot B=m.
$$
Since $\nt(L)\geqslant 3$, the line bundle $L$ is ample by \cref{lem:num_div}, and hence $m>0$. After the fixed translation identifying $\Pic^e(E)$ with $\Pic^0(E)$, the norm identity becomes
$$
\alpha_U(F\otimes a^*N)=\alpha_U(F)+[m]N.
$$

Consider the two morphisms
$$
\Phi_1,\Phi_2\colon\Pic^0(E)\times M(L)\longrightarrow\Pic^0(E)
$$
defined by
$$
\Phi_1(N,F)\coloneqq\alpha(F\otimes a^*N),
\qquad
\Phi_2(N,F)\coloneqq\alpha(F)+[m]N.
$$
They agree on the dense open subset $\Pic^0(E)\times M(L)_U$. Since the source is reduced and the target is separated, they agree everywhere. Therefore
$
\alpha(F\otimes a^*N)=\alpha(F)+[m]N
$
for every $N\in\Pic^0(E)$ and every $F\in M(L)$.

Fix a nonempty support fiber $h^{-1}(D)$ and a point $F\in h^{-1}(D)$. Its $\Pic^0(E)$-orbit remains inside $h^{-1}(D)$, and its image under $\alpha$ is
$
\alpha(F)+[m]\Pic^0(E).
$
Since $m>0$, multiplication by $m$ is an isogeny of the elliptic curve $\Pic^0(E)$ and is therefore surjective. Thus $\restrict{\alpha}{h^{-1}(D)}$ is surjective. If $Z\subset h^{-1}(D)$ is an irreducible component and $F\in Z$, the connected tensor action preserves $Z$, so the same orbit argument proves that $\restrict{\alpha}{Z}$ is surjective.

For smooth $D$, the identity $\Nm_{a_D}\circ a_D^*=[m]$ proves that $\Nm_{a_D}\colon\Pic^0(D)\to\Pic^0(E)$ is surjective.
\end{proof}

\begin{lemma}\label{lem:b1-lower-bound}
Under the assumptions of \cref{thm:first-betti-fixed-det}, one has
$$
b_1(M(L))\geqslant 2.
$$
\end{lemma}
\begin{proof}
Let $a\colon S\longrightarrow E\coloneqq\Alb(S)$ be the Albanese morphism of $S$. By \cref{lem:albanese-norm-map}, there is a surjective morphism $\alpha\colon M(L)\longrightarrow \Pic^0(E)$.
A surjective morphism from a projective variety to a smooth projective curve induces an injection on first rational cohomology. Indeed, after Stein factorization one reduces to a morphism with connected fibers, where the Leray spectral sequence gives injectivity, and to a finite morphism, where injectivity follows from the trace map. Therefore
$$
H^1(\Pic^0(E),\qq)\hookrightarrow H^1(M(L),\qq).
$$
Since $\Pic^0(E)$ is an elliptic curve, $\dim_{\qq}H^1(\Pic^0(E),\qq)=2$.
Hence $b_1(M(L))\geqslant 2$.
\end{proof}

\begin{proof}[Proof of \cref{thm:first-betti-fixed-det}]
By \cref{lem:b1-upper-bound}, we have $b_1(M(L))\leqslant 2$. By \cref{lem:b1-lower-bound}, we have $b_1(M(L))\geqslant 2$. Hence $b_1(M(L))=2$.
\end{proof}

\subsection{Properties of the fundamental group of \texorpdfstring{$M(L)$}{M(L)}}\label{subsec:fundamental-group-fixed-det}

\begin{theorem}\label{thm:pi1}\label{thm:fundamental-group-fixed-det}
Let $(S,H)$ be a polarized bielliptic surface and let
$
\mv=(0,c_1(L),\chi)
$
be a primitive rank-zero Mukai vector. Let $H$ be $\mv$-generic, and assume that
$\nt(L)\geqslant 3$. Then:
\begin{enumerate}
\item There is a surjective homomorphism
$$
\pi_1(S)\twoheadrightarrow \pi_1(M(L)).
$$
\item One has
$
\rank\bigl(\pi_1(M(L))_{ab}\bigr)=2.
$
\end{enumerate}
\end{theorem}
In order to prove it, we establish some technical results first.
We use the following form of Leibman's setting for fundamental group arguments, following \cite[Section 3.2]{Sac19} and \cite[1.11]{Lei93}.

\begin{lemma}\label{lem:leibman-exact-sequence}
Let
$
p:E\to B
$
be a surjective morphism of smooth connected varieties, let \(U\subset B\) be a Zariski open subset with complement \(W \coloneqq B\setminus U\)
of codimension at least one. Assume that \(p\) has a section and that
$$
p_U:E_U\coloneqq p^{-1}(U)\to U
$$
is a smooth topologically locally trivial fibration with connected fiber \(F\). Then there is an exact sequence
$$
1\to R\to\pi_1(F)\to\pi_1(E)\to\pi_1(B)\to 1,
$$
where \(R\) is the subgroup generated by the lifted local monodromy relations around the components of \(W\).
\end{lemma}
\begin{proof}
This is the exact sequence obtained from Leibman's diagram and the description of generators of \(\ker(\pi_1(U)\to\pi_1(B))\) by small loops around the components of \(W\); see \cite[Section 3.2]{Sac19} and \cite[1.11]{Lei93}.
\end{proof}

Let $|C|$ be a linear system of genus $g$ on a bielliptic surface $S$ such that $\nt(C) \geqslant 3$, and let $\Pi \subset |C|$ be a general pencil. It has
$$
C^2 - C \cdot K_S = C^2 = 2g - 2
$$
base points. The universal family $\uc_{\Pi} \subset \Pi \times S$ of curves is smooth, and the map
$
p \colon \uc_{\Pi} \longrightarrow \Pi
$
has a section denoted by $s$. Also note that
$
\pi_1(\uc_{\Pi}) \cong \pi_1(S).
$
Then $\uc_{\Pi} \to \Pi$ satisfies the assumptions of Leibman and we can consider the corresponding sequence.

\begin{lemma}\label{lem:R_C}
Let $C_t$ be a smooth fiber of $p \colon \uc_{\Pi} \to \Pi$. There is an exact sequence
$$
1 \longrightarrow R_C \longrightarrow \pi_1(C_t)
\longrightarrow \pi_1(S) \longrightarrow 1.
$$
Equivalently,
$
R_C =
\ker\bigl[\pi_1(C_t)\longrightarrow\pi_1(S)\bigr].
$
\end{lemma}
\begin{proof}
	The morphism \(p:\uc_{\Pi}\to\Pi\) satisfies the assumptions of \cref{lem:leibman-exact-sequence}: it has a section given by an exceptional divisor of
	$
	\uc_{\Pi}\simeq\operatorname{Bl}_Z S,
	$
	and over the smooth locus of the pencil it is a smooth proper topologically locally trivial fibration with connected fiber. Thus
	$$
	1 \longrightarrow R_C \longrightarrow \pi_1(C_t)
	\longrightarrow \pi_1(\uc_{\Pi})
	\longrightarrow \pi_1(\Pi)
	\longrightarrow 1.
	$$
	Since \(\Pi\simeq\ps^1\), we have \(\pi_1(\Pi)=1\). Since \(\uc_{\Pi}\simeq\operatorname{Bl}_Z S\), blowing up finitely many points does not change the fundamental group, so
	$
	\pi_1(\uc_{\Pi})\cong\pi_1(S),
	$
and the asserted exact sequence follows.
\end{proof}

\begin{remark}\label{rem:canonical-cover-not-universal}
In the Enriques case treated by Saccà, the double cover by the K3 surface is the universal cover. Thus the kernel in the analogue of \cref{lem:R_C} can be described as the image of the fundamental group of the inverse image of $C_t$ in the K3 cover. This description does not carry over directly to a bielliptic surface. Indeed, the canonical cover
$$
\pi \colon X \longrightarrow S
$$
is an abelian surface, not the universal cover of $S$. The universal cover of $S$ is $\cc^2$. Consequently, the subgroup
$$
\pi_*\pi_1\bigl(\pi^{-1}(C_t)\bigr)\subset \pi_1(C_t)
$$
is not, in general, equal to
$$
\ker\bigl[\pi_1(C_t)\to\pi_1(S)\bigr].
$$
Thus, in the bielliptic case, one should not identify $R_C$ with the image of the fundamental group of the inverse image of $C_t$ under the canonical cover.
\end{remark}

For the remainder of this subsection, retain the assumptions of
\cref{thm:fundamental-group-fixed-det}. Fix a general pencil
$\Pi\subset |L|$, let
$$
p:\uc_\Pi\longrightarrow\Pi
$$
be its family of curves, and set
$$
J_\Pi\coloneqq M(L)\times_{|L|}\Pi.
$$
Let $U\subset |L|$ be the open subset parametrizing smooth curves and set
$U_\Pi\coloneqq U\cap\Pi$.
By \cref{cor:rank-zero-terminal-lci}, $M(L)$ is a normal l.c.i. variety
whose singular locus has codimension at least three. Successive applications
of Bertini's theorem to the base-point-free linear system induced by the
support morphism show that, for a general pencil $\Pi$, the variety $J_\Pi$
is l.c.i. and regular in codimension one, hence normal.

Put $g\coloneqq g(L)$ and $d\coloneqq\chi+g-1$. Since
$$
\dim |L|=g-2,
\qquad
\dim M(L)=2g-2,
$$
and $J_\Pi$ is cut out in $M(L)$ by $g-3$ general members of the
base-point-free linear system induced by the support morphism, it is pure of
dimension $g+1$. Over $U_\Pi$, the morphism
$$
(J_\Pi)_{U_\Pi}\longrightarrow U_\Pi
$$
is a relative Picard torsor. Its base and fibers are connected, so its smooth
total space is irreducible. Any other irreducible component of $J_\Pi$ would
be supported over one of the finitely many points of
$\Pi\setminus U_\Pi$. By \cref{ass:comp_jacobian}, every such support fiber
has dimension at most $g$, whereas every irreducible component of $J_\Pi$
has dimension $g+1$. Therefore
$
J_\Pi=\overline{(J_\Pi)_{U_\Pi}}
$
is irreducible.

\begin{lemma}\label{lem:pencil-curve-to-jacobian-surjection}
The compactified Abel--Jacobi morphism induces a surjection
$$
\pi_1(\uc_\Pi)\twoheadrightarrow \pi_1(J_\Pi).
$$
Consequently, using $\pi_1(\uc_\Pi)\cong \pi_1(S)$, there is a surjection
$$
\pi_1(S)\twoheadrightarrow \pi_1(J_\Pi).
$$
\end{lemma}
\begin{proof}
Since $\nt(L)\geqslant 3$, \cref{lem:bertini_results} shows that $|L|$ is very ample and that the singular fibers of a general pencil are irreducible one-nodal curves. Fix one of the exceptional sections $s$ of the pencil. The exceptional divisors of
$
\uc_\Pi\cong\operatorname{Bl}_Z S
$
give such sections.
Over $U_\Pi$, the translated Abel--Jacobi map gives a morphism
$$
A_{U_\Pi}\colon (\uc_\Pi)_{U_\Pi}\longrightarrow (J_\Pi)_{U_\Pi}.
$$
On each smooth fiber $D_t$, it is given by
$$
x\longmapsto
\mathcal O_{D_t}\bigl(x+(d-1)s(t)\bigr)
\in\Pic^d(D_t).
$$
Hence it induces the abelianization map
$$
\pi_1(D_t)\twoheadrightarrow H_1(D_t,\zz)
=
\pi_1(\Pic^d(D_t)).
$$
Using the homotopy exact sequences of the two smooth fibrations over $U_\Pi$, it follows that
$$
\pi_1((\uc_\Pi)_{U_\Pi})\twoheadrightarrow
\pi_1((J_\Pi)_{U_\Pi}).
$$

The Abel--Jacobi map extends over irreducible one-nodal fibers by the standard compactified Abel--Jacobi construction, as in \cite[Lemmas 3.5--3.7 and Proposition 3.8]{Sac19}. Thus it defines a morphism
$$
A\colon \uc_\Pi\longrightarrow J_\Pi.
$$
The map induced by $A$ on fundamental groups is surjective because the previous surjection over $U_\Pi$ factors through $\pi_1(\uc_\Pi)$ and the inclusion
$
(J_\Pi)_{U_\Pi}\subset J_\Pi
$
induces a surjection on fundamental groups by
\cite[Proposition~2.10]{Kol95}. Finally,
$\uc_\Pi\cong\operatorname{Bl}_Z S$, so
$\pi_1(\uc_\Pi)\cong\pi_1(S)$.
\end{proof}

\begin{lemma}\label{lem:pencil-to-global-surjection}
The natural inclusion
$
J_\Pi\hookrightarrow M(L)
$
induces a surjection
$$
\pi_1(J_\Pi)\twoheadrightarrow \pi_1(M(L)).
$$
\end{lemma}
\begin{proof}
Let $U=|L|\setminus\Delta$ be the open subset of smooth curves and set
$U_\Pi=U\cap\Pi$. By the Zariski--Lefschetz theorem for
quasi-projective varieties \cite[Theorem~1.1.3]{HL85}, applied
successively to general hyperplane sections, the map
$$
\pi_1(U_\Pi)\longrightarrow \pi_1(U)
$$
is surjective.

Over $U_\Pi$ and $U$, the support morphisms are smooth families of Picard varieties. The homotopy exact sequences for these smooth fibrations, together with the surjectivity of $\pi_1(U_\Pi)\to\pi_1(U)$, give a surjection
$$
\pi_1((J_\Pi)_{U_\Pi})\twoheadrightarrow \pi_1(M(L)_U).
$$
The inclusions
$$
(J_\Pi)_{U_\Pi}\subset J_\Pi,
\qquad
M(L)_U\subset M(L)
$$
induce surjections on fundamental groups by
\cite[Proposition~2.10]{Kol95}. Therefore the map
$$
\pi_1(J_\Pi)\longrightarrow \pi_1(M(L))
$$
is surjective.
\end{proof}

\begin{proof}[Proof of \cref{thm:fundamental-group-fixed-det}]
By \cref{lem:pencil-curve-to-jacobian-surjection}, there is a surjection
$$
\pi_1(S)\twoheadrightarrow \pi_1(J_\Pi),
$$
By \cref{lem:pencil-to-global-surjection}, there is a surjection
$$
\pi_1(J_\Pi)\twoheadrightarrow \pi_1(M(L)).
$$
Their composition gives
$
\pi_1(S)\twoheadrightarrow \pi_1(M(L)).
$

For the abelianization, \cref{thm:first-betti-fixed-det} gives $b_1(M(L))=2$. By the Hurewicz theorem in degree one,
$$
H_1(M(L),\zz)\cong\pi_1(M(L))_{ab}.
$$
Thus
$
\rank\bigl(\pi_1(M(L))_{ab}\bigr)
=
\dim_{\qq}H_1(M(L),\qq)
=
b_1(M(L))
=
2.
$
\end{proof}

\subsection{First two Betti numbers of Albanese fibers of \texorpdfstring{$M(L)$}{M(L)}}\label{subsec:albanese-fibers-rank-zero}
We compute the first two Betti numbers of the Albanese fibers of $M(L)$.
\Cref{cor:alpha-is-albanese-up-to-isogeny} shows that the Albanese morphism is
a locally constant analytic fibration, so it suffices to compute the
cohomology of a general fiber using the support morphism.

\begin{theorem}\label{thm:b2-albanese-fiber}
Let $(S,H)$ be a polarized bielliptic surface, and let $\mv=(0,c_1(L),\chi)$ be a primitive Mukai vector with $\nt(L)\geqslant 3$. Assume that $H$ is $\mv$-generic.
Let $F$ be an Albanese fiber of $M(L)$. Then
$$
b_1(F)=0,
\qquad
b_2(F)=3.
$$
\end{theorem}

We first relate the Albanese morphism to the extended norm morphism.

Let $a:S\to E\coloneqq\Alb(S)$ be the Albanese morphism.

\begin{proposition}\label{cor:alpha-is-albanese-up-to-isogeny}
After translating the target by a constant, the morphism
$\alpha:M(L)\to\Pic^0(E)$ constructed in \cref{lem:alpha-construction}
factors as
$$
M(L)\xrightarrow{\operatorname{alb}_{M(L)}}\Alb(M(L))
\xrightarrow{\varphi}\Pic^0(E),
$$
where $\varphi$ is an isogeny. Moreover, the Albanese morphism of $M(L)$
is a locally constant analytic fibration with connected fibers. Consequently,
all Albanese fibers are mutually isomorphic normal irreducible projective
varieties with terminal l.c.i. singularities and torsion canonical bundle.
For every $\xi\in\Pic^0(E)$, the connected components of
$\alpha^{-1}(\xi)$ are precisely the Albanese fibers over the finitely many
points of $\varphi^{-1}(\xi)$, after accounting for the fixed translation.
\end{proposition}
\begin{proof}
After a translation of $\Pic^0(E)$, the universal property of the Albanese
variety gives the asserted factorization. By \cref{lem:alpha-construction},
$\alpha$ is surjective. By \cref{thm:first-betti-fixed-det}, $b_1(M(L))=2$,
so $\Alb(M(L))$ is an elliptic curve. Hence the induced homomorphism
$\varphi:\Alb(M(L))\to\Pic^0(E)$ is surjective and therefore is an isogeny.

By \cref{cor:rank-zero-terminal-lci}, $M(L)$ is a normal projective variety
with terminal l.c.i. singularities and torsion canonical bundle. In
particular, it is klt and $-K_{M(L)}$ is nef. Thus
\cite[Theorem~A]{Wan20}, applied with boundary zero, shows that the Albanese
morphism of $M(L)$ is a locally constant analytic fibration with connected
fibers. Its fibers are mutually biholomorphic and, being projective, mutually
algebraically isomorphic. The local analytic product structure shows that each
fiber is normal and l.c.i. and has terminal singularities. A connected normal
variety is irreducible. Finally, since the Albanese base is smooth with trivial
canonical bundle, adjunction in the local product structure identifies the
canonical bundle of a fiber with the restriction of $K_{M(L)}$; hence it is
torsion.

Up to the fixed translation, $\alpha=\varphi\circ\operatorname{alb}_{M(L)}$.
Since $\varphi$ is finite, $\alpha^{-1}(\xi)$ is the disjoint union of the
Albanese fibers over the finite set $\varphi^{-1}(\xi)$. These fibers are
connected, so they are exactly the connected components of
$\alpha^{-1}(\xi)$.
\end{proof}

Let $U\subset |L|$ be the open subset parametrizing smooth curves, and let
$\Delta\coloneqq |L|\setminus U$ be the discriminant. Since
$\nt(L)\geqslant 3$, \cref{lem:bertini_results} shows that $L$ is very ample,
$\Delta$ is irreducible, and a general point of $\Delta$ parametrizes an
irreducible curve with one ordinary node. We denote by
$q\colon\uc\longrightarrow |L|$ the universal curve, by
$q_U\colon\uc_U\longrightarrow U$ its restriction over $U$, and set
$M(L)_U\coloneqq h^{-1}(U)$.

\begin{lemma}\label{lem:universal-curve-cohomology-bielliptic}
The universal curve $\uc$ is a projective bundle over $S$. In particular,
$$
H^1(\uc,\cc)\cong H^1(S,\cc)\cong \cc^2
$$
and
$$
H^2(\uc,\cc)\cong H^2(S,\cc)\oplus H^0(S,\cc)
\cong \cc^3.
$$
Moreover,
$
\dim W_2H^2(\uc_U,\cc)=2.
$
\end{lemma}
\begin{proof}
Since $\nt(L)\geqslant 3$, the line bundle $L$ is base point free. Hence the projection
$$
p\colon \uc\longrightarrow S
$$
is the projective bundle whose fiber over $x\in S$ is the projective space of divisors in $|L|$ passing through $x$. By the projective bundle formula,
$$
H^1(\uc,\cc)\cong H^1(S,\cc)
$$
and
$$
H^2(\uc,\cc)\cong H^2(S,\cc)\oplus H^0(S,\cc).
$$
For a bielliptic surface $b_1(S)=2$ and $b_2(S)=2$, so this gives
$$
\dim H^1(\uc,\cc)=2,
\qquad
\dim H^2(\uc,\cc)=3.
$$

Since $\Delta$ is irreducible and its general point parametrizes an irreducible one-nodal curve, the divisor
$
\uc_\Delta\coloneqq \uc\setminus \uc_U
$
has a unique irreducible divisorial component $D_{\uc}$; write
$$
\uc_\Delta=D_{\uc}\cup Z,
\qquad
\operatorname{codim}_{\uc}Z\geqslant 2.
$$
The discriminant $\Delta\subset |L|$ is an effective Cartier divisor, and
$q^*\Delta$ is an effective Cartier divisor on $\uc$ whose unique divisorial
support is $D_{\uc}$. Hence
$$
q^*\Delta=m_{\uc}D_{\uc}
$$
for some $m_{\uc}\geqslant 1$, so $D_{\uc}$ is $\qq$-Cartier. Since $\uc$ is
smooth and projective, \cref{lem:low-degree-purity-one-boundary} gives an exact
sequence
$$
0\longrightarrow \cc(-1)
\longrightarrow H^2(\uc,\cc)
\longrightarrow W_2H^2(\uc_U,\cc)
\longrightarrow 0.
$$
Therefore
$
\dim W_2H^2(\uc_U,\cc)=3-1=2.
$
\end{proof}

\begin{lemma}\label{lem:invariant-cycles-bielliptic}
Let $D\in U$. Then
$$
H^1(D,\qq)^{\pi_1(U,D)}
=
\im\left(H^1(S,\qq)\xrightarrow{i_D^*} H^1(D,\qq)\right).
$$
Equivalently, if we denote by
$$
\mathbb V_D\coloneqq\ker\left((i_D)_*\colon H_1(D,\cc)\to H_1(S,\cc)\right),
$$
then, after dualizing, there is an exact sequence of local systems
$$
0\longrightarrow H^1(S,\cc)_U
\longrightarrow R^1q_{U,*}\cc
\longrightarrow \mathbb V^\vee
\longrightarrow 0,
$$
and the invariant sections of $R^1q_{U,*}\cc$ are exactly the constant subsystem $H^1(S,\cc)_U$.
\end{lemma}
\begin{proof}
The projection
$
p\colon \uc\longrightarrow S
$
is a projective bundle, hence
$
H^1(\uc,\qq)\cong H^1(S,\qq).
$
By the global invariant cycle theorem applied to
$
q_U\colon \uc_U\to U,
$
one has
$$
H^1(D,\qq)^{\pi_1(U,D)}
=
\im\left(H^1(\uc,\qq)\to H^1(D,\qq)\right).
$$
Using $H^1(\uc,\qq)\cong H^1(S,\qq)$, this becomes
$$
H^1(D,\qq)^{\pi_1(U,D)}
=
\im\left(H^1(S,\qq)\xrightarrow{i_D^*}H^1(D,\qq)\right).
$$
The statement about $\mathbb V$ follows by dualizing the exact sequence
$$
H_1(D,\cc)\longrightarrow H_1(S,\cc)\longrightarrow 0,
$$
where surjectivity follows from the Lefschetz hyperplane theorem for smooth ample curves on $S$.
\end{proof}

\begin{lemma}\label{lem:connected-norm-kernel}
Let $D\in U$ be a smooth curve and set $a_D\coloneqq\restrict{a}{D}$. Then
$$
(a_D)_*\colon H_1(D,\zz)\longrightarrow H_1(E,\zz)
$$
is surjective. Consequently,
$$
P_D\coloneqq\ker\left(\Nm_{a_D}\colon\Pic^0(D)\longrightarrow\Pic^0(E)\right)
$$
is a connected abelian variety of dimension $g(D)-1$.
\end{lemma}

\begin{proof}
By \cref{lem:bertini_results}, $L$ is very ample. Hence $D$ is a smooth hyperplane section of $S$ for the embedding defined by $|L|$. The integral Lefschetz hyperplane theorem \cite[Theorem~1.23]{VoisinHodge2} gives a surjection
$$
(i_D)_*\colon H_1(D,\zz)\longrightarrow H_1(S,\zz).
$$
The Albanese morphism induces the canonical quotient onto the Albanese lattice,
$$
a_*\colon H_1(S,\zz)\longrightarrow
H_1(S,\zz)/H_1(S,\zz)_{\mathrm{tors}}
\simeq H_1(E,\zz).
$$
Therefore $(a_D)_*=a_*\circ(i_D)_*$ is surjective.

Under the Abel--Jacobi identifications
$$
H_1(\Pic^0(D),\zz)\simeq H_1(D,\zz),
\qquad
H_1(\Pic^0(E),\zz)\simeq H_1(E,\zz),
$$
the homomorphism induced by $\Nm_{a_D}$ is $(a_D)_*$. Indeed, after choosing base points, functoriality of the Abel--Jacobi maps identifies the composite $\Nm_{a_D}\circ\operatorname{AJ}_D$ with $\operatorname{AJ}_E\circ a_D$ up to translation. Translations act trivially on homology. Thus \cref{lem:connected-torus-kernel} shows that the full kernel of $\Nm_{a_D}$ is connected. Its dimension is $g(D)-1$ because the norm map is surjective by \cref{lem:norm-map-albanese} and $\dim\Pic^0(E)=1$.
\end{proof}

\begin{lemma}\label{lem:relative-albanese-kernel}
Let $\xi\in\Pic^0(E)$ be general. Then
$\alpha^{-1}(\xi)\cap M(L)_U$ is connected. Let $F_{\mathrm{gen}}$ be the
unique connected component of $\alpha^{-1}(\xi)$ containing this open subset,
and set
$$
F_{\mathrm{gen},U}\coloneqq F_{\mathrm{gen}}\cap M(L)_U.
$$
Then $F_{\mathrm{gen}}$ is an Albanese fiber. Let
$
f_U:=h|_{F_{\mathrm{gen},U}}:F_{\mathrm{gen},U}\to U
$
and also let
$
q_U:\mathcal C_U\to U
$
be the universal smooth curve. Define the local system
$$
\mathbb V:=\ker\left(R_1q_{U,*}\mathbb C\to H_1(S,\mathbb C)_U\right),
$$
where \(H_1(S,\mathbb C)_U\) denotes the constant local system on \(U\) with fiber
\(H_1(S,\mathbb C)\). Thus, for \(D\in U\),
$$
\mathbb V_D
=
\ker\left(H_1(D,\mathbb C)\to H_1(S,\mathbb C)\right).
$$
Then \(f_U\) is a smooth proper fibration, and for \(D\in U\), the fiber \(f_U^{-1}(D)\)
is a torsor under
$$
P_D:=\ker\left(\operatorname{Pic}^0(D)\to \operatorname{Pic}^0(E)\right)
$$
and $\dim P_D=g(D)-1$. Moreover,
$$
R^1f_{U,*}\mathbb C\simeq \mathbb V^\vee,
\qquad
R^2f_{U,*}\mathbb C\simeq \bigwedge^2\mathbb V^\vee.
$$
\end{lemma}

\begin{proof}
The relative norm is a homomorphism of abelian schemes
$$
\Nm\colon\Pic^0(\mathcal C_U/U)\longrightarrow\Pic^0(E)\times U.
$$
For every geometric point $D\in U$, its restriction to the fiber is
$$
\Nm_{a_D}\colon\Pic^0(D)\longrightarrow\Pic^0(E),
$$
which is surjective by \cref{lem:norm-map-albanese}. Hence the image of $\Nm$ is a closed subgroup scheme of $\Pic^0(E)\times U$ whose every geometric fiber is the whole fiber. Therefore $\Nm$ is surjective.

Let
$
\mathcal P\coloneqq\ker(\Nm).
$
The kernel of a surjective homomorphism of abelian schemes is a smooth proper group scheme over the base. By \cref{lem:connected-norm-kernel}, every geometric fiber of $\mathcal P$ is connected. Consequently, $\mathcal P$ is an abelian scheme over $U$ of relative dimension $g(L)-1$.

By \cref{lem:alpha-construction}, over $U$ the morphism $\alpha$ agrees, up to a fixed translation of the target, with the relative norm map on the relative degree-$d$ Picard torsor. Hence
$
\alpha^{-1}(\xi)\cap M(L)_U
$
is a torsor under $\mathcal P$ for every $\xi\in\Pic^0(E)$.

The open subset $U$ is connected because it is the complement of an
irreducible hypersurface in the projective space $|L|$. Since the fibers of
$\mathcal P\to U$ are connected, the torsor
$$
\alpha^{-1}(\xi)\cap M(L)_U\longrightarrow U
$$
is connected. Consequently, it lies in a unique connected component
$F_{\mathrm{gen}}$ of $\alpha^{-1}(\xi)$, and
$$
F_{\mathrm{gen},U}
=F_{\mathrm{gen}}\cap M(L)_U
=\alpha^{-1}(\xi)\cap M(L)_U.
$$
By \cref{cor:alpha-is-albanese-up-to-isogeny}, the connected components of a
fiber of $\alpha$ are Albanese fibers, so $F_{\mathrm{gen}}$ is an Albanese
fiber. Moreover, $F_{\mathrm{gen},U}\to U$ is a torsor under $\mathcal P$.
Its fiber over $D$ is a torsor under the full norm kernel
$$
P_D=\ker\left(\operatorname{Pic}^0(D)\to \operatorname{Pic}^0(E)\right).
$$
By \cref{lem:connected-norm-kernel}, this is a connected abelian variety of dimension $g(D)-1$. Hence $f_U$ is smooth and proper with connected fibers.

Taking first homology gives
$$
H_1(P_D,\mathbb C)
=
\ker\left(H_1(D,\mathbb C)\to H_1(E,\mathbb C)\right).
$$
Since \(E=\operatorname{Alb}(S)\), the natural map
$
H_1(S,\mathbb C)\to H_1(E,\mathbb C)
$
is an isomorphism. Hence
$$
H_1(P_D,\mathbb C)
=
\ker\left(H_1(D,\mathbb C)\to H_1(S,\mathbb C)\right)
=
\mathbb V_D.
$$
Thus
$
H^1(P_D,\mathbb C)\simeq \mathbb V_D^\vee.
$
Since \(P_D\) is an abelian variety, its cohomology algebra is the exterior algebra on
\(H^1(P_D,\mathbb C)\). Therefore
$$
H^2(P_D,\mathbb C)
\simeq
\bigwedge^2\mathbb V_D^\vee.
$$
These identifications are induced by the exact sequence of integral homology local systems associated with the homomorphism of abelian schemes $\Nm$, and hence are compatible with monodromy. Therefore they globalize to the local-system identifications
$$
R^1f_{U,*}\mathbb C\simeq \mathbb V^\vee,
\qquad
R^2f_{U,*}\mathbb C\simeq \bigwedge^2\mathbb V^\vee.
$$
\end{proof}

\begin{lemma}\label{lem:variable-monodromy-irreducible}
The local system
$$
\mathbb V_D=\ker\left(H_1(D,\cc)\to H_1(S,\cc)\right)
$$
has no nonzero invariant sections. Moreover, the monodromy representation of $\pi_1(U,D)$ on $\mathbb V_D$ is irreducible, and
$$
H^0(U,\mathbb V^\vee)=0,
\qquad
H^0\left(U,\bigwedge^2\mathbb V^\vee\right)\cong \cc.
$$
\end{lemma}
\begin{proof}
By \cref{lem:invariant-cycles-bielliptic},
$
H^1(D,\cc)^{\pi_1(U,D)}
=
i_D^*H^1(S,\cc).
$
The exact sequence
$$
0\longrightarrow H^1(S,\cc)_U
\longrightarrow R^1q_{U,*}\cc
\longrightarrow \mathbb V^\vee
\longrightarrow 0
$$
is an exact sequence of polarizable variations of Hodge structure. By semisimplicity, taking invariant sections is exact on this sequence. Hence
$
H^0(U,\mathbb V^\vee)=0.
$

It remains to prove irreducibility of the monodromy on $\mathbb V_D$. Since $L$ is very ample, a general pencil in $|L|$ is a Lefschetz pencil. The total space of such a pencil is the blow-up of $S$ at the base points, hence has the same $H_1$ as $S$. Therefore the vanishing cycles generate precisely
$$
\ker\left(H_1(D,\cc)\to H_1(S,\cc)\right)=\mathbb V_D.
$$
The restriction of the intersection form to $i_D^*H^1(S,\cc)$ is
nondegenerate. Indeed,
$$
\int_D i_D^*\alpha\wedge i_D^*\beta
=
\int_S\alpha\wedge\beta\wedge[D],
$$
and hard Lefschetz for the ample class $[D]$ gives nondegeneracy. Hence its
symplectic orthogonal complement, identified with $\mathbb V_D$, is also
nondegenerate.
Because the discriminant is irreducible, the vanishing cycles form one monodromy orbit up to sign. Let $0\neq W\subset\mathbb V_D$ be monodromy invariant and choose $0\neq w\in W$. Since the vanishing cycles span $\mathbb V_D$ and the intersection pairing on $\mathbb V_D$ is nondegenerate, there is a vanishing cycle $\delta$ such that $(w,\delta)\neq 0$. The Picard--Lefschetz formula gives
$$
T_\delta(w)-w=\pm(w,\delta)\delta\in W,
$$
so $\delta\in W$. The single-orbit property then implies that $W$ contains every vanishing cycle and hence equals $\mathbb V_D$.

The intersection pairing induces a nondegenerate monodromy-invariant symplectic form on $\mathbb V_D$. By irreducibility and Schur's lemma, the invariant part of $\bigwedge^2\mathbb V_D^\vee$ is one-dimensional, generated by the symplectic class. Thus
	$$
	H^0\left(U,\bigwedge^2\mathbb V^\vee\right)\cong \cc.
	$$
\end{proof}

\begin{lemma}\label{lem:nodal-compactified-norm-fiber}
After possibly shrinking $\Delta^\circ$ while keeping
$
\operatorname{codim}_{|L|}(\Delta\setminus\Delta^\circ)\geqslant 2,
$
the following holds. For every $D_0\in\Delta^\circ$, let
$\nu\colon\widetilde D_0\to D_0$ be the normalization and set
$\widetilde a\coloneqq a|_{D_0}\circ\nu$. Then
$$
\ker\left(\Nm_{\widetilde a}\colon
\Pic^0(\widetilde D_0)\longrightarrow\Pic^0(E)\right)
$$
is connected, and every nonempty fiber of
$$
\restrict{\alpha}{h^{-1}(D_0)}\colon
h^{-1}(D_0)\longrightarrow\Pic^0(E)
$$
is irreducible.
\end{lemma}

\begin{proof}
By \cref{lem:num_div} and $\nt(L)\geqslant 3$, one has $L^2\geqslant 18>5$
and $L\cdot C\geqslant 3>2$ for every integral curve $C\subset S$. Hence the
theorem of Ein--Lazarsfeld \cite[Main Theorem, p.~178]{EinLaz93}, applied with
$e=2$, gives a finite
subset $Z_L\subset S$ such that
$
\varepsilon(L,x)>2
$
for every $x\notin Z_L$. For a fixed $x\in S$, the hyperplanes tangent to
$S\subset\ps^N$ at $x$ form a projective space of dimension $N-3$, whereas
$\dim\Delta=N-1$. We may therefore remove from $\Delta^\circ$ the curves whose
node lies in $Z_L$ without changing the codimension condition above.

Fix $D_0\in\Delta^\circ$ and let $x$ be its node. If
$\pi\colon\widehat S\to S$ is the blow-up at $x$, with exceptional curve $E_x$,
then the strict transform of $D_0$ is $\widetilde D_0$ and
$$
\widetilde D_0\equiv\pi^*L-2E_x.
$$
Since $2<\varepsilon(L,x)$, this divisor is ample. Applying the argument of
\cref{lem:connected-norm-kernel} to the ample curve
$\widetilde D_0\subset\widehat S$, and using
$H_1(\widehat S,\zz)\simeq H_1(S,\zz)$, shows that
$$
\widetilde a_*\colon H_1(\widetilde D_0,\zz)
\longrightarrow H_1(E,\zz)
$$
is surjective. By \cref{lem:connected-torus-kernel}, the kernel of
$\Nm_{\widetilde a}$ is connected.

We now use the presentation of the compactified Jacobian of an integral
one-nodal curve. Its normalization is a $\ps^1$-bundle
$$
\rho\colon P\longrightarrow\Pic^{d'}(\widetilde D_0)
$$
for the appropriate degree $d'$, and the two distinguished sections are
identified by translation by $\mathcal O_{\widetilde D_0}(p-q)$, where $p,q$
are the two points over the node; see \cite[Theorem~2.6, Lemma~2.9 and
Remark~2.10]{Mat25}. Let $I\subset |L|$ be the integral-support locus and let
$$
a_I\colon\mathcal C_I\longrightarrow E\times I
$$
be induced by the Albanese morphism. The morphism $a_I$ is finite, and the universal curve $\mathcal C_I$ is Cohen--Macaulay. Since $E\times I$ is smooth and has the same dimension, the derived pushforwards below are perfect. For a line bundle $A$ on $\mathcal C_I$, define its relative norm by the determinant-of-cohomology formula
$$
\Nm_{a_I}(A)
\coloneqq
\det R(a_I)_*A
\otimes
\det R(a_I)_*\mathcal O_{\mathcal C_I}^{-1}.
$$
This construction is compatible with base change and induces the usual norm morphism on the relative Picard functor; it does not require $a_I$ to be flat. The resulting relative norm and $\alpha$ agree on the dense smooth-support locus. Since $I$ is reduced and $\Pic^0(E)$ is separated, they agree everywhere. Consequently,
if $n\colon P\to h^{-1}(D_0)$ is the normalization morphism, then, up to a fixed
translation of the target,
$
\alpha\circ n=\Nm_{\widetilde a}\circ\rho.
$
Indeed, let $F$ be a line bundle on $D_0$ and set $A\coloneqq\nu^*F$. Writing
$x$ for the node, the normalization sequences are
$$
0\longrightarrow F
\longrightarrow\nu_*A
\longrightarrow k_x
\longrightarrow0
$$
and
$$
0\longrightarrow\mathcal O_{D_0}
\longrightarrow\nu_*\mathcal O_{\widetilde D_0}
\longrightarrow k_x
\longrightarrow0.
$$
Applying determinant of cohomology along $a|_{D_0}$, the two factors coming
from $k_x$ cancel. Hence, up to the fixed translation of the target,
$$
\Nm_{a|_{D_0}}(F)=\Nm_{\widetilde a}(A).
$$
Thus $\alpha\circ n=\Nm_{\widetilde a}\circ\rho$ on the dense line-bundle
locus, and hence everywhere because both sides are morphisms and the target
is separated. Notice also that
$\widetilde a(p)=\widetilde a(q)$, so
$$
\Nm_{\widetilde a}\bigl(\mathcal O_{\widetilde D_0}(p-q)\bigr)
\simeq\mathcal O_E,
$$
and the boundary identification preserves every norm fiber.

Every nonempty fiber of $\Nm_{\widetilde a}$ is a torsor under its connected
kernel and is therefore irreducible. Its inverse image under $\rho$ is an
irreducible $\ps^1$-bundle, whose image under $n$ is the corresponding fiber of
$\restrict{\alpha}{h^{-1}(D_0)}$. The latter is thus irreducible.
\end{proof}

\begin{lemma}\label{lem:boundary-purity-F}
Let $F_{\mathrm{gen}}$ and $F_{\mathrm{gen},U}$ be as in
\cref{lem:relative-albanese-kernel}. Then
$$
0\longrightarrow \cc(-1)
\longrightarrow H^2(F_{\mathrm{gen}},\cc)
\longrightarrow W_2H^2(F_{\mathrm{gen},U},\cc)
\longrightarrow 0.
$$
\end{lemma}
\begin{proof}
Let $\Delta^\circ\subset\Delta$ be the open subset fixed in
\cref{lem:nodal-compactified-norm-fiber}.

We first analyze the boundary over $\Delta^\circ$. By
\cref{lem:nodal-compactified-norm-fiber}, after shrinking $\Delta^\circ$ without
changing the codimension condition above, the compactified norm fiber over every
$D_0\in\Delta^\circ$ is irreducible of dimension $g(L)-1$. Let $\xi$ be the
general point of $\Pic^0(E)$ fixed in \cref{lem:relative-albanese-kernel}.
Over $U$, the space
$\alpha^{-1}(\xi)\cap M(L)_U=F_{\mathrm{gen},U}$ is irreducible by that
lemma; hence $F_{\mathrm{gen}}$ is the unique component of
$\alpha^{-1}(\xi)$ dominating $|L|$. The fiber dimension theorem shows that
$F_{\mathrm{gen}}\cap h^{-1}(D_0)$ has dimension at least $g(L)-1$. It is a closed subset of
the irreducible compactified norm fiber of that dimension, and therefore equals
the whole fiber. Since $\Delta^\circ$ is irreducible, the geometric generic
fiber of
$$
F_{\mathrm{gen}}\cap h^{-1}(\Delta^\circ)\longrightarrow\Delta^\circ
$$
is irreducible. After shrinking $\Delta^\circ$ once more, this total space is
irreducible. Let
$$
D_{\mathrm{gen}}\coloneqq
\overline{F_{\mathrm{gen}}\cap h^{-1}(\Delta^\circ)}
\subset F_{\mathrm{gen}}.
$$
Then $D_{\mathrm{gen}}$ is an irreducible divisorial component of
$F_{\mathrm{gen}}\setminus F_{\mathrm{gen},U}$.

It remains to prove that the boundary over
$W\coloneqq \Delta\setminus\Delta^\circ$ has codimension at least two in
$F_{\mathrm{gen}}$. Put $Y_W\coloneqq h^{-1}(W)\subset M(L)$.
By \cref{ass:comp_jacobian}, every support fiber has dimension at most \(g(L)\). Hence
$$
\dim Y_W\leqslant \dim W+g(L).
$$
Since
$
\operatorname{codim}_{|L|}W\geqslant 2,
$
we get
$$
\dim Y_W\leqslant \dim |L|-2+g(L).
$$
Now consider the restriction $\restrict{\alpha}{Y_W}:Y_W\longrightarrow \Pic^0(E)$.
By the tensor-equivariance statement of \cref{lem:norm-map-albanese}, the restriction of the extended norm morphism $\alpha$ to every nonempty support fiber over $W$ is surjective onto $\Pic^0(E)$. Therefore, if \(Y_W\neq\varnothing\), the image of \(Y_W\) under \(\alpha\) is one-dimensional.
For the general point $\xi$ fixed in
\cref{lem:relative-albanese-kernel}, the general-fiber dimension estimate
applied to $\restrict{\alpha}{Y_W}:Y_W\to \Pic^0(E)$ gives
$$
\dim(F_{\mathrm{gen}}\cap Y_W)
\leqslant
\dim Y_W-1
\leqslant
\dim |L|-2+g(L)-1.
$$
On the other hand, $\dim F_{\mathrm{gen}}=\dim |L|+g(L)-1$.
Therefore
$$
\operatorname{codim}_{F_{\mathrm{gen}}}
\bigl(F_{\mathrm{gen}}\cap h^{-1}(W)\bigr)
\geqslant 2.
$$
Thus
$$
F_{\mathrm{gen}}\setminus F_{\mathrm{gen},U}=D_{\mathrm{gen}}\cup Z,
$$
where $D_{\mathrm{gen}}$ is irreducible of codimension one and
$\operatorname{codim}_{F_{\mathrm{gen}}} Z\geqslant 2$.

By \cref{cor:alpha-is-albanese-up-to-isogeny}, $F_{\mathrm{gen}}$ is normal
and has terminal l.c.i. singularities.

The discriminant $\Delta\subset |L|$ is an effective Cartier divisor, and
$F_{\mathrm{gen}}$ is not contained in its inverse image. Hence
$$
\restrict{h}{F_{\mathrm{gen}}}^{*}\Delta
$$
is an effective Cartier divisor on $F_{\mathrm{gen}}$. Its unique divisorial
support is $D_{\mathrm{gen}}$, so, as Weil divisors,
$$
\restrict{h}{F_{\mathrm{gen}}}^{*}\Delta
=m_{\mathrm{gen}}D_{\mathrm{gen}}
$$
for some integer $m_{\mathrm{gen}}\geqslant 1$. In particular,
$D_{\mathrm{gen}}$ is $\qq$-Cartier; no assertion about $m_{\mathrm{gen}}$
is needed. Therefore
\cref{lem:low-degree-purity-one-boundary} applies and gives the desired exact
sequence.

\end{proof}

\begin{lemma}\label{lem:prym-comparison}
There is an equality
$$
\dim W_2H^2(F_{\mathrm{gen},U},\cc)
=
\dim W_2H^2(\uc_U,\cc).
$$
Consequently,
$
\dim W_2H^2(F_{\mathrm{gen},U},\cc)=2.
$
\end{lemma}
\begin{proof}
Over $U$, the morphism
$
f_U\colon F_{\mathrm{gen},U}\to U
$
is a smooth proper family whose fiber over $D\in U$ is a torsor under
$$
P_D=\ker\left(\Pic^0(D)\to\Pic^0(E)\right).
$$
Therefore
$
R^1f_{U,*}\cc\cong \mathbb V^\vee,
\,
R^2f_{U,*}\cc\cong \bigwedge^2\mathbb V^\vee.
$

For the universal curve
$$
q_U\colon \uc_U\to U,
$$
the invariant-cycle lemma and semisimplicity of polarizable variations of Hodge
structure give a splitting
$$
R^1q_{U,*}\cc
\cong H^1(S,\cc)_U\oplus \mathbb V^\vee.
$$
The Leray spectral sequences for the smooth proper morphisms $f_U$ and $q_U$
degenerate at $E_2$ by Deligne's degeneration theorem for smooth proper maps
\cite[Theorem~4.1.1(i)]{Del71}. We compare their three terms in total degree two. The $(2,0)$-terms
are both $H^2(U,\cc)$. Since $\Delta$ is irreducible, $H^1(U,\cc)=0$, and hence
$$
H^1\left(U,H^1(S,\cc)_U\right)=0.
$$
Thus the splitting above identifies the $(1,1)$-terms:
$$
H^1(U,R^1q_{U,*}\cc)
\cong
H^1(U,\mathbb V^\vee)
=
H^1(U,R^1f_{U,*}\cc).
$$
For the $(0,2)$-terms, one has
$
R^2q_{U,*}\cc\cong \cc_U
$
because the fibers of $q_U$ are smooth curves. Thus
$H^0(U,R^2q_{U,*}\cc)\cong \cc$. On the other hand, by
\cref{lem:variable-monodromy-irreducible},
$$
H^0\left(U,R^2f_{U,*}\cc\right)
=
H^0\left(U,\bigwedge^2\mathbb V^\vee\right)
\cong \cc.
$$
These one-dimensional spaces are pure of weight two. More generally, for
each of the two Leray spectral sequences,
$$
E_2^{p,q}=H^p(U,R^qf_{U,*}\cc)
$$
has weights at least $p+q$, and similarly for $q_U$. Since the spectral
sequences degenerate at $E_2$ in the category of mixed Hodge structures,
strictness of the weight filtration computes $W_2H^2$ term by term from the
three terms with $p+q=2$. The comparisons above are compatible with mixed
Hodge structures and identify their weight-two parts. Hence
$$
\dim W_2H^2(F_{\mathrm{gen},U},\cc)
=
\dim W_2H^2(\uc_U,\cc).
$$
By \cref{lem:universal-curve-cohomology-bielliptic},
$
\dim W_2H^2(\uc_U,\cc)=2
$
and therefore
$
\dim W_2H^2(F_{\mathrm{gen},U},\cc)=2.
$
\end{proof}

\begin{proof}[Proof of \cref{thm:b2-albanese-fiber}]
Since $\Delta$ is irreducible, $H^1(U,\cc)=0$. Moreover,
\cref{lem:relative-albanese-kernel,lem:variable-monodromy-irreducible} give
$$
R^1f_{U,*}\cc\simeq\mathbb V^\vee,
\qquad
H^0(U,\mathbb V^\vee)=0.
$$
The Leray spectral sequence therefore gives
$$
H^1(F_{\mathrm{gen},U},\cc)=0.
$$
By \cref{lem:low-degree-purity-one-boundary}, restriction induces an
isomorphism
$$
H^1(F_{\mathrm{gen}},\cc)
\simeq H^1(F_{\mathrm{gen},U},\cc).
$$
Thus
$
H^1(F_{\mathrm{gen}},\cc)=0.
$

For the second Betti number, \cref{lem:boundary-purity-F} gives an exact sequence
$$
0\longrightarrow \cc(-1)
\longrightarrow H^2(F_{\mathrm{gen}},\cc)
\longrightarrow W_2H^2(F_{\mathrm{gen},U},\cc)
\longrightarrow 0.
$$
Therefore
$$
b_2(F_{\mathrm{gen}})
=1+\dim W_2H^2(F_{\mathrm{gen},U},\cc).
$$
By \cref{lem:prym-comparison},
$\dim W_2H^2(F_{\mathrm{gen},U},\cc)=2$. Consequently,
$$
b_1(F_{\mathrm{gen}})=0,
\qquad
b_2(F_{\mathrm{gen}})=3.
$$
By \cref{cor:alpha-is-albanese-up-to-isogeny}, every Albanese fiber of
$M(L)$ is isomorphic to $F_{\mathrm{gen}}$. Hence the same equalities hold
for the Albanese fiber $F$ appearing in the statement.
\end{proof}

\subsection{First two Betti numbers of \texorpdfstring{$M$ and $M(L)$}{M and M(L)}}\label{subsec:betti-rank-zero}
This subsection completes the rank-zero calculation. We first compute the
Betti numbers of $M$ and then recover those of $M(L)$ from its description as
an elliptic slice of the Albanese fibration of $M$.
\begin{theorem}\label{thm:b1-b2-moduli}
Let $(S,H)$ be a polarized bielliptic surface and let $\mv=(0,c_1(L),\chi)$ be a primitive Mukai vector with $\nt(L)\geqslant 3$. Assume that $H$ is $\mv$-generic.
Then
$$
b_1(M(L))=2,
\qquad
b_2(M(L))=4,
$$
and
$$
b_1(M)=4,
\qquad
b_2(M)=9.
$$
\end{theorem}

We first record two consequences of the Albanese-fibration structure that will
also be used in the next subsection.

\begin{lemma}\label{lem:determinant-albanese-slices}
Let $M$ be a normal irreducible projective variety with terminal singularities
and torsion canonical bundle, and assume that
$
b_1(M)=4.
$
Let
$$
a_M:M\longrightarrow A\coloneqq\Alb(M)
$$
be its Albanese morphism, and let
$
d:M\longrightarrow P
$
be a surjective morphism to an elliptic curve which, after choosing origins
and translating the target, factors as
$
d=q\circ a_M
$
for a surjective homomorphism $q:A\to P$. For $L\in P$, let $C$ be a
connected component of $q^{-1}(L)$ and set
$
N_C\coloneqq a_M^{-1}(C).
$
Then:
\begin{enumerate}
    \item $A$ is an abelian surface and $C$ is an elliptic curve;
    \item the varieties $N_C$, as $C$ runs through the connected components
    of $q^{-1}(L)$, are precisely the irreducible components of $d^{-1}(L)$;
    \item the restriction
    $$
    a_M|_{N_C}:N_C\longrightarrow C
    $$
    is a locally constant analytic fibration whose fibers are Albanese fibers
    of $M$.
\end{enumerate}
If, in addition, $b_1(N_C)=2$, then, after identifying $C$ with
$\Alb(N_C)$, the morphism $a_M|_{N_C}$ is the Albanese morphism of $N_C$.
\end{lemma}
\begin{proof}
Since terminal singularities are rational, birational invariance of first
cohomology for varieties with rational singularities gives
$$
\dim\Alb(M)=\frac{1}{2}b_1(M)=2.
$$
See \cite[Theorem~5.22]{KollarMori} and
\cite[Lemma~2.1]{BakkerLehnTorelli}. Thus $A$ is an abelian surface. Since
$d=q\circ a_M$ is surjective, so is $q$. Therefore every connected component
of $q^{-1}(L)$ is a translate of the elliptic curve $\ker(q)^0$.

The variety $M$ is klt and $-K_M$ is nef, so \cite[Theorem~A]{Wan20} shows
that $a_M$ is a locally constant analytic fibration with connected fibers.
Thus each $N_C\to C$ has connected base and fiber and is normal by the local
product structure; hence it is irreducible. Moreover,
$$
d^{-1}(L)
=
a_M^{-1}\bigl(q^{-1}(L)\bigr)
=
\coprod_{C'\in\pi_0(q^{-1}(L))}N_{C'}.
$$
Thus the $N_C$ are exactly the irreducible components of $d^{-1}(L)$.

Suppose that $b_1(N_C)=2$. The local product structure gives terminal
singularities, while the triviality of the normal bundle of $C$ in $A$ and
adjunction show that $K_{N_C}$ is torsion. Hence $\Alb(N_C)$ is an elliptic
curve, and the universal property gives a factorization
$$
N_C\xrightarrow{\operatorname{alb}_{N_C}}\Alb(N_C)
\xrightarrow{\varphi}C.
$$
The homomorphism $\varphi$ is an isogeny. Both Albanese fibrations have
connected fibers by \cite[Theorem~A]{Wan20}; therefore $\deg\varphi=1$, for
otherwise a fiber of $a_M|_{N_C}$ would be a disjoint union of several
Albanese fibers of $N_C$. Thus $\varphi$ is an isomorphism.
\end{proof}

\begin{lemma}\label{lem:betti-elliptic-albanese-slice}
Let
$
a:M\longrightarrow A
$
be a locally constant analytic fibration over an abelian surface, with
connected fiber $F$. Assume that
$$
b_1(F)=0,
\qquad
b_2(F)=3,
\qquad
b_2(M)=9.
$$
Let $C\subset A$ be an elliptic curve and set
$
N\coloneqq a^{-1}(C).
$
Then
$$
b_1(N)=2,
\qquad
b_2(N)=4.
$$
\end{lemma}
\begin{proof}
Since $H^1(F,\cc)=0$, the Leray spectral sequence gives
$$
0\longrightarrow H^2(A,\cc)
\longrightarrow H^2(M,\cc)
\longrightarrow H^2(F,\cc)^{\pi_1(A)}.
$$
The dimensions $6$, $9$, and $3$ force
$$
H^2(F,\cc)^{\pi_1(A)}=H^2(F,\cc),
$$
so the monodromy on $H^2(F,\cc)$ is trivial.

Restricting to $C$, the Leray spectral sequence for $N\to C$ gives
$$
H^1(N,\cc)\simeq H^1(C,\cc)
$$
and an exact sequence
$$
0\longrightarrow H^2(C,\cc)
\longrightarrow H^2(N,\cc)
\longrightarrow H^2(F,\cc)
\longrightarrow 0.
$$
Therefore $b_1(N)=2$ and $b_2(N)=1+3=4$.
\end{proof}

Let $P\coloneqq\Pic^{c_1(\mv)}(S)$. This is a torsor under $\Pic^0(S)$, hence an elliptic curve. Choose a normalized Poincar\'e bundle $\mathcal Q$ on $S\times P$, and let
$$
p_P\colon S\times P\longrightarrow P
$$
be the projection. Since $h^0(S,N)$ is constant for $N\in P$, cohomology and base change show that
$$
\mathcal E\coloneqq p_{P,*}\mathcal Q
$$
is locally free and has fiber $H^0(S,N)$ over $N\in P$. Consequently, the relative complete linear system is
$$
\mathbb B
\coloneqq
\{(N,D)\mid N\in P,\ D\in |N|\}
\cong
\mathbb P_P(\mathcal E^\vee).
$$
Since $\nt(L)\geqslant 3$, every $N\in P$ is very ample. Thus the universal evaluation morphism on $S\times P$ is surjective and gives an exact sequence
$$
0\longrightarrow\mathcal K_0
\longrightarrow p_P^*\mathcal E
\longrightarrow\mathcal Q
\longrightarrow 0.
$$
The universal curve is therefore
$\mathfrak C\cong\mathbb P_{S\times P}(\mathcal K_0^\vee)$. In particular,
$\mathbb B\to P$ and $\mathfrak C\to S\times P$ are projective bundles. Their fibers have dimensions $g-2$ and $g-3$, respectively, where $g=g(L)$.

Let $U_M\subset\mathbb B$ be the open subset parametrizing smooth curves, let $m_U\colon M_U\to U_M$ be the smooth-support part of the determinant-unfixed component, and let $r_U\colon \mathfrak C_U\to U_M$ be the universal smooth curve.

\begin{lemma}\label{lem:relative-discriminant-irreducible}
Let $P\coloneqq\Pic^{c_1(\mv)}(S)$ and let $\mathbb B\coloneqq\{(N,D)\mid N\in P,\ D\in |N|\}$ be the relative complete linear system over \(P\). Assume that every \(N\in P\) is very ample. Let $\Delta_{\mathbb B}\subset\mathbb B$ be the relative discriminant parametrizing singular divisors. Then \(\Delta_{\mathbb B}\) is an irreducible hypersurface, and a general point of \(\Delta_{\mathbb B}\) parametrizes an irreducible curve with one ordinary node.
\end{lemma}
\begin{proof}
Consider the incidence variety
$$
\mathcal I
=
\{(N,D,x)\in \mathbb B\times S\mid x\in\operatorname{Sing}(D)\}.
$$
Equivalently, \(\mathcal I\) parametrizes triples \((N,D,x)\) such that the section defining \(D\) vanishes at \(x\) to order at least two.

The relative first-jet evaluation morphism on $S\times P$ is
$$
p_P^*\mathcal E
\longrightarrow
J^1_{p_P}(\mathcal Q).
$$
Since $N$ is very ample, the morphism defined by $|N|$ is a closed immersion.
It therefore separates values and all first-order tangent directions, so
$$
H^0(S,N)\longrightarrow
H^0(S,N\otimes\mathcal O_S/\mathfrak m_x^2)
$$
is surjective for every $x\in S$. Hence the relative first-jet evaluation
morphism is surjective. Its target has rank $3$, with fiber
$
H^0(S,N\otimes\mathcal O_S/\mathfrak m_x^2)
$
over $(x,N)$. Hence its kernel $\mathcal K_1$ is locally free, and
$
\mathcal I\cong\mathbb P_{S\times P}(\mathcal K_1^\vee).
$
Thus $\mathcal I$ is a projective bundle over the irreducible variety $S\times P$, and is therefore irreducible.

The image of \(\mathcal I\) under the projection
$
\mathcal I\to \mathbb B
$
is precisely the discriminant $\Delta_{\mathbb B}$. Hence
$\Delta_{\mathbb B}$ is irreducible.

Finally, the same Lefschetz-pencil argument used in
\cref{lem:bertini_results} shows that a general singular divisor in a very
ample system has exactly one ordinary double point and is otherwise smooth.
Hence a general point of $\Delta_{\mathbb B}$ parametrizes an irreducible
one-nodal curve. Since $\dim\mathcal I=\dim\mathbb B-1$ and the projection is
generically one-to-one over this locus, $\Delta_{\mathbb B}$ is a
hypersurface.
\end{proof}

\begin{lemma}\label{lem:universal-curve-full-base}
Let
$
\mathfrak C\to\mathbb B
$
be the universal curve over the determinant-unfixed relative linear system.
Then $\mathfrak C$ is a projective bundle over $S\times P$. Consequently,
$$
\dim H^2(\mathfrak C,\cc)=8 \text{\quad and \quad } \dim W_2H^2(\mathfrak C_U,\cc)=7.
$$
\end{lemma}
\begin{proof}
The construction above identifies
$\mathfrak C\cong\mathbb P_{S\times P}(\mathcal K_0^\vee)$. By the projective
bundle formula,
$$
H^2(\mathfrak C,\cc)
\cong
H^2(S\times P,\cc)\oplus H^0(S\times P,\cc).
$$
Since
$$
b_1(S)=2,\qquad b_2(S)=2,\qquad b_1(P)=2,\qquad b_2(P)=1,
$$
the Künneth formula gives
$$
b_2(S\times P)=b_2(S)+b_1(S)b_1(P)+b_2(P)=2+4+1=7.
$$
Therefore
$
\dim H^2(\mathfrak C,\cc)=7+1=8.
$

By \cref{lem:relative-discriminant-irreducible}, the discriminant in
\(\mathbb B\) is irreducible and its general point parametrizes an integral
one-nodal curve. Let $\Delta_{\mathbb B}^\circ$ be a dense open subset
parametrizing such curves. Hence
$
\mathfrak C\setminus\mathfrak C_U
$
has one irreducible divisorial component $D_{\mathfrak C}$, and the remaining
boundary has codimension at least two. Indeed, if
$$
W\coloneqq\Delta_{\mathbb B}\setminus\Delta_{\mathbb B}^\circ,
$$
then $\operatorname{codim}_{\mathbb B}W\geqslant 2$. Since the universal-curve fibers have dimension one,
$$
\dim\bigl(\mathfrak C\times_{\mathbb B}W\bigr)
\leqslant\dim W+1
\leqslant\dim\mathbb B-1
=\dim\mathfrak C-2.
$$
Thus
$$
\mathfrak C\setminus\mathfrak C_U=D_{\mathfrak C}\cup Z,
\qquad
\operatorname{codim}_{\mathfrak C}Z\geqslant 2.
$$
Write $r\colon\mathfrak C\to\mathbb B$ for the universal curve. Since
$\mathbb B$ is smooth, the hypersurface
$\Delta_{\mathbb B}\subset\mathbb B$ is an effective Cartier divisor. Its
pullback has unique divisorial support $D_{\mathfrak C}$, and hence
$$
r^*\Delta_{\mathbb B}=m_{\mathfrak C}D_{\mathfrak C}
$$
for some $m_{\mathfrak C}\geqslant 1$. Therefore $D_{\mathfrak C}$ is
$\qq$-Cartier. Since $\mathfrak C$ is smooth and projective,
\cref{lem:low-degree-purity-one-boundary} gives
$$
0\longrightarrow\cc(-1)
\longrightarrow H^2(\mathfrak C,\cc)
\longrightarrow W_2H^2(\mathfrak C_U,\cc)
\longrightarrow 0.
$$
Thus
$
\dim W_2H^2(\mathfrak C_U,\cc)=8-1=7.
$
\end{proof}

\begin{lemma}\label{lem:open-b2-full-M}
One has
$
\dim W_2H^2(M_U,\cc)=8.
$
\end{lemma}
\begin{proof}
Over $U_M$, the determinant-unfixed moduli space is the relative Picard variety
of the universal smooth curve. By \cref{lem:A1},
$$
R^1m_{U,*}\cc
\cong
R^1r_{U,*}\cc.
$$
The fixed part of this local system is still $H^1(S,\cc)$. Indeed, the total
universal curve $\mathfrak C$ has additional $H^1(P,\cc)$-classes, but these
restrict trivially to a curve fiber because a curve fiber lies over a single
point of $P$. Thus
$$
R^1m_{U,*}\cc
\cong
H^1(S,\cc)_{U_M}\oplus\mathbb V^\vee.
$$
Here $\mathbb V$ denotes the variable summand over $U_M$.
Since the fibers are abelian varieties,
$$
R^2m_{U,*}\cc
\cong
\bigwedge^2R^1m_{U,*}\cc.
$$

We compare the total-degree-two terms of the Leray spectral sequences for
$m_U$ and $r_U$. Their $(2,0)$-terms are both $H^2(U_M,\cc)$. Their
$(1,1)$-terms are canonically isomorphic by the first displayed isomorphism:
$$
H^1(U_M,R^1m_{U,*}\cc)
\cong
H^1(U_M,R^1r_{U,*}\cc).
$$
In particular, although $H^1(U_M,\cc)$ is nonzero, the contribution
$$
H^1(U_M,\cc)\otimes H^1(S,\cc)
$$
of the constant subsystem occurs in both spectral sequences and produces no
difference between them.

It remains to compare the $(0,2)$-terms. Restriction to the fixed-determinant
slice $U\subset U_M$ gives the monodromy representation considered in
\cref{lem:variable-monodromy-irreducible}. Hence
$
H^0(U_M,\mathbb V^\vee)=0
$
and
$
\dim H^0\left(U_M,\bigwedge^2\mathbb V^\vee\right)\leqslant 1.
$
The fiberwise symplectic form gives a nonzero global section, so
$$
H^0\left(U_M,\bigwedge^2\mathbb V^\vee\right)\cong\cc.
$$
Using the first vanishing to eliminate the mixed summand, taking invariant
sections gives
$$
H^0(U_M,R^2m_{U,*}\cc)
\cong
\bigwedge^2H^1(S,\cc)
\oplus
H^0\left(U_M,\bigwedge^2\mathbb V^\vee\right)
\cong\cc^2.
$$
On the other hand, the fibers of $r_U$ are smooth curves, so
$$
R^2r_{U,*}\cc\cong\cc_{U_M}
$$
and
$$
H^0(U_M,R^2r_{U,*}\cc)\cong\cc.
$$

The Leray spectral sequences degenerate at $E_2$ by Deligne's degeneration
theorem \cite[Theorem~4.1.1(i)]{Del71}. Their terms
$$
E_2^{p,q}=H^p(U_M,R^qf_*\cc)
$$
have weights at least $p+q$, where $f$ denotes $m_U$ or $r_U$. Hence
degeneration and strictness of the weight filtration compute $W_2H^2$ term by
term from the three terms with $p+q=2$. All the identifications above are
compatible with mixed Hodge structures. The $(2,0)$- and $(1,1)$-terms have
the same weight-two parts, while the $(0,2)$-term for $m_U$ has one additional
pure weight-two summand, namely $\bigwedge^2H^1(S,\cc)$. Therefore
$$
\dim W_2H^2(M_U,\cc)
=
\dim W_2H^2(\mathfrak C_U,\cc)+1.
$$
By \cref{lem:universal-curve-full-base},
$
\dim W_2H^2(\mathfrak C_U,\cc)=7.
$
Therefore
$
\dim W_2H^2(M_U,\cc)=8.
$
\end{proof}

\begin{lemma}\label{lem:boundary-purity-full-M}
Let $M_U\subset M$ be the smooth-support locus of $M$. Then
$$
0\longrightarrow \cc(-1)
\longrightarrow H^2(M,\cc)
\longrightarrow W_2H^2(M_U,\cc)
\longrightarrow 0.
$$
\end{lemma}
\begin{proof}
Let $m\colon M\to\mathbb B$ be the support morphism. By
\cref{lem:relative-discriminant-irreducible}, the discriminant
$\Delta_{\mathbb B}\subset\mathbb B$ is irreducible and its general member is
integral and one-nodal. By Rego's theorem \cite[Theorem~A]{Reg80}, the
compactified Jacobian of the generic curve in $\Delta_{\mathbb B}$ is
irreducible. Hence the support fiber
over the generic point of $\Delta_{\mathbb B}$ has a unique irreducible
component. Together with the dimension estimate of
\cref{ass:comp_jacobian}, this gives
$$
M\setminus M_U=D_M\cup Z
$$
where $D_M$ is irreducible and
$\operatorname{codim}_MZ\geqslant 2$.

The unique divisorial support of $m^*\Delta_{\mathbb B}$ is $D_M$, so
$
m^*\Delta_{\mathbb B}=m_MD_M
$
for some $m_M\geqslant 1$, and hence $D_M$ is $\qq$-Cartier. Now
\cref{cor:rank-zero-terminal-lci,lem:low-degree-purity-one-boundary} give the
asserted sequence.
\end{proof}

\begin{proof}[Proof of \cref{thm:b1-b2-moduli}]
We first compute the Betti numbers of $M$. For $b_1$, the base $U_M$ is the
complement of the irreducible relative discriminant divisor
$
\Delta_{\mathbb B}\subset\mathbb B.
$
Since \(\mathbb B\) is smooth and \(\Delta_{\mathbb B}\) is an irreducible divisor, the localization sequence gives
$$
H^1(\mathbb B,\cc)\longrightarrow H^1(U_M,\cc)
\longrightarrow H^0(\Delta_{\mathbb B},\cc)(-1)
\longrightarrow H^2(\mathbb B,\cc).
$$
The last map sends \(1\) to the cohomology class \([\Delta_{\mathbb B}]\). This class is nonzero because \(\Delta_{\mathbb B}\) is a nonzero effective Cartier divisor on the projective variety \(\mathbb B\). Hence
$$
H^1(\mathbb B,\cc)\xrightarrow{\sim} H^1(U_M,\cc).
$$
Since \(\mathbb B\to P\) is a projective bundle, the projective bundle formula gives $H^1(\mathbb B,\cc)\cong H^1(P,\cc)$. Therefore $\dim H^1(U_M,\cc)=2$.
The invariant part of $H^1(D,\cc)$ for the fiber
curve is the image of $H^1(S,\cc)$, also of dimension $2$. 
By \cref{lem:boundary-purity-full-M,lem:low-degree-purity-one-boundary},
restriction induces
$$
H^1(M,\cc)\xrightarrow{\sim}H^1(M_U,\cc).
$$
The Leray spectral sequence for $M_U\to U_M$ degenerates at $E_2$ by
\cite[Theorem~4.1.1(i)]{Del71}. Consequently,
$$
\dim H^1(M_U,\cc)
=
\dim H^1(U_M,\cc)
+
\dim H^0(U_M,R^1m_{U,*}\cc)
=2+2=4.
$$
Thus $b_1(M)=4$.

For $b_2$, \cref{lem:boundary-purity-full-M} gives
$$
0\longrightarrow\cc(-1)
\longrightarrow H^2(M,\cc)
\longrightarrow W_2H^2(M_U,\cc)
\longrightarrow0.
$$
Therefore $b_2(M)=1+\dim W_2H^2(M_U,\cc)$. By
\cref{lem:open-b2-full-M}, $\dim W_2H^2(M_U,\cc)=8$, and hence $b_2(M)=1+8=9$.

It remains to compute the Betti numbers of $M(L)$. Let
$$
a_M:M\longrightarrow A\coloneqq\Alb(M)
$$
be the Albanese morphism and let
$$
d\coloneqq\det^\circ:M\longrightarrow
P\coloneqq\Pic^{c_1(\mv)}(S).
$$
The morphism $d$ is surjective by \cref{lem:translation-base-change} and,
after choosing origins and translating the target, factors as
$$
d=q\circ a_M
$$
for a surjective homomorphism $q:A\to P$.

By \cref{lem:determinant-albanese-slices}, $M(L)$ is the inverse image under
$a_M$ of a connected component
$
C\subset q^{-1}(L),
$
and $C$ is an elliptic curve. By \cref{thm:first-betti-fixed-det},
$
b_1(M(L))=2.
$
Therefore the final assertion of \cref{lem:determinant-albanese-slices}
identifies
$$
M(L)\longrightarrow C
$$
with the Albanese morphism of $M(L)$. Its fibers are consequently both
Albanese fibers of $M(L)$ and Albanese fibers of $M$.

By \cref{thm:b2-albanese-fiber}, such a fiber $F$ satisfies
$$
b_1(F)=0,
\qquad
b_2(F)=3.
$$
Since $b_2(M)=9$, \cref{lem:betti-elliptic-albanese-slice} applied to
$
M(L)=a_M^{-1}(C)
$
gives
$
b_2(M(L))=4.
$
It also recovers $b_1(M(L))=2$.
\end{proof}

\subsection{Extension from rank zero to admissible Mukai vectors}\label{sec:ext-to-pos-rank}\label{subsec:extension-admissible}
We now extend the rank-zero computations to arbitrary primitive
admissible Mukai vectors using birational reduction, minimal-model
theory, and invariance of low-degree cohomology under isomorphisms in
codimension one when the singular locus has sufficiently high
codimension.

Recall the notations introduced in \cref{subsec:moduli-spaces} and \cref{subsec:admissible}. For a
primitive admissible Mukai vector $\mv$ and a $\mv$-generic polarization
$H$, set
$$
M:=M^\circ_{H,S}(\mv),
\qquad
\det^\circ:=\det|_M.
$$
For $L\in\operatorname{Pic}^{c_1(\mv)}(S)$, set
$
M(L):=M_{H,S}(\mv,L)
$
if $\rank(\mv)=0$. If $\rank(\mv)>0$, fix an
irreducible component of $(\det^\circ)^{-1}(L)$ and, by abuse of
notation, denote it also by $M(L)$. Finally, let $F$ denote an Albanese fiber
of $M(L)$.

\begin{theorem}\label{thm:betti-numbers-admissible}
Let $(S,H)$ be a polarized bielliptic surface, let $\mv$ be a primitive
admissible Mukai vector, and let
$
L\in\operatorname{Pic}^{c_1(\mv)}(S).
$
Assume that $H$ is $\mv$-generic. There is an elliptic curve
$
C_L\subset\operatorname{Alb}(M)
$
such that
$
M(L)=a_M^{-1}(C_L).
$
Moreover, the restriction
$$
a_M|_{M(L)}:M(L)\longrightarrow C_L
$$
is the Albanese morphism of $M(L)$, and its fibers are Albanese fibers
of $M$. The first two Betti numbers are
$$
\begin{array}{c|cc}
 & b_1 & b_2\\
\hline
F & 0 & 3\\
M(L) & 2 & 4\\
M & 4 & 9.
\end{array}
$$
\end{theorem}

\begin{proof}
Set $X:=M$. If $\rank(\mv)>0$, then the convention following
\cref{lem:admissible-positive-rank-to-rank-zero} gives a primitive rank-zero Mukai vector
$$
\mw=(0,D,\chi),
\qquad
\nt(D)\geq 3,
$$
and a $\mw$-generic polarization $H'$ such that
$$
X=M^\circ_{H,S}(\mv)
\dashrightarrow
Y:=M^\circ_{H',S}(\mw)
$$
is birational. If $\mv$ already has rank zero, take $\mw=\mv$, $H'=H$,
$Y=X$, and the identity birational map.

Write numerically
$
D\equiv aA_0+bB_0.
$
Since $\nt(D)\geqslant 3$, one has $a,b\geqslant 3$, and therefore
$$
\mv^2=\mw^2=D^2=2ab\geqslant 18\geqslant 3k_S.
$$
By \cite[Theorem~1.2 and Proposition~9.2]{Nue25}, $X$ and $Y$ are normal
projective varieties with terminal l.c.i. singularities and torsion Cartier
canonical divisor.

We next verify the codimension condition needed for
\cref{lem:h2-isomorphism-codimension-one}. The singular-locus estimate of
\cite[Proposition~9.1]{Nue25} gives
$$
\operatorname{codim}_X\operatorname{Sing}(X)
>
\frac{k_S-1}{k_S}\mv^2-1,
$$
and the analogous estimate for $Y$. Since $\mv^2\geqslant 18$ and
$k_S\in\{2,3,4,6\}$, the right-hand side is at least $8$. In particular,
$$
\operatorname{codim}_X\operatorname{Sing}(X)\geqslant 4,
\qquad
\operatorname{codim}_Y\operatorname{Sing}(Y)\geqslant 4.
$$

These varieties are locally factorial. Indeed, every point of codimension at
most three is regular, while at a singular point the local ring is a
complete-intersection local ring of dimension at least four. Grothendieck's
parafactoriality theorem \cite[Expos\'e~XI, Corollaire~3.14]{SGA2} therefore
applies. Thus $X$ and $Y$ are $\qq$-factorial.

Their canonical divisors are torsion and hence nef. Consequently, $X$ and $Y$
are $\qq$-factorial terminal minimal models. By \cite[Theorem~1]{Kaw08}, the
birational map $X\dashrightarrow Y$ is a composition of flops and, in
particular, is an isomorphism in codimension one. Hence
\cref{lem:h2-isomorphism-codimension-one} gives
$$
H^2(X,\cc)\simeq H^2(Y,\cc).
$$
Terminal singularities are rational; see
\cite[Theorem~5.22]{KollarMori}. Let $Z$ be a projective resolution of the
normalization of the graph of the birational map $X\dashrightarrow Y$, with
proper birational morphisms
$$
p:Z\longrightarrow X,
\qquad
q:Z\longrightarrow Y.
$$
Applying \cite[Lemma~2.1]{BakkerLehnTorelli} to $p$ and $q$ gives
$$
H^1(X,\zz)\simeq H^1(Z,\zz)\simeq H^1(Y,\zz).
$$
Since $\mw$ is rank zero and $\nt(D)\geqslant 3$,
\cref{thm:b1-b2-moduli} gives
$$
b_1(Y)=4,
\qquad
b_2(Y)=9.
$$
Hence
$
b_1(X)=4,
\,
b_2(X)=9.
$

We next determine the Albanese fibers of the rank-zero model $Y$.
Choose
$
\Lambda\in\operatorname{Pic}^{D}(S)
$
and set
$$
Y(\Lambda):=M_{H',S}(\mw,\Lambda).
$$
The restricted determinant morphism of $Y$ is surjective by
\cref{lem:translation-base-change}. Since
$$
b_1(Y)=4
\qquad\text{and}\qquad
b_1(Y(\Lambda))=2
$$
by \cref{thm:first-betti-fixed-det},
\cref{lem:determinant-albanese-slices} shows that the Albanese morphism of
$Y(\Lambda)$ is the restriction of the Albanese morphism of $Y$, and
that their fibers coincide. \Cref{thm:b2-albanese-fiber} therefore gives
$$
b_1(F_Y)=0,
\qquad
b_2(F_Y)=3
$$
for every Albanese fiber $F_Y$ of $Y$.

It remains to transfer this result from $Y$ to $X$. Since $X$ and $Y$ have
rational singularities, a common resolution identifies their Picard varieties
and hence their Albanese varieties. We therefore identify
$$
\Alb(X)\simeq\Alb(Y)\eqqcolon A
$$
so that, after a translation, the two Albanese morphisms agree on the common
domain of the birational map.

Write $a_X:X\to A$ for the Albanese morphism of $X$, and use the same symbol
$a_Y$ for the Albanese morphism of $Y$ after the identification above.

Choose common open subsets $U_X\subset X$ and $U_Y\subset Y$ on which the
birational map is an isomorphism and whose complements have codimension at
least two. For a general $t\in A$, put
$$
F_{X,t}\coloneqq a_X^{-1}(t),
\qquad
F_{Y,t}\coloneqq a_Y^{-1}(t).
$$
A general-fiber dimension estimate shows that
$$
F_{X,t}\setminus U_X,
\qquad
F_{Y,t}\setminus U_Y
$$
have codimension at least two in the respective fibers. Hence the induced
birational map
$$
F_{X,t}\dashrightarrow F_{Y,t}
$$
is an isomorphism in codimension one.

The Albanese morphisms of $X$ and $Y$ are locally constant analytic
fibrations by \cite[Theorem~A]{Wan20}. Their local product structures show
that their fibers are normal projective varieties with terminal l.c.i.
singularities and that
$$
\operatorname{codim}_{F_{X,t}}\operatorname{Sing}(F_{X,t})
=
\operatorname{codim}_X\operatorname{Sing}(X)
\geqslant 4,
$$
with the analogous equality for $Y$.
\Cref{lem:h2-isomorphism-codimension-one} therefore gives
$$
H^2(F_{X,t},\cc)\simeq H^2(F_{Y,t},\cc).
$$
Since the fibers have terminal, hence rational, singularities,
let $Z_t$ be a projective resolution of the normalization of the graph of
$F_{X,t}\dashrightarrow F_{Y,t}$. Applying
\cite[Lemma~2.1]{BakkerLehnTorelli} to the two proper birational morphisms
from $Z_t$ gives
$$
H^1(F_{X,t},\zz)\simeq H^1(Z_t,\zz)\simeq H^1(F_{Y,t},\zz).
$$
Thus
$$
b_1(F_{X,t})=0,
\qquad
b_2(F_{X,t})=3.
$$
Finally, all fibers of each Albanese morphism are mutually isomorphic by
\cite[Theorem~A]{Wan20}. The same equalities therefore hold for every
Albanese fiber $F$ of $X$.

It remains to compute the Betti numbers of the determinant components of
$X=M$. We first verify that
$$
d_X\coloneqq\det^\circ:X\longrightarrow
P_X\coloneqq\Pic^{c_1(\mv)}(S)
$$
is surjective. If $\rank(\mv)=0$, this is
\cref{lem:translation-base-change}. Suppose that
$$
r\coloneqq\rank(\mv)>0.
$$
The elliptic curve $\Pic^0(S)$ acts on the moduli space by tensor product.
Tensoring by a numerically trivial line bundle preserves the Mukai vector and
Gieseker stability. Since $\Pic^0(S)$ is connected, this action preserves the
chosen irreducible component $X$. For $Q\in\Pic^0(S)$ and $E\in X$, one has
$$
\det(E\otimes Q)\simeq\det(E)\otimes Q^{\otimes r}.
$$
Multiplication by $r$ on the elliptic curve $\Pic^0(S)$ is surjective. Hence
the determinants of the sheaves $E\otimes Q$ run through all of $P_X$, proving
that $d_X$ is surjective.

After choosing origins and translating the target, the determinant morphism
factors as
$
d_X=q_X\circ a_X
$
for a surjective homomorphism
$$
q_X:\Alb(X)\longrightarrow P_X.
$$
Let $L$ and $M(L)$ be as in the statement. By \cref{lem:determinant-albanese-slices}, there is a
connected component
$
C_L\subset q_X^{-1}(L)
$
such that
$$
M(L)=a_X^{-1}(C_L),
$$
and $C_L$ is an elliptic curve. The morphism
$$
M(L)\longrightarrow C_L
$$
is the restriction of the locally constant Albanese fibration of $X$,
with fiber $F$.

We have already proved that
$$
b_2(X)=9,
\qquad
b_1(F)=0,
\qquad
b_2(F)=3.
$$
Therefore \cref{lem:betti-elliptic-albanese-slice} gives
$$
b_1(M(L))=2,
\qquad
b_2(M(L))=4.
$$
The final assertion of \cref{lem:determinant-albanese-slices} then shows that
$$
M(L)\longrightarrow C_L
$$
is the Albanese morphism of $M(L)$. This proves all the assertions.
\end{proof}

\begin{remark}
When $\rank(\mv)>0$, the conclusions of \cref{thm:betti-numbers-admissible} hold
for every irreducible component of
$
(\det^\circ)^{-1}(L).
$
Indeed, the final part of the proof applies separately to each such
component. Thus, if $N$ is any irreducible component of this fiber,
then
$
N=a_M^{-1}(C)
$
for an elliptic curve $C\subset\operatorname{Alb}(M)$, the restriction
$
a_M|_N:N\longrightarrow C
$
is the Albanese morphism of $N$, and
$$
b_1(N)=2,
\qquad
b_2(N)=4.
$$
\end{remark}

\begin{remark}
In the cases covered by \cref{thm:rank-zero-hilbert-model}, the birational model for the moduli space is
$
\Pic^0(S)\times\Hilb^{\mv^2/2}(S)
$
and its first two Betti numbers are computed independently in
Appendix~\ref{app:betti-hilbert}.
\end{remark}

\subsection{Hodge structure and Calabi--Yau covers of Albanese fibers}
\label{subsec:hodge-cy-albanese-fibers}

We refine the second-cohomology calculation of
\cref{thm:betti-numbers-admissible} and, in the rank-one-orbit case, relate the
Albanese fibers to Calabi--Yau varieties. We retain the notation of
\cref{sec:ext-to-pos-rank}. Thus $F$ is an Albanese fiber of
$M_{H,S}(\mv,L)$ and, by \cref{thm:betti-numbers-admissible}, it is also an
Albanese fiber of $M^\circ_{H,S}(\mv)$.

For a normal variety $X$, we write
$
\Omega_X^{[p]}:=\left(\Omega_X^p\right)^{**}
$
for the sheaf of reflexive differential $p$-forms. Equivalently, its
sections are holomorphic $p$-forms on the smooth locus $X_{\mathrm{reg}}$
extended reflexively across the singular locus; see
\cite[Notation~2.19]{GGK19}.

\begin{theorem}\label{prop:albanese-fiber-hodge-type}
Let $(S,H)$ be a polarized bielliptic surface and let $\mv$ be a primitive
admissible Mukai vector. Assume that $H$ is $\mv$-generic. Then $H^2(F,\cc)$
is a pure Hodge structure of weight two and type $(1,1)$. Equivalently,
$$
H^2(F,\cc)=H^{1,1}(F).
$$
In particular,
$
H^0\left(F,\Omega_F^{[2]}\right)=0.
$
\end{theorem}

\begin{proof}
We first consider the rank-zero case. By
\cref{cor:alpha-is-albanese-up-to-isogeny}, all Albanese fibers are mutually
isomorphic, so it is enough to work with the fiber $F_{\mathrm{gen}}$ used in
\cref{lem:relative-albanese-kernel}. We retain the notation
$$
f_U:F_{\mathrm{gen},U}\longrightarrow U,
\qquad
q_U:\uc_U\longrightarrow U,
$$
and let $\mathbb V$ be the variable homology local system, with fiber
$$
\mathbb V_D
=
\ker\left(H_1(D,\cc)\longrightarrow H_1(S,\cc)\right)
$$
over $D\in U$, as in \cref{lem:invariant-cycles-bielliptic}.

For a smooth proper algebraic morphism $g:Z\to U$, the Leray spectral
sequence
$$
E_2^{p,q}(g)
=
H^p\left(U,R^qg_*\cc\right)
\Longrightarrow
H^{p+q}(Z,\cc)
$$
is a spectral sequence of mixed Hodge structures and degenerates at $E_2$.
If $\mathbb W$ is a polarizable variation of Hodge structure of weight $q$,
then $H^p(U,\mathbb W)$ has weights at least $p+q$; see
\cite[Theorem~3.7]{PetersSaito12}. Consequently, if $L^\bullet$ denotes the
Leray filtration, then
$$
\operatorname{Gr}_L^pW_2H^2(Z,\cc)
\simeq
W_2H^p\left(U,R^{2-p}g_*\cc\right).
$$

For $q_U$, the three associated graded pieces are
$$
W_2H^2(U,\cc),
\qquad
W_2H^1(U,\mathbb V^{\vee}),
\qquad
\cc(-1).
$$
Indeed,
$$
R^1q_{U,*}\cc
\simeq
H^1(S,\cc)_U\oplus\mathbb V^{\vee},
$$
and $H^1(U,\cc)=0$, so the constant summand contributes nothing.
For $f_U$, \cref{lem:relative-albanese-kernel} gives
$$
R^1f_{U,*}\cc\simeq\mathbb V^{\vee},
\qquad
R^2f_{U,*}\cc\simeq\bigwedge^2\mathbb V^{\vee}.
$$
By \cref{lem:variable-monodromy-irreducible}, the invariant part of the
second local system is the line generated by the fiberwise polarization
class. This class is flat and has Hodge type $(1,1)$ on every fiber. Hence,
as a Hodge structure,
$$
H^0\left(U,R^2f_{U,*}\cc\right)\simeq\cc(-1).
$$
It follows that the associated graded Hodge structures of
$$
W_2H^2(F_{\mathrm{gen},U},\cc)
\qquad\text{and}\qquad
W_2H^2(\uc_U,\cc)
$$
are isomorphic term by term with respect to their Leray filtrations.

By \cref{lem:universal-curve-cohomology-bielliptic}, the universal curve
$\uc$ is a projective bundle over $S$. Since a bielliptic surface has no
holomorphic two-forms, the projective bundle formula shows that
$H^2(\uc,\cc)$ is of type $(1,1)$. The exact sequence
$$
0\longrightarrow\cc(-1)
\longrightarrow H^2(\uc,\cc)
\longrightarrow W_2H^2(\uc_U,\cc)
\longrightarrow 0
$$
from the proof of \cref{lem:universal-curve-cohomology-bielliptic} therefore
shows that $W_2H^2(\uc_U,\cc)$ is of type $(1,1)$. Its Leray filtration is a
filtration by mixed Hodge substructures, so each of the three graded pieces
above is of type $(1,1)$. The corresponding associated graded pieces for
$
W_2H^2(F_{\mathrm{gen},U},\cc)
$
are therefore of type $(1,1)$. Since its Leray filtration is a filtration by
mixed Hodge substructures, $W_2H^2(F_{\mathrm{gen},U},\cc)$ is itself of
type $(1,1)$. Finally, \cref{lem:boundary-purity-F} gives an exact
sequence of mixed Hodge structures
$$
0\longrightarrow\cc(-1)
\longrightarrow H^2(F_{\mathrm{gen}},\cc)
\longrightarrow W_2H^2(F_{\mathrm{gen},U},\cc)
\longrightarrow 0.
$$
Both the subobject and the quotient are pure Hodge structures of weight two
and type $(1,1)$. Therefore $H^2(F_{\mathrm{gen}},\cc)$ is also pure of
weight two and of type $(1,1)$. The same holds for every rank-zero Albanese
fiber.

For an arbitrary admissible Mukai vector, the proof of
\cref{thm:betti-numbers-admissible} produces, for general Albanese fibers, a
birational map to a rank-zero Albanese fiber which is an isomorphism in
codimension one. The fibers are l.c.i., and their singular loci have
codimension at least four. The isomorphism of second cohomology groups supplied
by \cref{lem:h2-isomorphism-codimension-one} is an isomorphism of mixed Hodge
structures: its proof constructs it entirely from restriction morphisms to a
common smooth open subset. The rank-zero result therefore transfers to the
general fiber. Since all fibers of each Albanese morphism are mutually
isomorphic, it holds for every $F$.

The variety $F$ is projective and terminal, hence klt. By
\cite[Theorem~1]{SchwaldKLT}, $H^2(F,\cc)$ is pure and
$$
H^{2,0}(F)
\simeq
H^0\left(F,\Omega_F^{[2]}\right).
$$
The final assertion follows.
\end{proof}

We use the terminology of \cite{GGK19}. A finite surjective morphism between
normal varieties is called \textit{quasi-\'etale} if it is \'etale in codimension one.
A normal projective variety $X$ of dimension at least two is called a
Calabi--Yau variety if it has canonical singularities,
$
\omega_X\simeq\mathcal O_X,
$
and, for every finite quasi-\'etale cover $Y\to X$,
$$
H^0\left(Y,\Omega_Y^{[p]}\right)=0
\qquad
\text{for }0<p<\dim X.
$$
See \cite[Definitions~1.3 and~2.15]{GGK19}.

For a normal variety $X$, its reflexive tangent sheaf is
$\mathcal T_X:=\mathcal H om_{\mathcal O_X}(\Omega_X^1,\mathcal O_X)$. A
reflexive sheaf $\mathcal E$ on a normal projective variety $X$ of
dimension $n$ is called \textit{strongly stable} if, for every finite morphism
$f:Y\to X$ which is \'etale in codimension one and every choice of ample
divisors $H_1,\ldots,H_{n-1}$ on $Y$, the reflexive pullback
$$
f^{[*]}\mathcal E
:=
\left(f^*\mathcal E\right)^{**}
$$
is slope-stable with respect to $(H_1,\ldots,H_{n-1})$; see
\cite[Definition~7.2]{GKP16}.

\begin{theorem}\label{thm:albanese-fiber-cy-cover}
Let $(S,H)$ be a polarized bielliptic surface and let $\mv$ be a primitive
admissible Mukai vector. Assume that $H$ is $\mv$-generic and that
$
\la_S=\ell(\mv)=1.
$
Then there exists a finite quasi-\'etale cover
$$
\gamma:\widetilde F\longrightarrow F
$$
such that $\widetilde F$ is a Calabi--Yau variety in the sense of
\cite[Definition~1.3]{GGK19}. In particular, the reflexive tangent sheaf
$\mathcal T_F$ is strongly stable.
\end{theorem}

\begin{proof}
Set $n\coloneqq\mv^2/2$. Let
$
\mw=(0,D,\chi)
$
be the rank-zero reduction used in the proof of
\cref{thm:betti-numbers-admissible}. By
\cref{lem:l-invariant-under-autoeq},
$
\ell(\mw)=\ell(\mv)=1.
$
Since $\la_S=1$, \cref{thm:rank-zero-hilbert-model} and the remark following
it give a birational map
$$
M^\circ_{H',S}(\mw)
\dashrightarrow
\Pic^0(S)\times\Hilb^n(S).
$$

The varieties on both sides are projective $\qq$-factorial terminal minimal
models with torsion canonical divisor. Hence the birational map is an
isomorphism in codimension one by \cite[Theorem~1]{Kaw08}. Common resolutions
identify the Picard varieties, and hence the Albanese varieties, of all the
birational models involved. After translations of the targets, their
Albanese morphisms agree on the common domains of the birational maps. For a
general point of the Albanese variety, the complements of these common
domains meet the corresponding fibers in codimension at least two. Thus the
induced birational maps between the general Albanese fibers are isomorphisms
in codimension one.

The Albanese fiber of
$$
\Pic^0(S)\times\Hilb^n(S)
$$
is isomorphic to an Albanese fiber $K_n(S)$ of $\Hilb^n(S)$. Since all the
Albanese morphisms under consideration are locally constant analytic
fibrations, every fiber is isomorphic to a general fiber. Combining this with
the codimension-one comparison in the proof of
\cref{thm:betti-numbers-admissible} gives a birational map
$$
F\dashrightarrow K_n(S)
$$
which is an isomorphism in codimension one.

By \cite[Theorem~3.5 and its proof]{OS11}, the fundamental group of $K_n(S)$
is finite and its universal cover
$$
\tau:Z_n\longrightarrow K_n(S)
$$
is a smooth simply connected Calabi--Yau manifold of dimension $2n-1$.
Choose open subsets
$$
U\subset F,
\qquad
V\subset K_n(S)
$$
whose complements have codimension at least two and for which $U\simeq V$.
Since $K_n(S)$ is smooth, removing a subset of complex codimension at least
two does not change its fundamental group. Therefore
$$
\pi_1(V)\simeq\pi_1(K_n(S)),
$$
and $\tau^{-1}(V)$ is connected.

Via the identification $U\simeq V$, we identify the function fields of $F$
and $K_n(S)$. Let $\widetilde F$ be the normalization of $F$ in the resulting
finite extension $\cc(Z_n)/\cc(F)$. This gives a finite morphism
$$
\gamma:\widetilde F\longrightarrow F
$$
whose restriction over $U$ is identified with
$\tau^{-1}(V)\longrightarrow V$. It is therefore \'etale over $U$, and hence
quasi-\'etale. Because finite
morphisms preserve codimension, the complement of $\tau^{-1}(V)$ has
codimension at least two in both $Z_n$ and $\widetilde F$. Thus these
varieties contain an isomorphic common big open subset.

Since $\gamma$ is quasi-\'etale and $F$ is terminal, the variety
$\widetilde F$ is klt by \cite[Proposition~5.20]{KollarMori}. Its reflexive
canonical sheaf agrees with $\omega_{Z_n}$ on the common big open subset.
Since
$
\omega_{Z_n}\simeq\mathcal O_{Z_n},
$
and rank-one reflexive sheaves are determined in codimension one, one has
$$
\omega_{\widetilde F}\simeq\mathcal O_{\widetilde F}.
$$
Thus $K_{\widetilde F}$ is Cartier. A klt variety with Cartier canonical
divisor has canonical singularities. Hence $\widetilde F$ has canonical
singularities and trivial canonical sheaf. Since reflexive forms are
determined in codimension one, for every $p$ one has
$$
H^0\left(\widetilde F,\Omega_{\widetilde F}^{[p]}\right)
\simeq
H^0\left(Z_n,\Omega_{Z_n}^p\right).
$$
Consequently,
$
H^0\left(\widetilde F,\Omega_{\widetilde F}^{[p]}\right)=0
$
for $0<p<\dim\widetilde F$.

It remains to verify this condition after every finite quasi-\'etale cover.
Let $T\to\widetilde F$ be a connected finite quasi-\'etale cover. A connected
normal variety remains connected after removing a subset of codimension at
least two. Thus, on the common big open subset with $Z_n$, after removing a
further subset of codimension at least two, the morphism induces a connected
finite \'etale cover. Normalizing $Z_n$ in $\cc(T)$ gives a finite morphism
which is \'etale in codimension one. Since $Z_n$ is smooth, purity of the
branch locus \cite[Tag~0BMB]{Stacks} shows that this morphism is \'etale
everywhere. Since $Z_n$ is simply connected, the cover is trivial. Hence
$\cc(T)=\cc(\widetilde F)$, and the finite birational morphism
$T\to\widetilde F$ is an isomorphism.

Thus $\widetilde F$ is a Calabi--Yau variety in the sense of
\cite[Definition~1.3]{GGK19}. By \cite[Corollaries~4.8 and~7.4 and
Proposition~12.10]{GGK19}, the tangent sheaf $\mathcal T_{\widetilde F}$ is
stable with respect to every ample polarization.

Let $f:F'\to F$ be a connected finite quasi-\'etale cover, and let $W$ be a
connected component of the normalization of
$$
\widetilde F\times_FF'.
$$
Then $W\to\widetilde F$ is a connected finite quasi-\'etale cover. By the
preceding paragraph, $\widetilde F$ has no nontrivial connected finite
quasi-\'etale covers, so $W\simeq\widetilde F$. In particular, there is a
finite quasi-\'etale morphism
$$
\gamma':\widetilde F\longrightarrow F'.
$$

Suppose that $\mathcal T_{F'}$ were unstable with respect to an ample
polarization $A$ on $F'$. Pulling a saturated destabilizing subsheaf back to
$\widetilde F$ and taking its reflexive hull would destabilize
$$
\mathcal T_{\widetilde F}
\simeq
\left((\gamma')^*\mathcal T_{F'}\right)^{**}
$$
with respect to $(\gamma')^*A$, because slopes are multiplied by the degree
of the finite morphism. This is impossible. Hence $\mathcal T_{F'}$ is stable
for every connected finite quasi-\'etale cover $F'\to F$ and every ample
polarization. Thus $\mathcal T_F$ is strongly stable.
\end{proof}

\appendix
\newpage
\section{Cohomological tools}\label[appendix]{app:cohomological-tools}
\thispagestyle{plain}
\begin{lemma}\label{lem:A1}
\cite[Lemma 4.10]{Sac19} Let $q \colon \uc \longrightarrow B$ be a family of smooth curves, and let $h \colon \bar{J}^d_{\uc} \longrightarrow B$ be the degree $d$ relative compactified Jacobian. For every $i$, there is a natural morphism
$$
R^ih_*\qq_{\bar{J}^d_{\uc}}
\longrightarrow
R^iq_*\qq_{\uc},
$$
which is an isomorphism for $i = 1$. Moreover, the same holds for the higher direct images of the structure sheaves.
\end{lemma}

\begin{lemma}\label{lem:low-degree-purity-one-boundary}
Let $X$ be a projective complex variety with terminal l.c.i. singularities.
Let $U\subset X$ be a Zariski open subset such that $X\setminus U=D\cup Z$,
where $D$ is an irreducible $\qq$-Cartier divisor and
$\operatorname{codim}_X Z\geqslant 2$.
Then
$$
H^1_{X\setminus U}(X,\cc)=0,
\qquad
H^2_{X\setminus U}(X,\cc)\simeq\cc(-1),
$$
and the natural map
$$
H^2_{X\setminus U}(X,\cc)\longrightarrow H^2(X,\cc)
$$
is injective. Consequently, restriction induces an isomorphism
$$
H^1(X,\cc)\simeq H^1(U,\cc),
$$
and the mixed-Hodge localization sequence gives an exact sequence
$$
0\longrightarrow \cc(-1)
\longrightarrow H^2(X,\cc)
\longrightarrow W_2H^2(U,\cc)
\longrightarrow 0.
$$
\end{lemma}
\begin{proof}
Let $Y\coloneqq X\setminus U=D\cup Z$.
We first isolate the contribution of the divisorial part of the boundary.
Since \(X\) has terminal singularities, it is smooth in codimension two; equivalently,
$$
\operatorname{codim}_X\operatorname{Sing}(X)\geqslant 3
$$
by \cite[Corollary 5.18]{KollarMori}. Moreover, the sheaf $\cc_X[\dim X]$ is perverse because \(X\) is l.c.i.; see \cite[Theorem 5.1.20]{DimcaSheaves}. By the support-vanishing property for perverse sheaves, if \(T\subset X\) is a closed subset of codimension at least \(c\), then $\cc_X[\dim X]$ has no cohomology with supports on \(T\) in the corresponding low degrees. In the present situation, since
$$
\operatorname{codim}_X\operatorname{Sing}(X)\geqslant 3,
$$
this gives
$$
H^i_{\operatorname{Sing}(X)}(X,\cc)=0
\qquad\text{for }i<3.
$$
In particular,
$$
H^1_{\operatorname{Sing}(X)}(X,\cc)=
H^2_{\operatorname{Sing}(X)}(X,\cc)=0.
$$
Set $X^\circ\coloneqq X_{\mathrm{reg}}$.
On $X^\circ$, Thom--Gysin purity for the divisorial part of the boundary
and purity for subsets of codimension at least two give
$$
H^1_{Y\cap X^\circ}(X^\circ,\cc)=0.
$$
The restriction sequence for cohomology with supports therefore gives
$$
H^1_Y(X,\cc)=0.
$$
The restriction sequence for supports gives an injection
$$
H^2_Y(X,\cc)\hookrightarrow H^2_{Y\cap X^\circ}(X^\circ,\cc).
$$

We now compute the target. The variety $X^\circ$ is smooth. Put
$D^\circ\coloneqq D\cap X^\circ$. Since $D$ is irreducible, there exists a
dense smooth open subset $D_{\mathrm{sm}}\subset D^\circ$. Let
$T\subset X^\circ$ be the union of $Z\cap X^\circ$ and
$D^\circ\setminus D_{\mathrm{sm}}$. Then $T$ has codimension at least two in
$X^\circ$. By excision and the usual purity theorem for cohomology with
supports on a smooth variety, closed subsets of codimension at least two do
not contribute to $H^2$ with supports. Therefore
$$
H^2_{Y\cap X^\circ}(X^\circ,\cc)
\simeq
H^2_{D_{\mathrm{sm}}}(X^\circ\setminus T,\cc).
$$
By the Thom--Gysin purity isomorphism for a smooth divisor in a smooth complex variety,
$$
H^2_{D_{\mathrm{sm}}}(X^\circ\setminus T,\cc)
\simeq
H^0(D_{\mathrm{sm}},\cc)(-1).
$$
Since $D$ is irreducible, the open subset $D_{\mathrm{sm}}$ is connected. Hence
$$
H^0(D_{\mathrm{sm}},\cc)\simeq \cc,
$$
and therefore
$$
H^2_{Y\cap X^\circ}(X^\circ,\cc)\simeq \cc(-1).
$$
Since \(H^2_Y(X,\cc)\) injects into this one-dimensional vector space, we have
$$
\dim H^2_Y(X,\cc)\leqslant 1.
$$

On the other hand, choose $r\geqslant 1$ such that $rD$ is Cartier. Its
cohomology class
$$
[D]\coloneqq \frac{1}{r}c_1(\mathcal O_X(rD))\in H^2(X,\cc)
$$
is nonzero. Indeed, if $[D]=0$, then for an ample divisor $H$ on $X$ one
would have
$$
D\cdot H^{\dim X-1}=0,
$$
contradicting the positivity of the degree of the nonzero effective Cartier
divisor $rD$ against $H^{\dim X-1}$. Since the canonical section of
$\mathcal O_X(rD)$ is nowhere vanishing on
$$
U=X\setminus Y,
$$
the class $[D]$ lies in the kernel of the restriction map
$$
H^2(X,\cc)\longrightarrow H^2(U,\cc).
$$
By exactness of the localization sequence
$$
H^2_Y(X,\cc)\longrightarrow H^2(X,\cc)\longrightarrow H^2(U,\cc),
$$
this nonzero class comes from $H^2_Y(X,\cc)$. Hence
$H^2_Y(X,\cc)\neq 0$. Together with the inequality above, this gives
$$
H^2_Y(X,\cc)\simeq \cc(-1),
$$
and the image of
$$
H^2_Y(X,\cc)\longrightarrow H^2(X,\cc)
$$
is the one-dimensional subspace generated by \([D]\).

Thus the localization sequence gives an exact sequence
$$
0\longrightarrow \cc(-1)
\longrightarrow H^2(X,\cc)
\longrightarrow \operatorname{Im}\bigl(H^2(X,\cc)\to H^2(U,\cc)\bigr)
\longrightarrow 0.
$$
Finally, functoriality and the weight bounds for proper varieties imply that
$$
\operatorname{Im}\bigl(H^2(X,\cc)\longrightarrow H^2(U,\cc)\bigr)
=
W_2H^2(U,\cc);
$$
see \cite[Corollaire~3.2.17]{Del71} and
\cite[Th\'eor\`eme~8.2.4 and Proposition~8.2.5]{Del74}. Therefore
$$
0\longrightarrow \cc(-1)
\longrightarrow H^2(X,\cc)
\longrightarrow W_2H^2(U,\cc)
\longrightarrow 0.
$$
All maps in the displayed sequences are morphisms of mixed Hodge structures,
so the final sequence is exact in the category of mixed Hodge structures.
\end{proof}

\begin{lemma}\label{lem:connected-torus-kernel}
Let
$
\varphi\colon A\longrightarrow B
$
be a surjective homomorphism of complex tori. Then there is a natural isomorphism
$$
\pi_0(\ker\varphi)
\simeq
\operatorname{coker}\left(H_1(A,\zz)\longrightarrow H_1(B,\zz)\right).
$$
In particular, $\ker\varphi$ is connected if and only if the induced map on integral first homology is surjective.
\end{lemma}

\begin{proof}
Write $A=V_A/\Lambda_A$ and $B=V_B/\Lambda_B$, where
$
\Lambda_A=H_1(A,\zz),
\qquad
\Lambda_B=H_1(B,\zz),
$
and let $\widetilde\varphi\colon V_A\to V_B$ be the complex-linear lift of $\varphi$. Since $\varphi$ is surjective, so is $\widetilde\varphi$. Moreover,
$
\ker\varphi=\widetilde\varphi^{-1}(\Lambda_B)/\Lambda_A,
$
while its identity component is
$
(\ker\widetilde\varphi)/(\ker\widetilde\varphi\cap\Lambda_A).
$
The homomorphism induced by $\widetilde\varphi$ identifies the quotient of $\ker\varphi$ by its identity component with
$
\Lambda_B/\widetilde\varphi(\Lambda_A).
$
Since $\widetilde\varphi|_{\Lambda_A}$ is the map induced by $\varphi$ on integral first homology, the assertion follows.
\end{proof}

\begin{lemma}\label{lem:h2-isomorphism-codimension-one}
Let $X$ and $Y$ be irreducible complex algebraic varieties which are local
complete intersections. Assume that
$
\operatorname{codim}_X\operatorname{Sing}(X)\geqslant 4$
and
$
\operatorname{codim}_Y\operatorname{Sing}(Y)\geqslant 4.
$
Suppose that there is a birational map
$
\varphi:X\dashrightarrow Y
$
which is an isomorphism in codimension one. Then $\varphi$ induces an
isomorphism of mixed Hodge structures
$$
H^2(X,\cc)\simeq H^2(Y,\cc).
$$
In particular, $b_2(X)=b_2(Y)$.
\end{lemma}

\begin{proof}
Let $X^{\mathrm{reg}}$ and $Y^{\mathrm{reg}}$ denote the regular loci. The
support-vanishing argument used in the proof of
\cref{lem:low-degree-purity-one-boundary} gives
$$
H^2(X,\cc)\xrightarrow{\ \sim\ }
H^2(X^{\mathrm{reg}},\cc),
\qquad
H^2(Y,\cc)\xrightarrow{\ \sim\ }
H^2(Y^{\mathrm{reg}},\cc),
$$
because the singular loci have codimension at least four.

Because $\varphi$ is an isomorphism in codimension one, there are open
subsets $U_X\subset X$ and $U_Y\subset Y$ whose complements have codimension
at least two and for which $\varphi$ restricts to an isomorphism
$U_X\simeq U_Y$. Intersecting with the regular loci and using the fact that an
isomorphism preserves regularity, we obtain a common smooth open subset $U$
such that
$$
\operatorname{codim}_{X^{\mathrm{reg}}}
(X^{\mathrm{reg}}\setminus U)\geqslant 2,
\qquad
\operatorname{codim}_{Y^{\mathrm{reg}}}
(Y^{\mathrm{reg}}\setminus U)\geqslant 2.
$$

Purity for cohomology with supports on a smooth variety gives
$$
H^i_{X^{\mathrm{reg}}\setminus U}(X^{\mathrm{reg}},\cc)=0
\qquad\text{for }i<4.
$$
Hence restriction induces
$$
H^2(X^{\mathrm{reg}},\cc)
\xrightarrow{\ \sim\ }H^2(U,\cc).
$$
The same argument gives
$$
H^2(Y^{\mathrm{reg}},\cc)
\xrightarrow{\ \sim\ }H^2(U,\cc).
$$
Combining these isomorphisms yields
$$
H^2(X,\cc)\simeq H^2(Y,\cc),
$$
as required. All the restriction morphisms used above are morphisms of mixed
Hodge structures. Since they are isomorphisms on the underlying vector
spaces, they are isomorphisms of mixed Hodge structures.
\end{proof}

\section{Betti numbers of \texorpdfstring{$\Pic^0(S)\times\Hilb^n(S)$}{Pic0(S) x Hilb n(S)}}\label[appendix]{app:betti-hilbert}

Let \(S\) be a bielliptic surface. Then $b_1(S)=2$ and $b_2(S)=2$, and \(\Pic^0(S)\) is an elliptic curve. Hence $(b_0,b_1,b_2)(\Pic^0(S))=(1,2,1)$.

\begin{theorem}\label{thm:betti-pic-hilb}
Let $X_n\coloneqq\Pic^0(S)\times\Hilb^n(S)$.
Then, for every \(n\geqslant 0\),
$$
(b_0,b_1,b_2)(X_n)=
\begin{cases}
(1,2,1),& n=0,\\
(1,4,7),& n=1,\\
(1,4,9),& n\geqslant 2.
\end{cases}
$$
\end{theorem}

\begin{proof}
By G\"ottsche's formula,
$$
\sum_{n\geqslant 0}P_t(\Hilb^n(S))q^n
=
\prod_{m\geqslant 1}\prod_{j=0}^4
\left(1-(-t)^{2m-2+j}q^m\right)^{(-1)^{j+1}b_j(S)}.
$$
Expanding to order \(t^2\), using \(b_1(S)=2\) and \(b_2(S)=2\), gives
$$
(b_0,b_1,b_2)(\Hilb^n(S))=
\begin{cases}
(1,0,0),& n=0,\\
(1,2,2),& n=1,\\
(1,2,4),& n\geqslant 2.
\end{cases}
$$
Since \(\Pic^0(S)\) is an elliptic curve, $(b_0,b_1,b_2)(\Pic^0(S))=(1,2,1)$.
The K\"unneth formula gives $b_1(X_n)=b_1(\Pic^0(S))+b_1(\Hilb^n(S))$ and
$$
b_2(X_n)
=
b_2(\Pic^0(S))
+b_1(\Pic^0(S))b_1(\Hilb^n(S))
+b_2(\Hilb^n(S)).
$$
Substituting the values above gives the three cases.
\end{proof}

\phantomsection

\par\bigskip
\begingroup
\footnotesize

\noindent\textsc{Faculty of Mathematics, Technion - Israel Institute of Technology, Haifa, Israel}\par

\noindent
\textit{E-mail address: arenaki.csr@gmail.com}
\href{mailto:arenaki.csr@gmail.com}{\nolinkurl{}}\par

\noindent
\textit{URL:}
\url{https://sites.google.com/view/aleksei-piskunov}\par

\endgroup

\end{document}